\documentclass[10.8pt]{amsart}   	
\usepackage{geometry}                		

\usepackage{amsmath}    
\usepackage{graphicx} 
\usepackage{verbatim} 
\usepackage{subfigure}  
\usepackage{bbm}
\usepackage{enumerate}
\usepackage[super]{nth}

\usepackage[backref]{hyperref}
\hypersetup{
  colorlinks   = true,          
  urlcolor     = blue,          
  linkcolor    = purple,          
  citecolor   = blue             
}
			
\usepackage{amssymb,mathrsfs}
\usepackage{amssymb,amsthm,amsmath,stmaryrd}
\usepackage{tikz}
\usepackage{tikz-cd}
\usepackage{color}
\usepackage{eucal,accents,upgreek,enumerate}
\usepackage[headings]{fullpage}

\usepackage{pdfpages}
\usepackage[font={small,it}]{caption}

\usetikzlibrary{calc} 
\usetikzlibrary{decorations.pathreplacing}

\usepackage{wrapfig}
\usepackage{float}
\usepackage{multicol}

\newtheorem*{theorem*}{Main Theorem}
\newtheorem*{lemma*}{Lemma}

\newtheorem{theorem}{Theorem}[subsection]
\newtheorem{proposition}{Proposition}[subsection]
\newtheorem{lemma}[theorem]{Lemma}

\newtheorem{corollary}{Corollary}[subsection]

\newtheorem{conj}{Conjecture}

\newtheorem{question}{Question}

\newtheorem{remark}{Remark}[subsection]

\newcommand{\supp}{\mathop{\rm supp}}

\newcommand{\rk}{\mathop{\rm rk}}
\newcommand{\val}{\mathop{\rm val}}
\newcommand{\Div}{\mathop{\rm Div}}

\makeatletter
\let\@@pmod\pmod
\DeclareRobustCommand{\pmod}{\@ifstar\@pmods\@@pmod}
\def\@pmods#1{\mkern4mu({\operator@font mod}\mkern 6mu#1)}
\makeatother

\begin{document}
\title[Brill--Noether existence for graphs]{\raisebox{0.2in}{A note on Brill--Noether existence for graphs of low genus}}

\author{Stanislav Atanasov \& Dhruv Ranganathan}

\address{Department of Mathematics, Yale University, New Haven, CT 06511}
\email{stanislav.atanasov@yale.edu}

\address{Department of Mathematics, Massachussets Institute of Technology, Cambridge, MA 02138}
\email{dhruvr@mit.edu}



\begin{abstract}
In an influential 2008 paper, Baker proposed a number of conjectures relating the Brill--Noether theory of algebraic curves with a divisor theory on finite graphs. In this note, we examine Baker's Brill--Noether existence conjecture for special divisors. For $g\leq 5$ and $\rho(g,r,d)$ non-negative, every graph of genus $g$ is shown to admit a divisor of rank $r$ and degree at most $d$.  As further evidence, the conjecture is shown to hold in rank $1$ for a number families of highly connected combinatorial types of graphs. In the relevant genera, our arguments give the first combinatorial proof of Brill--Noether existence theorem for metric graphs, giving a partial answer to a related question of Baker. 
\end{abstract}

\tikzset{me/.style={to path={
\pgfextra{%
 \pgfmathsetmacro{\startf}{-(#1-1)/2}  
 \pgfmathsetmacro{\endf}{-\startf} 
 \pgfmathsetmacro{\stepf}{\startf+1}}
 \ifnum 1=#1 -- (\tikztotarget)  \else
     let \p{mid}=($(\tikztostart)!0.5!(\tikztotarget)$) 
         in
\foreach \i in {\startf,\stepf,...,\endf}
    {%
     (\tikztostart) .. controls ($ (\p{mid})!\i*6pt!90:(\tikztotarget) $) .. (\tikztotarget)
      }
      \fi   
     \tikztonodes
}}}   
\maketitle

\section{Introduction}

 \subsection{Statement of main results}   The last decade has seen a number of results exploring the interplay between the divisor theory of algebraic curves and an analogous theory on graphs, developed by Baker and Norine~\cite{BakerNorine}. These theories are interlinked by Baker's Specialization Lemma~\cite[Lemma~2]{Baker}, which states that the rank of a divisor on an algebraic curve over a valued field can only increase upon specialization to a skeleton. This has led to numerous applications in algebraic geometry and number theory; see the survey~\cite{BakerJensen}. For example, combinatorial Brill--Noether theory has been successfully employed to produce tropical proofs of the Brill--Noether and the Gieseker--Petri Theorems in algebraic geometry, and has provided insights on the maximal rank conjecture~\cite{CDPR,JensenPayneI,JensenPayneII}. Divisors on graphs are also of purely combinatorial interest, for instance, through connections as diverse as $G$-parking functions~\cite{PostnikovShapiro} and cryptosystems~\cite{Shokrieh}. 
    
In his paper on the specialization lemma~\cite{Baker}, Baker conjectured a number of combinatorial results concerning the divisor theory of graphs based on theorems in algebraic geometry. Many of these conjectures have now been proved~\cite{CDPR,HladkyKralNorine} and have been the basis for substantial additional progress. In this paper we study one of the remaining open questions, the combinatorial counterpart to the existence part of the Brill--Noether theorem\footnote{Experts have recorded a gap in the proof of this conjecture that appears in~\cite[Theorem 6.3]{Caporaso}. We direct the reader to the discussion in~\cite[Remark 4.8 and Footnote 5]{BakerJensen}.}. 

Recall that for nonnegative integers $g,r$, and $d$, the Brill--Noether number is defined to be $\rho(g,r,d) = g-(r+1)(g-d+r)$. 

\begin{conj}{\emph{(Brill--Noether existence conjecture for graphs)}}\label{conj:b.n.existence} If $\rho(g,r,d)$ is nonnegative, then every graph of genus $g$ admits a divisor $D$ with $\rk(D)=r$ and $\deg (D) \leq d.$
\end{conj}

A number of researchers have demonstrated an intricate Brill--Noether theory entirely within the realm of (finite or metric) graphs, see for instance~\cite{BN09,CapGon,Len14,LPP}. The conjecture above is a central question in this area. In this paper, we confirm Baker's conjecture in genera up to $5$. 

\begin{theorem*}
\label{thm:B.N.-for-finite-graphs}
The Brill--Noether existence conjecture holds for all finite graphs of genus at most~$5$.
\end{theorem*}

The specialization lemma immediately implies the Brill--Noether existence conjecture for all \textit{metric} graphs, where chips may need to be placed in the interiors of edges. Baker asks the following question.

\begin{question}
Can the Brill--Noether existence theorem for metric graphs be proved using purely combinatorial methods?
\end{question}

The proof of the main theorem, with superficial changes, furnishes such a proof for all metric graphs of genus at most $5$. 

In a complementary direction, one could ask for an algebro-geometric proof of Brill--Noether existence for finite graphs. This question is closely related to the existence of divisors on curves over discretely valued fields that are expressible as sums of rational points, as well as bounds on degrees of ramified base changes in semistable reduction. We are not aware of any substantial progress in this direction. 

As further evidence for the conjecture, we exhibit a highly connected homeomorphism classes of graphs in increasing genus, for which the existence conjecture holds in rank $1$ for all representatives of that class. These results are stated precisely in Section~\ref{sec:graphs-of-high-genus}.

\subsection{Context from algebraic geometry} The fact that, when $\rho\geq 0$, every algebraic curve admits a divisor of rank $r$ and degree at most $d$ was proved by Kempf, Kleiman, and Laksov~\cite{Kempf,KleimanLaksov}. It is considered to be the easier part of the Brill--Noether theorem. The harder direction, showing the nonexistence of special divisors when $\rho$ is negative, was proved by Griffiths and Harris~\cite{GriffithsHarris}. Kempf, Kleiman, and Laksov's proof of the existence of special divisors follows from Schubert calculus techniques and the Thom--Porteous determinantal formula. However, such techniques are not available in the discrete setting. On the other hand, the harder direction, the existence of Brill--Noether general graphs in every genus was proved purely combinatorially by Cools, Draisma, Payne, and Robeva and implies the harder direction of the Brill--Noether theorem~\cite{CDPR}.

\subsection{Related results}
In \cite{Baker}, Baker shows that any finite graph $G$ can be uniformly rescaled to a graph $G'$ for which Conjecture~\ref{conj:b.n.existence} holds. More precisely, there exists an integer $m_G$ such the inflated graph~$G'$ obtained by putting $m_G - 1$ bivalent vertices on each edge of $G$ satisfies Conjecture~\ref{conj:b.n.existence}. Conjecture~\ref{conj:b.n.existence} then asserts that we can always pick $m_G=1$ for every finite graph $G$. No effective bounds on the value of $m_G$ are known. Conjecture~\ref{conj:b.n.existence} remains open for $r=1$ and $g\geq 6$, where it is equivalent to the following.

\begin{conj}{\emph{(Gonality conjecture)}}\label{conj:gonality-conj} The gonality of any graph of genus $g$ is at most $\lfloor (g+3)/2 \rfloor$.
\end{conj}

Recall that the gonality of an algebraic curve is the smallest degree of a rank one divisor. After the results of this paper the next outstanding case of Conjecture~\ref{conj:b.n.existence} is $g=6, {r=1}, {d=4}$. 

For $g\geq 6$ the strongest result concerning the gonality conjecture is a recent result of Cools and Draisma~\cite{CoolsDraisma}. They show that for any topologically trivalent genus g graph $G=(V,E)$ there exists a nonempty open cone $C_G\subseteq \mathbb{R}_{>0}^{|E|}$ whose image in $\mathcal{M}^{\text{trop}}_{g}$ consists entirely of metric graphs with gonality exactly $d:=\lfloor (g+3)/2 \rfloor$. Furthermore, any graph corresponding to a lattice point of $C_G$ satisfies the existence conjecture. Their approach relies on studying harmonic morphisms to trees and the techniques of~\cite{ABBR}. We are not aware of any systematic results in higher genus for which the cone $C_G$ is known to be the entire orthant. 


\subsection{Outline of the Paper.}
In Section~\ref{sec:Preliminaries} we briefly recall the Baker-Norine theory of divisors on finite graphs, reduced divisors, and Dhar's burning algorithm. In Section~\ref{sec:reduction-to-rank-1} we reduce the existence conjecture to the rank~$1$ case. In Sections~\ref{sec:B.N-for-genus-4} and~\ref{sec:B.N-for-genus-5} we prove the Main Theorem for graphs of genus $4$ and $5$, respectively. In both sections we produce divisors of prescribed degree and rank for topologically trivalent and then degenerate the construction for general graphs. In Section~\ref{sec:graphs-of-high-genus} we exhibit families of graphs of increasing genus for which the existence conjecture in rank $1$ holds. 

\subsection*{Acknowledgements}
We acknowledge helpful conversations with Matt Baker, Derek Boyer, Dave Jensen, Andr\'e Moura, Sam Payne, and Scott Weady during the course of this work and Kalina Petrova, for assistance with computational aspects of the project. This project was part of the Summer Undergraduate Mathematics Research at Yale Program (S.U.M.R.Y) where S.A. was supported as a student and D.R. was supported as a mentor. We thank Sam Payne, Jos\'e Gonz\'alez, and Michael Magee for organizing the program. This research was supported in part by NSF grant CAREER DMS-1149054 (PI: Sam Payne). Finally, we thank the referee for their meticulous reading and helpful comments.

\section{Divisor theory on finite graphs}
\label{sec:Preliminaries}
The main reference for this section is the original paper of Baker and Norine~\cite{BakerNorine}. A graph $G$ will mean a finite connected graph possibly with loops and multiple edges. The vertex and edge sets of $G$ will be denoted $V(G)$ and $E(G)$ respectively. The \textbf{genus} of $G$, denoted $g(G)$, is defined to be
\[
g(G):=|E(G)|-|V(G)| +1.
\]

A \textbf{divisor} $D$ on a graph $G$ is a formal $\mathbb{Z}$-linear combination on its vertices
\[
D=\sum_{v\in V(G)} D(v)\cdot v.
\]

The \textbf{degree} of a divisor, denoted $\deg(D)$, equals $\sum_{v\in V(G)} D(v)$ and a divisor is said to be \textbf{effective} if $D(v)\geq 0$ for all $v\in V(G)$. The set of all divisors on a graph $G$ will be denoted by $\Div(G)$ and it has a natural grading $\Div(G) = \bigoplus_{d\in \mathbb{Z}} {\Div}^d(G)$ induced by the degree. Same holds for ${\Div}_+(G)$, the set of all effective divisors.

It is often useful to think of the integers $D(v)$ above as the number of \textit{chips} or \textit{antichips} placed on $v\in V(G)$. Given $D$ and a vertex $v$, we may obtain a new divisor by means of a \textbf{chip firing move} as follows. The vertex $v$ sends one chip to its neighbors, along each of the outgoing edges connecting them. Thus, $D(v)$ decreases by the valence of $v$, and for each $w$ a neighbor of $v$, $D(w)$ increases by the number of edges between $v$ and $w$. Chip-firing generates an equivalence relation on the set $\mathrm{Div}(G)$ of divisors on $G$ known as \textbf{linear equivalence}. The class $[D]$ is said to be \textbf{effective} if it contains an effective representative. For an alternative definition of this equivalence in terms of piecewise linear functions, see~\cite{Caporaso}.

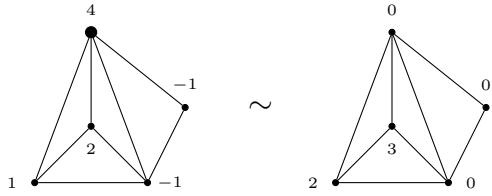
\begin{figure}[H]
\begin{tikzpicture}
\node[circle,fill=black,inner sep=0.8pt,draw] (a) at (0,0) {};
\node[circle,fill=black,inner sep=0.8pt,draw] (b) at (1.5,0) {};
\node[circle,fill=black,inner sep=1.5pt,draw] (c) at (.75,2) {};
\node[circle,fill=black,inner sep=0.8pt,draw] (d) at (0.75,.75) {};
\node[circle,fill=black,inner sep=0.8pt,draw] (e) at (2,1) {};

\node () at (-0.3,0) {\tiny$1$};
\node () at (0.75,0.45) {\tiny$2$};
\node () at (0.75,2.3) {\tiny$4$};
\node () at (2,1.3) {\tiny$-1$};
\node () at (1.8,0) {\tiny$-1$};

\draw (a)--(b)--(e)--(c) --(a)--(d)--(b)--(c);
\draw (c)--(d);

\node at (3,1) {\Large{$\sim$}};

\begin{scope}[shift={(+4,0)}]
\node[circle,fill=black,inner sep=0.8pt,draw] (a) at (0,0) {};
\node[circle,fill=black,inner sep=0.8pt,draw] (b) at (1.5,0) {};
\node[circle,fill=black,inner sep=0.8pt,draw] (c) at (.75,2) {};
\node[circle,fill=black,inner sep=0.8pt,draw] (d) at (0.75,.75) {};
\node[circle,fill=black,inner sep=0.8pt,draw] (e) at (2,1) {};

\node () at (-0.3,0) {\tiny$2$};
\node () at (0.75,0.45) {\tiny$3$};
\node () at (0.75,2.3) {\tiny$0$};
\node () at (2,1.3) {\tiny$0$};
\node () at (1.8,0) {\tiny$0$};

\draw (a)--(b)--(e)--(c) --(a)--(d)--(b)--(c);
\draw (c)--(d);

\end{scope}
\end{tikzpicture}
\caption{\small{The larger vertex fires once to move from the left configuration to the right configuration.}}
\end{figure}

It is a well-known fact that every two same degree divisors on a tree are equivalent. For that reason, when studying divisors on graphs, no information is lost by contracting all grafted trees and assuming that all vertices have valency at least two. 

The central invariant in the divisor theory of graphs is the rank of a divisor. If $[D]$ is effective, the \textbf{rank} of $D$ is defined as
\[\rk(D):=\max \{k\in \mathbb{Z}_{\geq 0}~ | ~[D - E]~\text{~is effective},~\forall E\in {\Div}^k_{+}(G) \}.
\]

If $[D]$ is not effective, then we set $\rk(D) = -1$. Motivated by the classical result in the theory of algebraic curves, Baker and Norine~\cite{BakerNorine} exhibited a Riemann-Roch theorem for graphs.

\begin{theorem}{\emph{(Riemann-Roch for graphs)}}
\label{thm:Riemann-Roch}
Let $D$ be a divisor on $G$. Then
\[
\rk (D) - \rk(K_G-D) = \deg (D) - g +1,
\]
where $K_G = \sum_{v\in V(G)} (\val (v) - 2)(v)$.
\end{theorem}

\subsection{Reduced divisors and Dhar's burning algorithm}
Given a divisor $D$ on $G$ and a vertex~$v_0$, we say that $D$ is $v_0$-\textbf{reduced} if 
\begin{enumerate}
\item $D(v)\geq 0$ for all $v\neq v_0$, and
\item every non-empty set $A\subseteq V(G)\backslash \{v_0\}$ contains a vertex $v$ such that {$\mathop{\rm{outdeg}}_A (v) > D(v)$}.
\end{enumerate}

Here $\mathop{\rm{outdeg}}_A (v)$ denotes the outdegree of $v$ with respect to $A$, i.e., the number
of edges connecting $v$ to a vertex not in $A$. Every divisor is equivalent to a unique $v_0$-reduced divisor. Moreover, a divisor class is effective if and only its reduced form is effective. As a result, reduced divisors are central to calculating ranks of divisors. There is an efficient computational procedure to yield a reduced divisor known as \textit{Dhar{'}s burning algorithm}. 

Suppose that $D$ is such that $D(v)\geq 0$ for all $v\neq v_0$. At each vertex $v\neq v_0$ place $D(v)$ ``firefighters.'' Each firefighter is capable of controlling precisely one fire. Start a fire at $v_0$. The fire spreads through the graph, so that an edge burns if one of its endpoints burns. A vertex burns if the number of burning edges incident to it exceeds the number of firefighters placed on it. If the entire graph burns, then $D$ is $v_0$-reduced. If not, we chip fire all the unburnt vertices and repeat the procedure on the newly obtained divisor. The algorithm terminates at the $v_0$-reduced representative. For a detailed description, see \cite[Section~5.1]{BakerShokrieh} and \cite{Dhar}.


\section{Reduction to rank $1$}
\label{sec:reduction-to-rank-1}
In this short section we show that for genera up to $5$ proving the Brill--Noether conjecture reduces to establishing the validity of the gonality conjecture. We take advantage of the relatively high rank of the canonical divisor for graphs of small genus.

Let $G$ be a genus g graph and $D\in \Div(G)$. Since $\rk (D) \geq -1$, Riemann-Roch theorem implies that $\rk(D)\geq \deg(D)-g$. This inequality and the following result are sufficient to prove Conjecture~\ref{conj:b.n.existence} for $g\leq 3$. 

\begin{lemma}{\cite[Lemma~2.7]{Baker}}
\label{lem:removing-one-reducing-rank-by-one}
Let $G$ be a graph and $D\in \Div(G)$. If $\rk(D)\geq 0$, then $\rk(D-v)=\rk(D)-1$ for some $v\in V(G)$.
\end{lemma}

The same argument may be applied to reduce the Brill--Noether existence conjecture to rank $1$ in the genera of interest.

\begin{proposition}
\label{prop:gonality-implies-b.n.existence-for-genus-4}
Fix $g\geq 0$ and suppose $r \geq \lfloor g/2 \rfloor$. If $d\geq 0$ is such that $\rho (g,r,d) \geq 0$, then every graph $G$ of genus $g$ has a divisor $D$ with $\deg(D) \leq d$ and $\rk(D)=r$.
\end{proposition}

\begin{corollary}
\label{cor:B.N-reduces-to-Gonality-for-g=4,5}
Let $G$ be a graph of genus $4$ or $5$. Then Brill--Noether existence conjecture holds for $G$ if and only if the gonality conjecture does, i.e if every $G$ of genus 4 (resp. 5) admits a degree~$3$ divisor (resp. $4$) of rank at least $1$.
\end{corollary}
\section{Brill--Noether existence for graphs of genus 4}
\label{sec:B.N-for-genus-4}


\noindent
\textbf{Notation.} \textit{In the rest of this paper we will use a large number of figures. To support the exposition, divisors on graphs will be depicted by placing chips on vertices that are larger in size.}

\subsection{Auxiliary results}
Let $G$ be a connected graph. A vertex $v\in V(G)$ is said to be \textbf{topological} if it has valency at least three. A path between two topological vertices consisting solely of bivalent edges will be considered a \textbf{topological edge}. It may be visualized as the edges of the finite graph obtained erasing all the bivalent edges. Graphs $G$ and $G'$ are \textbf{homeomorphic} if $G'$ is an inflation (resp. deflation) of $G$ obtained by placing (resp. removing) bivalent vertices on the edges of $G$. We reserve the Greek letter~$\varepsilon$ to denote topological edges. The \textbf{length} of a topological edge $\varepsilon$ is equal to one more than the number of bivalent vertices on $\varepsilon$. It is equal to the length of the path $\varepsilon$ in a geometric realization of $G$ where all edge lengths are $1$. \\

We will require the following elementary lemma.

\begin{lemma}
\label{lem:all-genus-4-at-most-6-primary}
Every genus $g$ graph has at most $2g-2$ topological vertices. 
\end{lemma}


\noindent
An edge $e$ is called a \textbf{bridge} if its removal increases the number of connected components. A graph has a \textbf{$(g_1,g_2)$-bridge decomposition} if it has a bridge separating two components of genera $g_1$ and $g_2$, respectively.

\begin{lemma}{\emph{(Bridge lemma)}}
\label{lem:bridge-lemma}
Let $g_1$ and $g_2$ be positive integers, at least one among which is even. If the gonality conjecture holds for all graphs of genus $g_1$ and $g_2$, then it is also true for all graphs $G$ with $(g_1,g_2)$-bridge decomposition.
\end{lemma}

\begin{proof}
Let $g_1$ and $g_2$ be as above and consider a genus g graph $G$ with $(g_1,g_2)$-bridge decomposition. Let $e$ be a bridge connecting two connected subgraphs $G_1$ and $G_2$ of genera $g_1$ and $g_2$. Let $u\in V(G_1)$ and $v\in V(G_2)$ be its endpoints. Since the gonality conjecture holds for $G_1$ and $G_2$, there exists $D_i\in \Div(G_i)$ of $\deg (D_i) \leq \lfloor (g_i+3)/2 \rfloor$ and $\rk(D_i)\geq 1$ for $i=1,2$. By the definition of the rank, there exist effective $D'_1 \sim D_1 - (u)$ and $D'_2 \sim D_2 - (v).$ Therefore, $D'_1 + (u) \sim D_1$ and $D'_2 +(v) \sim D_2$, and set $D:= D'_1 + D'_2 + (u)$. By firing all vertices of $G_1$, we see $D \sim D'_1 + D'_2 + (v)$. Pick $w\in V(G)$ and without loss of generality assume $w\in G_1$. Then $D - (w)$ is equivalent to an effective divisor, because $D_1 - (w)$ is and we can fire $G_2\cup {\{e\}}$ in place of $u$. Since $w$ was arbitrary, $\rk (D) \geq 1$. Since $g_1$ and $g_2$ are not simultaneously odd, this concludes the proof. 
\end{proof}

As an immediate consequence, the Brill--Noether existence conjecture holds for all graphs with $(2,2)$-bridge decomposition.

\begin{lemma}{\emph{(Loop lemma)}}
\label{lem:loop-lemma}
The gonality conjecture holds for any genus g graph with at least one topological loop if $g=5$, or at least two topological loops if $g=4$.
\end{lemma}
\begin{proof}
Let $G$ be a graph genus 4 with at least two loops. Suppose that two of its loops, denoted $\varepsilon_1$ and $\varepsilon_2$, are located at vertices $v$ and $w$. If $v=w$, the claim follows by the Bridge lemma, so suppose $v\neq w$. Let $G'$ be the graph obtained by contracting both loops. It is of genus 2. If $G'$ has only two vertices, it must necessarily be a banana graph with two loops attached, which has a divisor of degree 3 and rank at least $1$. Otherwise, $G'$ has another vertex $u$ and then the divisor $D:=2\cdot(v) + 2 \cdot (w) - (u)$ on $G'$ has rank at least $1$ by Riemann-Roch. It is not hard to see that $\rk(D)\geq 1$ when viewed as a divisor on $G$ as well. A similar argument works for the genus $5$ statement.
\end{proof}


\subsection{Topologically trivalent graphs}
\label{sec:topologically-trivalent-graphs}
A graph $G$ is said to be \textbf{topologically trivalent} if all of its vertices have valency $2$ or $3$. Starting from a trivalent graph $G$ we elongate the edges by inserting bivalent vertices on its edges. In this manner, we produce graphs homeomorphic to $G$. The set of of topologically trivalent graphs is in natural bijection with the integral points in the interiors of maximal dimensional cells in the moduli space $\mathcal{M}^{\mathrm{trop}}_g $ of tropical curves of genus $g$.

A topologically trivalent graph has genus $4$ if and only if it has precisely $6$ topologically vertices. Using this characterization, we generate all trivalent graphs of genus 4, shown in Figure~\ref{fig:trivalent simple graphs with loops of genus 4}. These were verified with the help of the database~\cite{Meringer}. 

\begin{figure}[H]
\begin{tikzpicture}

\node[circle,fill=black,inner sep=0.8pt,draw] (11) at (0,0) {};
\node[circle,fill=black,inner sep=0.8pt,draw]  (12) at (1,0) {};
\node[circle,fill=black,inner sep=0.8pt,draw] (13) at (0.5,0.43) {};
\node[circle,fill=black,inner sep=0.8pt,draw] (14) at (1.25,0.43) {};
\node[circle,fill=black,inner sep=0.8pt,draw] (15) at (0.5,-0.43) {};
\node[circle,fill=black,inner sep=0.8pt,draw] (16) at (1.25,-0.43) {};

\draw (1.4,0.43) circle (0.15); 
\draw (1.4,-0.43) circle (0.15);

\draw (11) -- (13) -- (12); \draw (11)--(15)--(12); \draw (11)--(12); \draw (13) -- (14); \draw (15)--(16);

\begin{scope}[shift={(+2.5,-0.5)}]
\node[circle,fill=black,inner sep=0.8pt,draw] (21) at (0,0) {};
\node[circle,fill=black,inner sep=0.8pt,draw] (22) at (1,0) {};
\node[circle,fill=black,inner sep=0.8pt,draw] (23) at (0.5,0.86) {};
\node[circle,fill=black,inner sep=0.8pt,draw] (24) at (0,1) {};
\node[circle,fill=black,inner sep=0.8pt,draw] (25) at (1.25,0.86) {};
\node[circle,fill=black,inner sep=0.8pt,draw] (26) at (1.75,0) {};

\draw (-0.15,1) circle (0.15); 
\draw (1.4,0.86) circle (0.15);
\draw (1.9,0) circle (0.15);

\draw (21) -- (23) -- (22) -- (21); \draw (21)--(24); \draw (23) -- (25); \draw (22)--(26);
\end{scope}

\begin{scope}[shift={(+5.8,0)}]
\node[circle,fill=black,inner sep=0.8pt,draw] (61) at (-0.25,-0.43) {};
\node[circle,fill=black,inner sep=0.8pt,draw] (62) at (-0.25,0.43) {};
\node[circle,fill=black,inner sep=0.8pt,draw] (63) at (0,0) {};
\node[circle,fill=black,inner sep=0.8pt,draw] (64) at (.5,0) {};
\node[circle,fill=black,inner sep=0.8pt,draw] (65) at (.75,-0.43) {};
\node[circle,fill=black,inner sep=0.8pt,draw] (66) at (.75, 0.43) {};

\draw (-0.4,-0.43) circle (0.15);
\draw (-0.40,0.43) circle (0.15);
\draw (.90,-0.43) circle (0.15);
\draw (.90,0.43) circle (0.15);

\draw (61) -- (63) -- (64) -- (65); \draw (62) --(63); \draw (64) -- (66);
\end{scope}

\begin{scope}[shift={(+5.4,-2)}]
\node[circle,fill=black,inner sep=0.8pt,draw] (41) at (0,0) {};
\node[circle,fill=black,inner sep=0.8pt,draw] (42) at (1,0) {};
\node[circle,fill=black,inner sep=0.8pt,draw] (43) at (0.5,0.43) {};
\node[circle,fill=black,inner sep=0.8pt,draw] (44) at (1.5,0) {};
\node[circle,fill=black,inner sep=0.8pt,draw] (45) at (0.5,-0.43) {};
\node[circle,fill=black,inner sep=0.8pt,draw] (46) at (1.5,-0.43) {};

\draw (1.65,-0.43) circle (0.15);

\draw (41) -- (43) -- (42); \draw (41)--(45)--(42); \draw (41)--(42); \draw (43) -- (44); \draw (45)--(44); \draw (44) -- (46);
\end{scope}

\begin{scope}[shift={(+3.5,-2)}]
\node[circle,fill=black,inner sep=0.8pt,draw] (50) at (-1.0,0) {};
\node[circle,fill=black,inner sep=0.8pt,draw] (51) at (-0.5,0.5) {};
\node[circle,fill=black,inner sep=0.8pt,draw] (52) at (0.5,0.5) {};
\node[circle,fill=black,inner sep=0.8pt,draw] (53) at (1.0,0) {};
\node[circle,fill=black,inner sep=0.8pt,draw] (54) at (-0.5,-0.5) {};
\node[circle,fill=black,inner sep=0.8pt,draw] (55) at (0.5,-0.5) {};

\draw (50)--(51)--(52)--(53)--(55)--(54)--(50);\draw (51)--(55);\draw (52)--(54);\draw (50)--(53);
\end{scope}

\begin{scope}[shift={(+0.5,-2.3)}]
\node[circle,fill=black,inner sep=0.8pt,draw] (31) at (0.25,0.43) {};
\node[circle,fill=black,inner sep=0.8pt,draw] (32) at (0,0) {};
\node[circle,fill=black,inner sep=0.8pt,draw] (33) at (0.5,0) {};
\node[circle,fill=black,inner sep=0.8pt,draw] (34) at (0.25,1.0) {};
\node[circle,fill=black,inner sep=0.8pt,draw] (35) at (-.35,-.35) {};
\node[circle,fill=black,inner sep=0.8pt,draw] (36) at (.85,-.35) {};

\draw (31) -- (33) -- (32) -- (31); \draw (31)--(34); \draw (33) -- (36); \draw (32)--(35); \draw (34) -- (35) -- (36) -- (34);
\end{scope}

\end{tikzpicture}

\begin{tikzpicture}
\node[circle,fill=black,inner sep=0.8pt,draw] (a1) at (-.43,-.5) {};
\node[circle,fill=black,inner sep=0.8pt,draw] (b1) at (-.43,.5) {};
\node[circle,fill=black,inner sep=0.8pt,draw] (c1) at (0,0) {};
\node[circle,fill=black,inner sep=0.8pt,draw] (d1) at (.5,0) {};
\node[circle,fill=black,inner sep=0.8pt,draw] (e1) at (.93,-0.5) {};
\node[circle,fill=black,inner sep=0.8pt,draw] (f1) at (.93,0.5) {};


\draw (a1) edge[me=2] (b1); 
\draw (c1) edge (b1);
\draw (a1) edge (c1);
\draw (c1) edge (d1);
\draw (d1) edge (e1);
\draw (d1) edge (f1);
\draw (e1) edge[me=2] (f1);

\begin{scope}[shift={(+2.5,0)}]
\node[circle,fill=black,inner sep=0.8pt,draw] (a2) at (-.43,-.5) {};
\node[circle,fill=black,inner sep=0.8pt,draw] (b2) at (-.43,.5) {};
\node[circle,fill=black,inner sep=0.8pt,draw] (c2) at (0,0) {};
\node[circle,fill=black,inner sep=0.8pt,draw] (d2) at (.5,0) {};
\node[circle,fill=black,inner sep=0.8pt,draw] (e2) at (.93,-0.5) {};
\node[circle,fill=black,inner sep=0.8pt,draw] (f2) at (.93,0.5) {};


\draw (a2) edge[me=2] (b2); 
\draw (c2) edge (b2);
\draw (a2) edge (c2);
\draw (c2) edge (d2);
\draw (d2) edge (e2);
\draw (d2) edge (f2);
\draw (1.07, -0.5) circle (0.15);
\draw (1.07, 0.5) circle (0.15);
\end{scope}

\begin{scope}[shift={(+5,0)}]
\node[circle,fill=black,inner sep=0.8pt,draw] (a3) at (-.43,-.5) {};
\node[circle,fill=black,inner sep=0.8pt,draw] (b3) at (-.43,.5) {};
\node[circle,fill=black,inner sep=0.8pt,draw] (c3) at (0,0) {};
\node[circle,fill=black,inner sep=0.8pt,draw] (d3) at (0.43,-0.5) {};
\node[circle,fill=black,inner sep=0.8pt,draw] (e3) at (0.43,0.5) {};
\node[circle,fill=black,inner sep=0.8pt,draw] (f3) at (1,0) {};


\draw (a3) edge[me=2] (b3); 
\draw (c3) edge (b3);
\draw (a3) edge (c3);
\draw (c3) edge (e3);
\draw (d3) edge[me=2] (e3);
\draw (d3) edge (f3);
\draw (1.15, 0) circle (0.15);
\end{scope}

\begin{scope}[shift={(+7.5,0)}]
\node[circle,fill=black,inner sep=0.8pt,draw] (a4) at (-.43,-.5) {};
\node[circle,fill=black,inner sep=0.8pt,draw] (b4) at (-.43,.5) {};
\node[circle,fill=black,inner sep=0.8pt,draw] (c4) at (.07,-0.5) {};
\node[circle,fill=black,inner sep=0.8pt,draw] (d4) at (.07,0.5) {};
\node[circle,fill=black,inner sep=0.8pt,draw] (e4) at (.5,-0.5) {};
\node[circle,fill=black,inner sep=0.8pt,draw] (f4) at (.5 ,0.5) {};


\draw (a4) edge[me=2] (b4); 
\draw (c4) edge (a4);
\draw (b4) edge (d4);
\draw (c4) edge (d4);
\draw (d4) edge (f4);
\draw (c4) edge (e4);
\draw (.65, -0.5) circle (0.15);
\draw (.65, 0.5) circle (0.15);
\end{scope}

\begin{scope}[shift={(+10,0)}]
\node[circle,fill=black,inner sep=0.8pt,draw] (a7) at (-.43,-.5) {};
\node[circle,fill=black,inner sep=0.8pt,draw] (b7) at (-.43,.5) {};
\node[circle,fill=black,inner sep=0.8pt,draw] (c7) at (.07,-0.5) {};
\node[circle,fill=black,inner sep=0.8pt,draw] (d7) at (.07,0.5) {};
\node[circle,fill=black,inner sep=0.8pt,draw] (e7) at (0.5,0) {};
\node[circle,fill=black,inner sep=0.8pt,draw] (f7) at (0.5,0.5) {};


\draw (a7) edge[me=2] (b7); 
\draw (c7) edge (a7);
\draw (b7) edge (d7);
\draw (d7) edge (e7);
\draw (d7) edge (f7);
\draw (0.65,-0) circle (0.15);
\draw (0.65,0.5) circle (0.15);
\draw (.22,-0.5) circle (0.15);
\end{scope}

\begin{scope}[shift={(+12.5,0)}]
\node[circle,fill=black,inner sep=0.8pt,draw] (a8) at (-.43,-.5) {};
\node[circle,fill=black,inner sep=0.8pt,draw] (b8) at (-.43,.5) {};
\node[circle,fill=black,inner sep=0.8pt,draw] (c8) at (.07,-0.5) {};
\node[circle,fill=black,inner sep=0.8pt,draw] (d8) at (.07,0.5) {};
\node[circle,fill=black,inner sep=0.8pt,draw] (e8) at (0.65,0) {};
\node[circle,fill=black,inner sep=0.8pt,draw] (f8) at (0.65,0.5) {};


\draw (a8) edge[me=2] (b8); 
\draw (c8) edge (a8);
\draw (b8) edge (d8);
\draw (d8) edge[me=2] (f8);
\draw (f8) edge (e8);
\draw (0.65,-0.15) circle (0.15);
\draw (.22,-0.5) circle (0.15); 
\end{scope}

\begin{scope}[shift={(+3.5,-2)}]
\node[circle,fill=black,inner sep=0.8pt,draw] (a5) at (-.43,-.5) {};
\node[circle,fill=black,inner sep=0.8pt,draw] (b5) at (-.43,.5) {};
\node[circle,fill=black,inner sep=0.8pt,draw] (c5) at (.07,-0.5) {};
\node[circle,fill=black,inner sep=0.8pt,draw] (d5) at (.07,0.5) {};
\node[circle,fill=black,inner sep=0.8pt,draw] (e5) at (.5,0) {};
\node[circle,fill=black,inner sep=0.8pt,draw] (f5) at (1,0) {};

\draw (a5) edge[me=2] (b5); 
\draw (c5) edge (a5);
\draw (b5) edge (d5);
\draw (c5) edge (d5);
\draw (c5) edge (e5);
\draw (d5) edge (e5);
\draw (e5) edge (f5);
\draw (1.15,0) circle (0.15);
\end{scope}

\begin{scope}[shift={(+1,-2)}]
\node[circle,fill=black,inner sep=0.8pt,draw] (a6) at (-.43,-.5) {};
\node[circle,fill=black,inner sep=0.8pt,draw] (b6) at (-.43,.5) {};
\node[circle,fill=black,inner sep=0.8pt,draw] (c6) at (.07,-0.5) {};
\node[circle,fill=black,inner sep=0.8pt,draw] (d6) at (.07,0.5) {};
\node[circle,fill=black,inner sep=0.8pt,draw] (e6) at (0.57,-0.5) {};
\node[circle,fill=black,inner sep=0.8pt,draw] (f6) at (0.57,0.5) {};

\draw (a6) edge[me=2] (b6); 
\draw (c6) edge (a6);
\draw (b6) edge (d6);
\draw (c6) edge (d6);
\draw (d6) edge (f6);
\draw (c6) edge (e6);
\draw (e6) edge[me=2] (f6);
\end{scope}

\begin{scope}[shift={(+6.5,-2)}]
\node[circle,fill=black,inner sep=0.8pt,draw] (a9) at (-1,0) {};
\node[circle,fill=black,inner sep=0.8pt,draw] (b9) at (-0.5, 0.5) {};
\node[circle,fill=black,inner sep=0.8pt,draw] (c9) at (0.5, 0.5) {};
\node[circle,fill=black,inner sep=0.8pt,draw] (d9) at (1,0) {};
\node[circle,fill=black,inner sep=0.8pt,draw] (e9) at (0.5, -0.5) {};
\node[circle,fill=black,inner sep=0.8pt,draw] (f9) at (-0.5, -0.5) {};

\draw (a9)--(b9);\path (b9) edge [bend left] (c9); \path (b9) edge [bend right] (c9);
\draw (c9)--(d9);\path (d9) edge [bend left] (e9); \path (d9) edge [bend right] (e9);
\draw (f9)--(e9);\path (f9) edge [bend left] (a9); \path (a9) edge [bend left] (f9);
\end{scope}

\begin{scope}[shift={(+9.25,-2)}]
\node[circle,fill=black,inner sep=0.8pt,draw] (a10) at (-.43,-.5) {};
\node[circle,fill=black,inner sep=0.8pt,draw] (b10) at (-.43,.5) {};
\node[circle,fill=black,inner sep=0.8pt,draw] (c10) at (.2,-0.5) {};
\node[circle,fill=black,inner sep=0.8pt,draw] (d10) at (.2,0.5) {};
\node[circle,fill=black,inner sep=0.8pt,draw] (e10) at (0.6,0) {};
\node[circle,fill=black,inner sep=0.8pt,draw] (f10) at (1.1,0) {};

\draw (1.25,0) circle (0.15);

\draw (a10) edge (b10); 
\draw (c10) edge[me=2] (a10);
\draw (b10) edge[me=2] (d10);
\draw (d10) edge (e10);
\draw (c10) edge (e10);
\draw (e10) edge (f10);  
\end{scope}

\begin{scope}[shift={(+12,-2 )}]
\node[circle,fill=black,inner sep=0.8pt,draw] (a11) at (-.43,-.5) {};
\node[circle,fill=black,inner sep=0.8pt,draw] (b11) at (-.43,.5) {};
\node[circle,fill=black,inner sep=0.8pt,draw] (c11) at (.07,-0.5) {};
\node[circle,fill=black,inner sep=0.8pt,draw] (d11) at (.07,0.5) {};
\node[circle,fill=black,inner sep=0.8pt,draw] (e11) at (0.57,-0.5) {};
\node[circle,fill=black,inner sep=0.8pt,draw] (f11) at (0.57,0.5) {};

\draw (a11) edge[me=2] (b11); 
\draw (c11) edge (a11);
\draw (b11) edge (d11);
\draw (c11) edge (f11);
\draw (d11) edge (f11);
\draw (d11) edge (e11);
\draw (c11) edge (e11);
\draw (e11) edge (f11);
\end{scope}
\end{tikzpicture}  
\caption{\small{Topological types of trivalent genus $4$ graphs, possibly with loops.}}
\label{fig:trivalent simple graphs with loops of genus 4}
\end{figure}
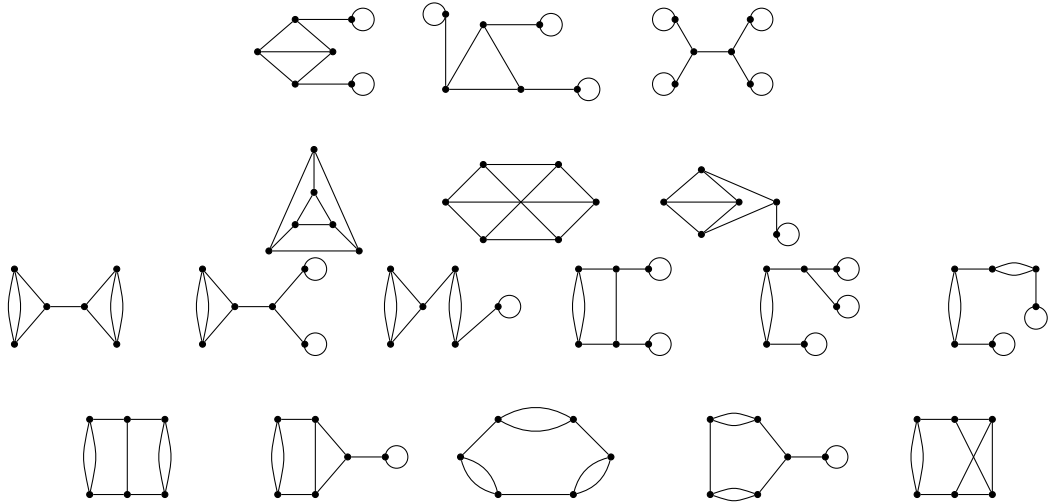

We can apply the Bridge and Loop lemmas to all graphs homeomorphic to the ones on the first and third rows of Figure~\ref{fig:trivalent simple graphs with loops of genus 4}. Bridges and loops are denoted by dashed lines and the vertices decorated with nonzero integers are made large. 

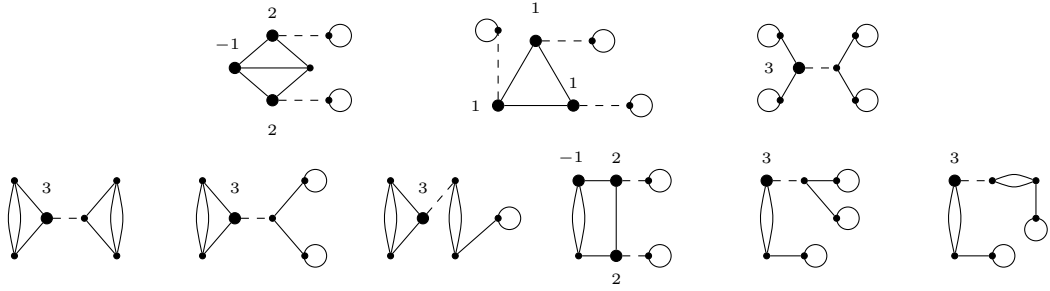
\begin{figure}[H]
\begin{tikzpicture}
\begin{scope}[shift={(+2.5,0)}]
\node[circle,fill=black,inner sep=1.5pt,draw] (11) at (0,0) {};
\node[circle,fill=black,inner sep=0.8pt,draw] (12) at (1,0) {};
\node[circle,fill=black,inner sep=1.5pt,draw] (13) at (0.5,0.43) {};
\node[circle,fill=black,inner sep=0.8pt,draw] (14) at (1.25,0.43) {};
\node[circle,fill=black,inner sep=1.5pt,draw] (15) at (0.5,-0.43) {};
\node[circle,fill=black,inner sep=0.8pt,draw] (16) at (1.25,-0.43) {};

\node () at (-0.1,0.3) {\tiny$-1$};
\node () at (0.5,.73) {\tiny$2$};
\node () at (0.5,-0.83) {\tiny$2$};

\draw (1.4,0.43) circle (0.15); 
\draw (1.4,-0.43) circle (0.15);
\draw (11) -- (13) -- (12); \draw (11)--(15)--(12); \draw (11)--(12); \draw[dashed] (13) -- (14); \draw[dashed] (15)--(16);
\end{scope}

\begin{scope}[shift={(+6,-0.5)}]
\node[circle,fill=black,inner sep=1.5pt,draw] (21) at (0,0) {};
\node[circle,fill=black,inner sep=1.5pt,draw] (22) at (1,0) {};
\node[circle,fill=black,inner sep=1.5pt,draw](23) at (0.5,0.86) {};
\node[circle,fill=black,inner sep=0.8pt,draw] (24) at (0,1) {};
\node[circle,fill=black,inner sep=0.8pt,draw] (25) at (1.25,0.86) {};
\node[circle,fill=black,inner sep=0.8pt,draw](26) at (1.75,0) {};

\draw (-0.15,1) circle (0.15); 
\draw (1.4,0.86) circle (0.15);
\draw (1.9,0) circle (0.15);

\node () at (-0.3,0) {\tiny$1$};
\node () at (1,.3) {\tiny$1$};
\node () at (0.5,1.3) {\tiny$1$};

\draw (21) -- (23) -- (22) -- (21); \draw (21)[dashed]--(24); \draw (23)[dashed] -- (25); \draw[dashed] (22)--(26);
\end{scope}

\begin{scope}[shift={(+10,0)}]
\node[circle,fill=black,inner sep=0.8pt,draw] (61) at (-0.25,-0.43) {};
\node[circle,fill=black,inner sep=0.8pt,draw] (62) at (-0.25,0.43) {};
\node[circle,fill=black,inner sep=1.5pt,draw] (63) at (0,0) {};
\node[circle,fill=black,inner sep=0.8pt,draw] (64) at (.5,0) {};
\node[circle,fill=black,inner sep=0.8pt,draw] (65) at (.75,-0.43) {};
\node[circle,fill=black,inner sep=0.8pt,draw] (66) at (.75, 0.43) {};

\draw (-0.4,-0.43) circle (0.15);
\draw (-0.40,0.43) circle (0.15);
\draw (.90,-0.43) circle (0.15);
\draw (.90,0.43) circle (0.15);

\node () at (-0.4,0) {\tiny$3$};

\draw (61) -- (63); \draw (63)[dashed]-- (64); \draw (64) -- (65); \draw (62) --(63); \draw (64) -- (66);
\end{scope}

\begin{scope}[shift={(0,-2)}]
\node[circle,fill=black,inner sep=0.8pt,draw] (a1) at (-.43,-.5) {};
\node[circle,fill=black,inner sep=0.8pt,draw] (b1) at (-.43,.5) {};
\node[circle,fill=black,inner sep=1.5pt,draw] (c1) at (0,0) {};
\node[circle,fill=black,inner sep=0.8pt,draw] (d1) at (.5,0) {};
\node[circle,fill=black,inner sep=0.8pt,draw] (e1) at (.93,-0.5) {};
\node[circle,fill=black,inner sep=0.8pt,draw] (f1) at (.93,0.5) {};

\node () at (0,0.4) {\tiny$3$};

\draw (a1) edge[me=2] (b1); 
\draw (c1) edge (b1);
\draw (a1) edge (c1);
\draw (c1) edge[dashed] (d1);
\draw (d1) edge (e1);
\draw (d1) edge (f1);
\draw (e1) edge[me=2] (f1); 
\end{scope}

\begin{scope}[shift={(+2.5,-2)}]
\node[circle,fill=black,inner sep=0.8pt,draw] (a2) at (-.43,-.5) {};
\node[circle,fill=black,inner sep=0.8pt,draw] (b2) at (-.43,.5) {};
\node[circle,fill=black,inner sep=1.5pt,draw] (c2) at (0,0) {};
\node[circle,fill=black,inner sep=0.8pt,draw] (d2) at (.5,0) {};
\node[circle,fill=black,inner sep=0.8pt,draw] (e2) at (.93,-0.5) {};
\node[circle,fill=black,inner sep=0.8pt,draw] (f2) at (.93,0.5) {};

\node () at (0,0.4) {\tiny$3$};

\draw (a2) edge[me=2] (b2); 
\draw (c2) edge (b2);
\draw (a2) edge (c2);
\draw (c2) edge[dashed] (d2);
\draw (d2) edge (e2);
\draw (d2) edge (f2);
\draw (1.07, -0.5) circle (0.15);
\draw (1.07, 0.5) circle (0.15);
\end{scope}

\begin{scope}[shift={(+5,-2)}]
\node[circle,fill=black,inner sep=0.8pt,draw] (a3) at (-.43,-.5) {};
\node[circle,fill=black,inner sep=0.8pt,draw] (b3) at (-.43,.5) {};
\node[circle,fill=black,inner sep=1.5pt,draw] (c3) at (0,0) {};
\node[circle,fill=black,inner sep=0.8pt,draw] (d3) at (0.43,-0.5) {};
\node[circle,fill=black,inner sep=0.8pt,draw] (e3) at (0.43,0.5) {};
\node[circle,fill=black,inner sep=0.8pt,draw] (f3) at (1,0) {};

\node () at (0,0.4) {\tiny$3$};

\draw (a3) edge[me=2] (b3); 
\draw (c3) edge (b3);
\draw (a3) edge (c3);
\draw (c3) edge[dashed] (e3);
\draw (d3) edge[me=2] (e3);
\draw (d3) edge (f3);
\draw (1.15, 0) circle (0.15);
\end{scope}

\begin{scope}[shift={(+7.5,-2)}]
\node[circle,fill=black,inner sep=0.8pt,draw] (a4) at (-.43,-.5) {};
\node[circle,fill=black,inner sep=1.5pt,draw] (b4) at (-.43,.5) {};
\node[circle,fill=black,inner sep=1.5pt,draw] (c4) at (.07,-0.5) {};
\node[circle,fill=black,inner sep=1.5pt,draw] (d4) at (.07,0.5) {};
\node[circle,fill=black,inner sep=0.8pt,draw] (e4) at (.5,-0.5) {};
\node[circle,fill=black,inner sep=0.8pt,draw] (f4) at (.5 ,0.5) {};

\node () at (.07,-0.8) {\tiny$2$};
\node () at (.07,0.8) {\tiny$2$};
\node () at (-.53,0.8) {\tiny$-1$};

\draw (a4) edge[me=2] (b4); 
\draw (c4) edge (a4);
\draw (b4) edge (d4);
\draw (c4) edge (d4);
\draw (d4) edge[dashed] (f4);
\draw (c4) edge[dashed] (e4);
\draw (.65, -0.5) circle (0.15);
\draw (.65, 0.5) circle (0.15);
\end{scope}

\begin{scope}[shift={(+10,-2)}]
\node[circle,fill=black,inner sep=0.8pt,draw] (a7) at (-.43,-.5) {};
\node[circle,fill=black,inner sep=1.5pt,draw] (b7) at (-.43,.5) {};
\node[circle,fill=black,inner sep=0.8pt,draw] (c7) at (.07,-0.5) {};
\node[circle,fill=black,inner sep=0.8pt,draw] (d7) at (.07,0.5) {};
\node[circle,fill=black,inner sep=0.8pt,draw] (e7) at (0.5,0) {};
\node[circle,fill=black,inner sep=0.8pt,draw] (f7) at (0.5,0.5) {};

\node () at (-.43,0.8) {\tiny$3$};

\draw (a7) edge[me=2] (b7); 
\draw (c7) edge (a7);
\draw (b7) edge[dashed] (d7);
\draw (d7) edge (e7);
\draw (d7) edge (f7);
\draw (0.65,-0) circle (0.15);
\draw (0.65,0.5) circle (0.15);
\draw (.22,-0.5) circle (0.15);
\end{scope}

\begin{scope}[shift={(+12.5,-2)}]
\node[circle,fill=black,inner sep=0.8pt,draw] (a8) at (-.43,-.5) {};
\node[circle,fill=black,inner sep=1.5pt,draw] (b8) at (-.43,.5) {};
\node[circle,fill=black,inner sep=0.8pt,draw] (c8) at (.07,-0.5) {};
\node[circle,fill=black,inner sep=0.8pt,draw] (d8) at (.07,0.5) {};
\node[circle,fill=black,inner sep=0.8pt,draw] (e8) at (0.65,0) {};
\node[circle,fill=black,inner sep=0.8pt,draw] (f8) at (0.65,0.5) {};

\node () at (-.43,0.8) {\tiny$3$};

\draw (a8) edge[me=2] (b8); 
\draw (c8) edge (a8);
\draw (b8) edge[dashed] (d8);
\draw (d8) edge[me=2] (f8);
\draw (f8) edge (e8);
\draw (0.65,-0.15) circle (0.15);
\draw (.22,-0.5) circle (0.15); 
\end{scope}
\end{tikzpicture}
\caption{\small{Genus 4 graphs with bridges or more than one loop.}}
\end{figure}

The next result reduces the search for divisors $D$ with $\deg(D)=3$ and $\rk(D)\geq 1$ to one of finding decompositions of graphs into appropriately chosen connected subgraphs, determined by running Dhar's burning algorithm once. Note that an effective divisor $D$ on $G$ is of rank at least $1$ if, for any $v\in V(G)$, the algorithm applied to $D - (v)$ terminates in an effective divisor.

\begin{lemma}
\label{lem:configurations-for-genus-4}
Let $G$ be a graph of genus 4 and let $D\in \Div_{+}(G)$ be of degree 3. For any $v\in V(G) \backslash \supp(D)$, run Dhar's burning algorithm for the divisor $D - (v)$, starting the fire at $v$. Let $G_v$ be the closure of the connected graph consisting of all burnt vertices and edges on the first run of the algorithm. Let $D_v$ be the restriction from $D$ to~$G_v$. If, for every $v\in V(G) \backslash \supp(D)$, the corresponding $G_v$ and $D_v$ are among those in Figure~\ref{fig: possible-divisors-4}, then $\rk(D) \geq 1$.

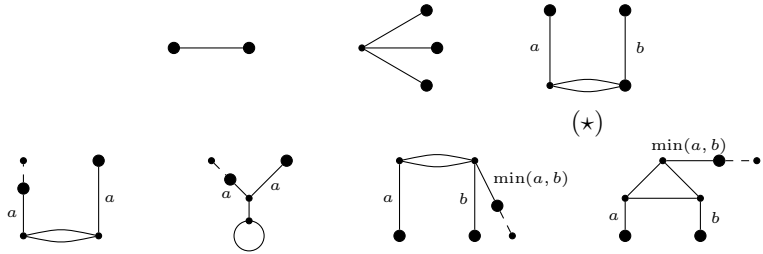
\begin{figure}[H]
\begin{tikzpicture}[scale=1]
\begin{scope}[shift={(+3.5,0)}]
\node[circle,fill=black,inner sep=1.5pt,draw] (a) at (0,0) {};
\node[circle,fill=black,inner sep=1.5pt,draw] (b) at (1,0) {};

\draw (a) edge (b);
\end{scope}

\begin{scope}[shift={(+6,0)}]
\node[circle,fill=black,inner sep=0.8pt,draw] (a) at (0,0) {};
\node[circle,fill=black,inner sep=1.5pt,draw] (b) at (1,0) {};
\node[circle,fill=black,inner sep=1.5pt,draw] (c) at (0.86,-0.5) {};
\node[circle,fill=black,inner sep=1.5pt,draw] (d) at (0.86,0.5) {};

\draw (a) edge (b);
\draw (a) edge (c);
\draw (a) edge (d);
\end{scope}

\begin{scope}[shift={(+8.5,-0.5)}]
\node[circle,fill=black,inner sep=0.8pt,draw] (a) at (0,0) {};
\node[circle,fill=black,inner sep=1.5pt,draw] (b) at (1,0) {};
\node[circle,fill=black,inner sep=1.5pt,draw] (c) at (0,1) {};
\node[circle,fill=black,inner sep=1.5pt,draw] (d) at (1,1) {};

\node () at (-.2,0.5) {\tiny$a$};
\node () at (1.2,0.5) {\tiny$b$};
\node () at (0.5,-0.5) {$(\star)$};

\draw (a) edge[me=2] (b);
\draw (a) edge (c);
\draw (b) edge (d);
\end{scope}

\begin{scope}[shift={(1.5,-2.5)}]
\node[circle,fill=black,inner sep=0.8pt,draw] (a) at (0,0) {};
\node[circle,fill=black,inner sep=0.8pt,draw] (b) at (1,0) {};
\node[circle,fill=black,inner sep=1.5pt,draw] (c) at (1,1) {};
\node[circle,fill=black,inner sep=0.8pt,draw] (d) at (0,1) {};
\node[circle,fill=black,inner sep=1.5pt,draw] (new) at (0,0.63) {};

\node () at (1.15,0.5) {\tiny$a$};
\node () at (-0.15,0.3) {\tiny$a$};

\draw (a) edge[me=2] (b);
\draw (a) -- (new);
\draw (new) edge[dashed] (d);
\draw (b) edge (c);
\end{scope}

\begin{scope}[shift={(+4,-2.5)}]
\node[circle,fill=black,inner sep=0.8pt,draw] (a) at (0,1) {};
\node[circle,fill=black,inner sep=1.5pt,draw] (b) at (1,1) {};
\node[circle,fill=black,inner sep=0.8pt,draw] (c) at (.5,.5) {};
\node[circle,fill=black,inner sep=0.8pt,draw] (d) at (.5,.2) {};
\node[circle,fill=black,inner sep=1.5pt,draw] (new) at (0.25,0.75) {};

\node () at (0.2,0.55) {\tiny$a$};
\node () at (.85,0.65) {\tiny$a$};

\draw (a) edge[dashed] (new);
\draw (new) -- (c);
\draw (c) -- (b); \draw (c) -- (d);
\draw (.5,0) circle (0.2);
\end{scope}
\begin{scope}[shift={(6.5,-1.5)}]
\node[circle,fill=black,inner sep=0.8pt,draw] (a) at (0,0) {};
\node[circle,fill=black,inner sep=0.8pt,draw] (b) at (1,0) {};
\node[circle,fill=black,inner sep=1.5pt,draw] (c) at (0,-1) {};
\node[circle,fill=black,inner sep=1.5pt,draw] (d) at (1,-1) {};
\node[circle,fill=black,inner sep=0.8pt,draw] (e) at (1.5,-1) {};
\node[circle,fill=black,inner sep=1.5pt,draw] (mid-be) at (1.3,-.6) {};

\node () at (.85,-.5) {\tiny$b$};
\node () at (-.15,-.5) {\tiny$a$};
\node () at (1.75,-.25) {\tiny$\min(a,b)$};

\draw (a) edge[me=2] (b);
\draw (a) -- (c);
\draw (d) -- (b) -- (mid-be);
\draw (mid-be) edge[dashed] (e);
\end{scope}

\begin{scope}[shift={(9.5,-2)}]
\node[circle,fill=black,inner sep=0.8pt,draw] (a) at (0,0) {};
\node[circle,fill=black,inner sep=0.8pt,draw] (b) at (1,0) {};
\node[circle,fill=black,inner sep=0.8pt,draw] (c) at (0.5,.5) {};
\node[circle,fill=black,inner sep=1.5pt,draw] (d) at (1.25,.5) {};
\node[circle,fill=black,inner sep=1.5pt,draw] (e) at (0,-.5) {};
\node[circle,fill=black,inner sep=1.5pt,draw] (f) at (1,-.5) {};
\node[circle,fill=black,inner sep=0.8pt,draw] (new) at (1.75,.5) {};

\node () at (-.15,-.25) {\tiny$a$};
\node () at (0.85,.7) {\tiny$\min(a,b)$};
\node () at (1.2,-.25) {\tiny$b$};

\draw (a)--(b)--(c)--(d);
\draw (d) edge[dashed] (new);
\draw (a) -- (c);
\draw (a) -- (e);
\draw (b) -- (f);

\end{scope}
\end{tikzpicture}

\caption{\small{Possible divisors $D_v$ on $G_v$. Note that $G_v$ consists of the solid edges.}}
\label{fig:configurations-for-genus-4}\label{fig: possible-divisors-4}
\end{figure}
\end{lemma}
\begin{proof}

To show that $\rk(D) \geq 1$, choose a vertex $v\in V(G)\backslash \supp(D)$. By assumption, $G_v$ and $D_v$ are among the above configurations. Note that $\rk (D_v) \geq 1$ on $G_v$ by running Dhar's burning algorithm. This can be easily checked on all cases with the possible exception of~$(\star)$. For this particular case, we consider two cases: either $a\geq b$, or $a<b$. We can chip fire the configuration as shown below. 

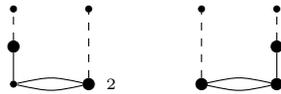
\begin{figure}[H]
\begin{tikzpicture}
\node[circle,fill=black,inner sep=0.8pt,draw] (a) at (0,0) {};
\node[circle,fill=black,inner sep=1.5pt,draw] (b) at (1,0) {};
\node[circle,fill=black,inner sep=0.8pt,draw] (c) at (0,1) {};
\node[circle,fill=black,inner sep=0.8pt,draw] (d) at (1,1) {};
\node[circle,fill=black,inner sep=1.5pt,draw] (f) at (0,0.5) {};

\node () at (1.3,0) {\tiny$2$};

\draw (a) edge[me=2] (b);
\draw (a) edge (f);
\draw (f) edge[dashed] (c);
\draw (b) edge[dashed] (d);

\begin{scope}[shift={(+2.5,0)}]
\node[circle,fill=black,inner sep=1.5pt,draw] (a) at (0,0) {};
\node[circle,fill=black,inner sep=1.5pt,draw] (b) at (1,0) {};
\node[circle,fill=black,inner sep=0.8pt,draw] (c) at (0,1) {};
\node[circle,fill=black,inner sep=0.8pt,draw] (d) at (1,1) {};
\node[circle,fill=black,inner sep=1.5pt,draw] (f) at (1,0.5) {};

\draw (a) edge[me=2] (b);
\draw (a) edge[dashed] (c);
\draw (b) edge (f);
\draw (f) edge[dashed] (d);
\end{scope}
\end{tikzpicture}
\end{figure}

\noindent
Note that at each run of the Dhar's burning for the divisor $D-(v)$, the set of unburnt vertices is contained in $G_v$. Thus, to find an effective representative of $D-(v)$, we simply run Dhar's burning algorithm, chip firing the unburned vertices. Since $D_v-(v)\sim E_v$ for an effective $E_v\in \Div(G_v)$, it follows that $D- (v) \sim E$ for an effective $E\in \Div(G)$. The choice of $v$ was arbitrary, so $\rk (D)\geq 1$.
 \end{proof}

\begin{remark}
\label{rmrk:after-genus4}
{\em Note that in the third (resp. in the fourth) configuration on the top row of Figure~\ref{fig: possible-divisors-4}, we can chip fire away from the cycle (resp. the loop) until one of the two chips lands on a topological vertex. Hereafter, we assume that whenever $D$ has a $D_v$ among these two configurations, at least one of the two chips sits on a topological vertex. For the configurations on the second row, we assume that the chip at distance $\max(a,b)$ sits on a topological vertex.}
\end{remark} 

For each of the remaining families of topologically trivalent graphs of genus~$4$, we separately construct a divisor $D$ with $\deg(D)=3$ and $\rk (D)\geq 1$. The families are numbered in the order in which they appear in Figure~\ref{fig:trivalent simple graphs with loops of genus 4}. The first three are simple graphs with loops and the remaining five are multigraphs. In all figures below the letters $a,b,c,d,x,y$ denote the length of the topological edge situated next to them. Many of the shown divisors have rank at least $1$ as a consequence of Lemma~\ref{lem:configurations-for-genus-4}. For the cases when Lemma~\ref{lem:configurations-for-genus-4} applies, we draw all edges participating in the same configuration $G_v$ with corresponding $D_v$ as having the same edge pattern (dotted, dashed, etc.). For each divisor the vertices with chips are made larger. 

\subsubsection{Straightforward cases.}
\label{sec:straightforward-cases-genus4}
Some homeomorphic families of graphs admit a degree 3 divisor of rank at least~$1$, for any choice of edge length. For these families, such divisors are shown in Figure~\ref{straightforward-cases}. They have rank at least $1$ by Lemma~\ref{lem:configurations-for-genus-4}. 
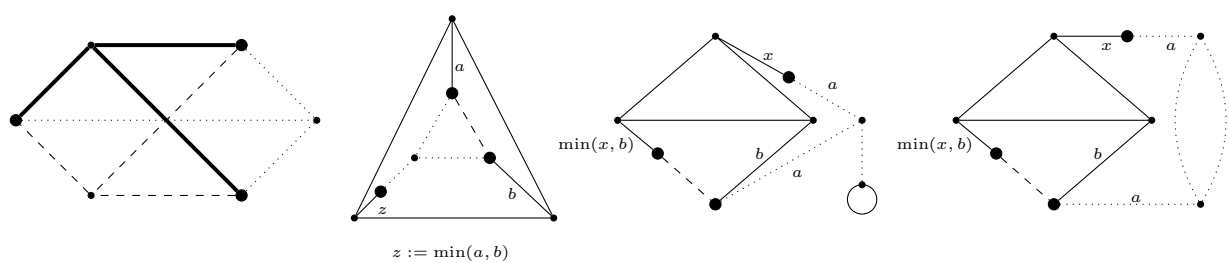
\begin{figure}[H]
\begin{tikzpicture}
\node[circle,fill=black,inner sep=1.5pt,draw] (a50) at (-2,0) {};
\node[circle,fill=black,inner sep=0.8pt,draw] (a51) at (-1,1) {};
\node[circle,fill=black,inner sep=1.5pt,draw] (a52) at (1,1) {};
\node[circle,fill=black,inner sep=0.8pt,draw] (a53) at (2,0) {};
\node[circle,fill=black,inner sep=0.8pt,draw] (a54) at (-1,-1) {};
\node[circle,fill=black,inner sep=1.5pt,draw] (a55) at (1,-1) {};

\draw (a50) edge[line width=1.5pt] (a51); \draw(a51) edge[line width=1.5pt] (a52);

\draw (a52)[dotted]--(a53); \draw (a53)[dotted]--(a55); \draw (a55)[dashed]--(a54); \draw (a54)[dashed]--(a50);\draw (a51)[line width=1.5pt]--(a55);\draw (a52)[dashed]--(a54);\draw (a50)[dotted]--(a53);

\begin{scope}[shift={(+3.3,-0.5)}, scale=0.5]
\node[circle,fill=black,inner sep=1.5pt,draw] (31) at (1,1.72) {};
\node[circle,fill=black,inner sep=0.8pt,draw] (32) at (0,0) {};
\node[circle,fill=black,inner sep=1.5pt,draw] (33) at (2,0) {};
\node[circle,fill=black,inner sep=0.8pt,draw] (34) at (1,3.7) {};
\node[circle,fill=black,inner sep=0.8pt,draw] (35) at (-1.6,-1.6) {};
\node[circle,fill=black,inner sep=0.8pt,draw] (36) at (3.7,-1.6) {};
\node[circle,fill=black,inner sep=1.5pt,draw] (37) at (-0.9,-0.9) {};

\node () at (-0.85,-1.4) {\tiny$z$};
\node () at (1.2,2.40) {\tiny$a$};
\node () at (2.6,-0.95) {\tiny$b$};

\node () at (2.5,-1,4) {\tiny $z:=\min(a,b)$};

\draw (31) -- (34) -- (35) -- (37);
\draw (35) -- (36) -- (33);
\draw (36) -- (34);

\draw (31) edge[dashed] (33);

\draw (37) edge[dotted] (32);
\draw (31) edge[dotted] (32);
\draw (33) edge[dotted] (32);
\end{scope}

\begin{scope}[shift={(+6,0)}, scale = 0.65]
\node[circle,fill=black,inner sep=0.8pt,draw] (41) at (0,0) {};
\node[circle,fill=black,inner sep=0.8pt,draw] (42) at (4,0) {};
\node[circle,fill=black,inner sep=0.8pt,draw] (43) at (2,1.72) {};
\node[circle,fill=black,inner sep=0.8pt,draw] (44) at (5,0) {};
\node[circle,fill=black,inner sep=1.5pt,draw] (45) at (2,-1.72) {};
\node[circle,fill=black,inner sep=0.8pt,draw] (46) at (5,-1.32) {};
\node[circle,fill=black,inner sep=1.5pt,draw] (47) at (3.5,.88) {};
\node[circle,fill=black,inner sep=1.5pt,draw] (48) at (.82,-.68) {};

\node () at (4.4,0.7) {\tiny$a$};
\node () at (3.08,1.3) {\tiny$x$};
\node () at (3.68,-1.1) {\tiny$a$};
\node () at (-0.43,-0.5) {\tiny$\min(x,b)$};
\node () at (2.9,-0.67) {\tiny$b$};

\draw (5,-1.62) circle (0.3);

\draw (47) -- (43)--(42)--(41)--(48);
\draw (42)--(45);
\draw (41) -- (43);

\draw (48) edge[dashed] (45);

\draw (47) edge[dotted] (44);
\draw (45) edge[dotted] (44);
\draw (46) edge[dotted] (44);
\end{scope}

\begin{scope}[shift={(+10.5,0)}, scale=0.65]
\node[circle,fill=black,inner sep=0.8pt,draw] (41) at (0,0) {};
\node[circle,fill=black,inner sep=0.8pt,draw] (42) at (4,0) {};
\node[circle,fill=black,inner sep=0.8pt,draw] (43) at (2,1.72) {};
\node[circle,fill=black,inner sep=0.8pt,draw] (44) at (5,1.72) {};
\node[circle,fill=black,inner sep=1.5pt,draw] (45) at (2,-1.72) {};
\node[circle,fill=black,inner sep=0.8pt,draw] (46) at (5,-1.72) {};
\node[circle,fill=black,inner sep=1.5pt,draw] (47) at (3.5,1.72) {};
\node[circle,fill=black,inner sep=1.5pt,draw] (48) at (.82,-.68) {};

\node () at (4.4,1.5) {\tiny$a$};
\node () at (3.08,1.5) {\tiny$x$};
\node () at (3.68,-1.6) {\tiny$a$};
\node () at (-0.43,-0.5) {\tiny$\min(x,b)$};
\node () at (2.9,-0.67) {\tiny$b$};


\draw (47) -- (43)--(42)--(41)--(48);
\draw (42)--(45);
\draw (41) -- (43);

\draw (48) edge[dashed] (45);

\draw (47) edge[dotted] (44);
\draw (45) edge[dotted] (46);
\draw (46) edge[bend left,dotted] (44);
\draw (46) edge[bend right,dotted] (44);
\end{scope}
\end{tikzpicture}
\caption{\small{In the graphs above, chips are placed on the large vertices. These divisors have degree~$3$ and rank at least $1$.}}\label{straightforward-cases}
\end{figure}
\vspace{-.7cm}
\subsubsection{First family}
Let $b$ be the length of the top right topological edge, and $c$ be that of the bottom right topological edge. For this family we consider three separate cases, depicted in Figure~\ref{genus-4-first} below. The leftmost depicts the rank $1$ divisor, when $b\geq c$. The rank calculation follows immediately from Lemma~\ref{lem:configurations-for-genus-4}. The next two cases depict the situation where $b < c$. Set $y=c-b$. The second one occurs when $y\geq \min (x,d)$, and the third occurs otherwise. That these divisors have rank $1$ follows by running Dhar's algorithm. 
\begin{figure}[H]
\begin{tikzpicture}[scale=0.8]
\node[circle,fill=black,inner sep=0.8pt,draw] (a6) at (-1,-1) {};
\node[circle,fill=black,inner sep=0.8pt,draw] (b6) at (-1,1) {};
\node[circle,fill=black,inner sep=1.5pt,draw] (c6) at (1,-1) {};
\node[circle,fill=black,inner sep=0.8pt,draw] (d6) at (1,1) {};
\node[circle,fill=black,inner sep=0.8pt,draw] (e6) at (3,-1) {};
\node[circle,fill=black,inner sep=0.8pt,draw] (f6) at (3,1) {};
\node[circle,fill=black,inner sep=1.5pt,draw] (additional-a) at (0.2,1) {};
\node[circle,fill=black,inner sep=1.5pt,draw] (additional-c) at (1.9,-1) {};

\node () at (0,-1.2) {\tiny $a$};
\node () at (-0.3,1.2) {\tiny $a$};
\node () at (0.6,1.2) {\tiny $x$};
\node () at (2,1.2) {\tiny $b$};
\node () at (1.2,0) {\tiny $d$};
\node () at (2.5,-1.2) {\tiny $z$};
\node () at (1.3,-1.7) {\tiny $z:=b+\min(x,d)$};

\draw (a6) edge[me=2] (b6); 
\draw (c6) edge (a6);
\draw (b6) edge (additional-a);

\draw (additional-a) edge(d6);
\draw (c6) edge (d6);
\draw (d6) edge (f6);
\draw (e6) edge[me=2] (f6); 
\draw (additional-c) edge (e6);

\draw (c6) edge (additional-c);

\begin{scope}[shift={(5.5,0)}]
\node[circle,fill=black,inner sep=0.8pt,draw] (a6) at (-1,-1) {};
\node[circle,fill=black,inner sep=0.8pt,draw] (b6) at (-1,1) {};
\node[circle,fill=black,inner sep=1.5pt,draw] (c6) at (1,-1) {};
\node[circle,fill=black,inner sep=0.8pt,draw] (d6) at (1,1) {};
\node[circle,fill=black,inner sep=0.8pt,draw] (e6) at (3,-1) {};
\node[circle,fill=black,inner sep=0.8pt,draw] (f6) at (3,1) {};
\node[circle,fill=black,inner sep=1.5pt,draw] (additional-a) at (0.2,1) {};
\node[circle,fill=black,inner sep=.7pt,draw] (additional-c) at (1.7,-1) {};
\node[circle,fill=black,inner sep=1.5pt,draw] (additional-d) at (1,0.4) {};

\node () at (0,-1.2) {\tiny $a$};
\node () at (-0.3,1.2) {\tiny $a$};
\node () at (0.6,1.2) {\tiny $x$};
\node () at (1.2,0.7) {\tiny $y$};
\node () at (1.35,-1.2) {\tiny $y$};
\node () at (2,1.2) {\tiny $b$};
\node () at (1.45,-0.3) {\tiny $d-y$};
\node () at (2.4,-1.2) {\tiny $b$};

\draw (a6) edge[me=2] (b6); 
\draw (c6) edge (a6);
\draw (b6) edge(additional-a);

\draw (c6) edge (additional-d);

\draw (additional-a) -- (d6);
\draw (additional-d)--(d6);
\draw (d6) edge (f6);
\draw (c6) edge (e6);
\draw (e6) edge[me=2] (f6); 
\end{scope}

\begin{scope}[shift={(-5.5,0)}]
\node[circle,fill=black,inner sep=0.8pt,draw] (a6) at (-1,-1) {};
\node[circle,fill=black,inner sep=0.8pt,draw] (b6) at (-1,1) {};
\node[circle,fill=black,inner sep=1.5pt,draw] (c6) at (1,-1) {};
\node[circle,fill=black,inner sep=0.8pt,draw] (d6) at (1,1) {};
\node[circle,fill=black,inner sep=0.8pt,draw] (e6) at (3,-1) {};
\node[circle,fill=black,inner sep=0.8pt,draw] (f6) at (3,1) {};
\node[circle,fill=black,inner sep=1.5pt,draw] (additional-a) at (0.2,1) {};
\node[circle,fill=black,inner sep=1.5pt,draw] (additional-b) at (1.7,1) {};

\node () at (0,-1.2) {\tiny $a$};
\node () at (-0.3,1.2) {\tiny $a$};
\node () at (0.6,1.2) {\tiny $x$};
\node () at (2,-1.2) {\tiny $c$};
\node () at (1.2,0) {\tiny $d$};
\node () at (2.5,1.2) {\tiny $c$};

\draw (a6) edge[me=2] (b6); 
\draw (c6) edge (a6);
\draw (b6) edge(additional-a);

\draw (additional-a) edge (d6);
\draw (c6) edge (d6);
\draw (d6) edge (additional-b);

\draw (additional-b)-- (f6);
\draw (c6) edge (e6);
\draw (e6) edge[me=2] (f6); 
\end{scope}
\end{tikzpicture}
\caption{\small{Degree $3$ and rank at least $1$ divisors on graphs of this combinatorial type, depending on the edge lengths as indicated. Lengths of edges are denoted by small letters adjacent to the edge.}}\label{genus-4-first}
\end{figure}
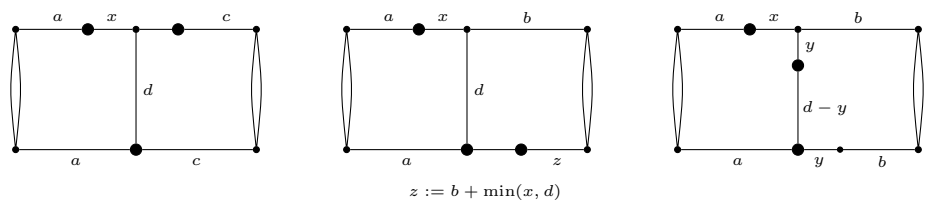

\subsubsection{Second family}
\label{sec:fifth-family-genus4}
Place the first two chips as depicted in Figure~\ref{genus-4-second-family}. The third chip is placed at a distance $\min(x,c+b)$ from the grey vertex along the dashed path. As in the previous case, we see that all divisors are of rank at least $1$. The left and right divisors are the $v$-reduced representatives in the class of the special divisor, where $v$ is the dashed vertex.
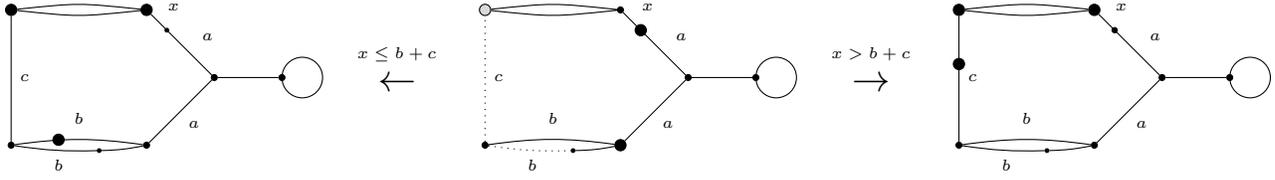
\begin{figure}[H]
\begin{tikzpicture}[scale=0.9]
\node[circle,fill=black,inner sep=0.8pt,draw] (a10) at (-1,-1) {};
\node[circle,fill=gray!30,inner sep=1.5pt,draw] (b10) at (-1,1) {};
\node[circle,fill=black,inner sep=1.5pt,draw] (c10) at (1,-1) {};
\node[circle,fill=black,inner sep=0.8pt,draw] (d10) at (1,1) {};
\node[circle,fill=black,inner sep=0.8pt,draw] (e10) at (2,0) {};
\node[circle,fill=black,inner sep=0.8pt,draw] (f10) at (3,0) {};
\node[circle,fill=black,inner sep=1.5pt,draw] (mid-outside) at (1.3,0.7) {};
\node[circle,fill=black,inner sep=.5pt,draw] (mid-inside) at (0.3, -1.08) {};

\draw (3.3,0) circle (0.3);
\node () at (1.9,0.6) {\tiny$a$};
\node () at (-0.8,0) {\tiny$c$};
\node () at (1.7,-0.7) {\tiny$a$};
\node () at (1.4, 1.04) {\tiny$x$};
\node () at (0.0, -0.6) {\tiny$b$};
\node () at (-0.3, -1.3) {\tiny$b$};

\node () at (-2.3,-0.1) {\LARGE $\leftarrow$};
\node () at (-2.3,0.35) {\tiny$x\leq b+c$};

\draw (a10) edge[dotted] (b10); 
\draw (c10) edge[bend right = 8] (a10);
\draw (a10) edge[dotted, bend right =5] (mid-inside);
\draw (mid-inside) edge[bend right=5] (c10);
\draw (b10) edge[me=2] (d10);
\draw (d10) edge (e10);
\draw (c10) edge (e10);
\draw (e10) edge (f10);

\begin{scope}[shift={(-7,0)}]
\node[circle,fill=black,inner sep=0.8pt,draw] (a10) at (-1,-1) {};
\node[circle,fill=black,inner sep=1.5pt,draw] (b10) at (-1,1) {};
\node[circle,fill=black,inner sep=0.8pt,draw] (c10) at (1,-1) {};
\node[circle,fill=black,inner sep=1.5pt,draw] (d10) at (1,1) {};
\node[circle,fill=black,inner sep=0.8pt,draw] (e10) at (2,0) {};
\node[circle,fill=black,inner sep=0.8pt,draw] (f10) at (3,0) {};
\node[circle,fill=black,inner sep=.5pt,draw] (mid-outside) at (1.3,0.7) {};
\node[circle,fill=black,inner sep=0.5pt,draw] (mid-inside) at (0.3, -1.08) {};
\node[circle,fill=black,inner sep=1.5pt,draw] (mid-down) at (-0.3, -0.92) {};

\draw (3.3,0) circle (0.3);
\node () at (1.9,0.6) {\tiny$a$};
\node () at (-0.8,0) {\tiny$c$};
\node () at (1.7,-0.7) {\tiny$a$};
\node () at (1.4, 1.04) {\tiny$x$};
\node () at (0.0, -0.6) {\tiny$b$};
\node () at (-0.3, -1.3) {\tiny$b$};

\draw (a10) edge (b10); 
\draw (c10) edge[bend right = 8] (a10);
\draw (a10) edge[bend right =5] (mid-inside);
\draw (mid-inside) edge[bend right=5] (c10);
\draw (b10) edge[me=2] (d10);
\draw (d10) edge (e10);
\draw (c10) edge (e10);
\draw (e10) edge (f10);
\end{scope}

\begin{scope}[shift={(+7,0)}]
\node[circle,fill=black,inner sep=0.8pt,draw] (a10) at (-1,-1) {};
\node[circle,fill=black,inner sep=1.5pt,draw] (b10) at (-1,1) {};
\node[circle,fill=black,inner sep=0.8pt,draw] (c10) at (1,-1) {};
\node[circle,fill=black,inner sep=1.5pt,draw] (d10) at (1,1) {};
\node[circle,fill=black,inner sep=0.8pt,draw] (e10) at (2,0) {};
\node[circle,fill=black,inner sep=0.8pt,draw] (f10) at (3,0) {};
\node[circle,fill=black,inner sep=.7pt,draw] (mid-outside) at (1.3,0.7) {};
\node[circle,fill=black,inner sep=.5pt,draw] (mid-inside) at (0.3, -1.08) {};
\node[circle,fill=black,inner sep=1.5pt,draw] (mid-left) at (-1, 0.2) {};

\node () at (-2.3,-0.1) {\LARGE $\rightarrow$};
\node () at (-2.3,0.35) {\tiny$x>b+c$};

\draw (3.3,0) circle (0.3);
\node () at (1.9,0.6) {\tiny$a$};
\node () at (-0.8,0) {\tiny$c$};
\node () at (1.7,-0.7) {\tiny$a$};
\node () at (1.4, 1.04) {\tiny$x$};
\node () at (0.0, -0.6) {\tiny$b$};
\node () at (-0.3, -1.3) {\tiny$b$};

\draw (a10) edge (b10); 
\draw (c10) edge[bend right = 8] (a10);
\draw (a10) edge[bend right =5] (mid-inside);
\draw (mid-inside) edge[bend right=5] (c10);
\draw (b10) edge[me=2] (d10);
\draw (d10) edge (e10);
\draw (c10) edge (e10);
\draw (e10) edge (f10);
\end{scope}
\end{tikzpicture}
\caption{\small{Degree $3$  and rank at least $1$ configurations on the second family. The leftmost and rightmost graphs depict the special divisor depending on which among $x$ and $b+c$ is larger.}}\label{genus-4-second-family}
\end{figure}

\subsubsection{Third family}
For this family we consider four cases. Let $y=c-b$, where $b$ and $c$ are the lengths of the top and bottom topological edges, respectively. Each divisor has rank at least $1$ which follows by Dhar's burning algorithm. It suffices to run the algorithm for one vertex on each topological edge. 

\begin{figure}[H]
\begin{tikzpicture}[scale=0.8]
\node[circle,fill=black,inner sep=0.8pt,draw] (a10) at (-1,-1) {};
\node[circle,fill=black,inner sep=0.8pt,draw] (b10) at (-1,1) {};
\node[circle,fill=black,inner sep=1.5pt,draw] (c10) at (1,-1) {};
\node[circle,fill=black,inner sep=0.8pt,draw] (d10) at (1,1) {};
\node[circle,fill=black,inner sep=0.8pt,draw] (e10) at (2,0) {};
\node[circle,fill=black,inner sep=0.8pt,draw] (f10) at (3,0) {};
\node[circle,fill=black,inner sep=1.5pt,draw] (mid-outside) at (1.3,0.7) {};
\node[circle,fill=black,inner sep=1.5pt,draw] (mid-inside) at (0.3, 1) {};

\draw (3.30,0) circle (0.30);
\node () at (1.7,0.5) {\tiny$a$};
\node () at (1.6,-0.6) {\tiny$a$};
\node () at (1.3, .97) {\tiny$x$};
\node () at (0.0, -1.2) {\tiny$c$};
\node () at (-0.3, 1.2) {\tiny$c$};
\node () at (0.8, 0) {\tiny$d$};

\draw (a10) edge[me=2] (b10); 
\draw (c10) edge (a10);
\draw (b10) edge (d10);
\draw (d10) edge (e10);
\draw (c10) edge (e10);
\draw (e10) edge (f10);
\draw (c10) edge (d10);

\begin{scope}[shift={(+5.5,0)}]
\node[circle,fill=black,inner sep=0.8pt,draw] (a10) at (-1,-1) {};
\node[circle,fill=black,inner sep=0.8pt,draw] (b10) at (-1,1) {};
\node[circle,fill=black,inner sep=1.5pt,draw] (c10) at (1,-1) {};
\node[circle,fill=black,inner sep=0.8pt,draw] (d10) at (1,1) {};
\node[circle,fill=black,inner sep=0.8pt,draw] (e10) at (2,0) {};
\node[circle,fill=black,inner sep=0.8pt,draw] (f10) at (3,0) {};
\node[circle,fill=black,inner sep=1.5pt,draw] (mid-outside) at (1.3,0.7) {};
\node[circle,fill=black,inner sep=1.5pt,draw] (mid-inside) at (0.4, -1) {};

\draw (3.30,0) circle (0.30);
\node () at (1.7,0.5) {\tiny$a$};
\node () at (1.6,-0.6) {\tiny$a$};
\node () at (1.3, .97) {\tiny$x$};
\node () at (0.0, 1.2) {\tiny$b$};
\node () at (-0.4, -1.2) {\tiny$b+z$};
\node () at (0.8, 0) {\tiny$d$};
\node () at (0, -1.7) {\tiny$z:=\min(x,d)$};

\draw (a10) edge[me=2] (b10); 
\draw (c10) --(mid-inside)-- (a10);
\draw (b10) edge (d10);
\draw (d10) edge (e10);
\draw (c10) edge (d10);
\draw (c10) edge (e10);
\draw (e10) edge (f10);
\end{scope}

\begin{scope}[shift={(+11,0)}]
\node[circle,fill=black,inner sep=0.8pt,draw] (a10) at (-1,-1) {};
\node[circle,fill=black,inner sep=0.8pt,draw] (b10) at (-1,1) {};
\node[circle,fill=black,inner sep=1.5pt,draw] (c10) at (1,-1) {};
\node[circle,fill=black,inner sep=0.8pt,draw] (d10) at (1,1) {};
\node[circle,fill=black,inner sep=0.8pt,draw] (e10) at (2,0) {};
\node[circle,fill=black,inner sep=0.8pt,draw] (f10) at (3,0) {};
\node[circle,fill=black,inner sep=1.5pt,draw] (mid-outside) at (1.3,0.7) {};
\node[circle,fill=black,inner sep=.7pt,draw] (mid-inside) at (0.3, -1) {};
\node[circle,fill=black,inner sep=1.5pt,draw] (mid-right) at (1, .3) {};

\draw (3.30,0) circle (0.30);
\node () at (1.7,0.5) {\tiny$a$};
\node () at (1.6,-0.6) {\tiny$a$};
\node () at (1.3, .97) {\tiny$x$};
\node () at (0.0, 1.2) {\tiny$b$};
\node () at (-0.3, -1.2) {\tiny$b$};
\node () at (0.6, -1.2) {\tiny$y$};
\node () at (0.8, 0.6) {\tiny$y$};

\draw (a10) edge[me=2] (b10); 
\draw (c10) edge (a10);
\draw (b10) edge (d10);
\draw (d10) edge (e10);
\draw (c10) --(mid-right)--(d10);
\draw (c10) edge (e10);
\draw (e10) edge (f10);
\end{scope}

\begin{scope}[shift={(+16.5,0)}]
\node[circle,fill=black,inner sep=0.8pt,draw] (a10) at (-1,-1) {};
\node[circle,fill=black,inner sep=0.8pt,draw] (b10) at (-1,1) {};
\node[circle,fill=black,inner sep=1.5pt,draw] (c10) at (1,-1) {};
\node[circle,fill=black,inner sep=0.8pt,draw] (d10) at (1,1) {};
\node[circle,fill=black,inner sep=0.8pt,draw] (e10) at (2,0) {};
\node[circle,fill=black,inner sep=0.8pt,draw] (f10) at (3,0) {};
\node[circle,fill=black,inner sep=1.5pt,draw] (mid-outside) at (1.3,0.7) {};
\node[circle,fill=black,inner sep=0.7pt,draw] (mid-inside) at (0.3, -1) {};
\node[circle,fill=black,inner sep=1.5pt,draw] (mid-inside1) at (0.5, -1) {};

\draw (3.30,0) circle (0.30);
\node () at (1.7,0.5) {\tiny$a$};
\node () at (1.6,-0.6) {\tiny$a$};
\node () at (1.3, .97) {\tiny$x$};
\node () at (0.0, 1.2) {\tiny$b$};
\node () at (-0.3, -1.2) {\tiny$b$};
\node () at (0.8, 0) {\tiny$d$};
\node () at (0.7,-1.6){\tiny $y$};
\node () at (0.75,-0.8){\tiny $z$};
\node () at (2.57,-1.5){\tiny $z:=y-d$};

\draw (a10) edge[me=2] (b10); 
\draw (c10) --(mid-inside1)-- (a10);
\draw (b10) edge (d10);
\draw (d10) edge (e10);
\draw (c10) edge (d10);
\draw (c10) edge (e10);
\draw (e10) edge (f10);
\draw [decoration={brace,mirror,raise=5pt},decorate] (mid-inside.south west) --  node[below=10pt]{}(c10.south east); 

\end{scope}
\end{tikzpicture}
\caption{\small{Cases from left to right are as follows: \textit{1)~$b\geq c$.
2)~$c\geq b$ and $y\geq x$.
3)~$c\geq b$, $y < x$, and $d\geq y$.
4)~$c\geq b$, $y < x$, and $d < y$.}}}
\end{figure}
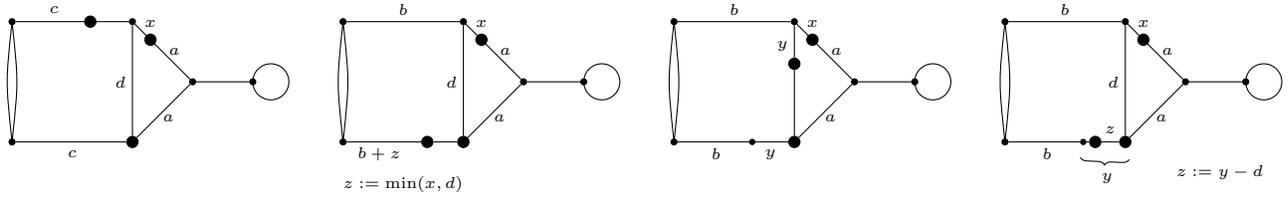

\subsubsection{Fourth family}
\label{subsubsec:loop-of-loops-genus-4}
Let $a,b,c$ be the lengths of the simple edges. Assume $a\leq c \leq b$ and place the chips as shown in Figure~\ref{genus-4-third}. The third chips is placed at length $\min (x,c+d)$ from the gray vertex along the dotted path. To show that the divisor has rank at least $1$ we run Dhar's burning algorithm for one vertex on each topological edge.

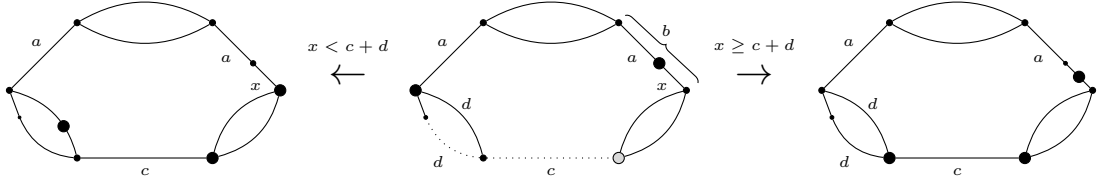
\begin{figure}[H]
\begin{tikzpicture}[scale=0.9]
\begin{scope}[shift={(-6,0)}]
\node[circle,fill=black,inner sep=0.8pt,draw] (a9) at (-2,0) {};
\node[circle,fill=black,inner sep=0.8pt,draw] (b9) at (-1, 1) {};
\node[circle,fill=black,inner sep=0.8pt,draw] (c9) at (1, 1) {};
\node[circle,fill=black,inner sep=1.5pt,draw] (d9) at (2,0) {};
\node[circle,fill=black,inner sep=1.5pt,draw] (e9) at (1, -1) {};
\node[circle,fill=black,inner sep=0.8pt,draw] (f9) at (-1, -1) {};

\node[circle,fill=black,inner sep=0.7pt,draw] (additional-on-b) at (1.6, 0.4) {};
\node[circle,fill=black,inner sep=0,draw] (additional-on-d) at (-1.85, -0.4) {};
\node[circle,fill=black,inner sep=1.5pt,draw] (additional2-on-d) at (-1.2, -0.53) {};

\node () at (-1.6,0.7) {\tiny$a$};
\node () at (0,-1.2) {\tiny$c$};
\node () at (1.2,0.47) {\tiny$a$};
\node () at (1.64,0.06) {\tiny$x$};

\draw (a9)--(b9);\path (b9) edge [bend left] (c9); \path (b9) edge [bend right] (c9);
\draw (c9)--(d9);\path (d9) edge [bend left] (e9); \path (d9) edge [bend right] (e9);
\draw (f9)--(e9);\path (f9) edge [bend left](additional-on-d); \path (additional-on-d) edge (a9); \path (a9) edge [bend left] (f9);

\end{scope}

\node[circle,fill=black,inner sep=1.5pt,draw] (a9) at (-2,0) {};
\node[circle,fill=black,inner sep=0.8pt,draw] (b9) at (-1, 1) {};
\node[circle,fill=black,inner sep=0.8pt,draw] (c9) at (1, 1) {};
\node[circle,fill=black,inner sep=0.8pt,draw] (d9) at (2,0) {};
\node[circle,fill=gray!30,inner sep=1.5pt,draw] (e9) at (1, -1) {};
\node[circle,fill=black,inner sep=0.8pt,draw] (f9) at (-1, -1) {};

\node[circle,fill=black,inner sep=1.5pt,draw] (additional-on-b) at (1.6, 0.4) {};
\node[circle,fill=black,inner sep=.5pt,draw] (additional-on-d) at (-1.85, -0.4) {};

\node () at (-3,0.2) {\LARGE $\leftarrow$};
\node () at (-3,0.65) {\tiny$x< c+d$};

\node () at (-1.6,0.7) {\tiny$a$};
\node () at (1.7,0.88) {\tiny$b$};
\node () at (0,-1.2) {\tiny$c$};
\node () at (-1.25,-0.2) {\tiny$d$};
\node () at (-1.67,-1.05) {\tiny$d$};
\node () at (1.2,0.47) {\tiny$a$};
\node () at (1.64,0.06) {\tiny$x$};

\draw (a9)--(b9);\path (b9) edge [bend left] (c9); \path (b9) edge [bend right] (c9);
\draw (c9)--(d9);\path (d9) edge [bend left] (e9); \path (d9) edge [bend right] (e9);
\draw (f9)[dotted]--(e9);\path (f9) edge [dotted,bend left](additional-on-d); \path (additional-on-d) edge (a9); \path (a9) edge [bend left] (f9);

\draw [decoration={brace,raise=5pt},decorate] (c9.south west) --  node[below=10pt]{}(d9.south east); 

\begin{scope}[shift={(+6,0)}]
\node[circle,fill=black,inner sep=0.8pt,draw] (a9) at (-2,0) {};
\node[circle,fill=black,inner sep=0.8pt,draw] (b9) at (-1, 1) {};
\node[circle,fill=black,inner sep=0.8pt,draw] (c9) at (1, 1) {};
\node[circle,fill=black,inner sep=0.8pt,draw] (d9) at (2,0) {};
\node[circle,fill=black,inner sep=1.5pt,draw] (e9) at (1, -1) {};
\node[circle,fill=black,inner sep=1.5pt,draw] (f9) at (-1, -1) {};

\node[circle,fill=black,inner sep=0.5pt,draw] (additional-on-b) at (1.6, 0.4) {};
\node[circle,fill=black,inner sep=1.5pt,draw] (additional2-on-b) at (1.8, 0.2) {};
\node[circle,fill=black,inner sep=0.5pt,draw] (additional-on-d) at (-1.85, -0.4) {};

\node () at (-1.6,0.7) {\tiny$a$};
\node () at (0,-1.2) {\tiny$c$};
\node () at (-1.25,-0.2) {\tiny$d$};
\node () at (-1.67,-1.05) {\tiny$d$};
\node () at (1.2,0.47) {\tiny$a$};

\node () at (-3,0.2) {\LARGE $\rightarrow$};
\node () at (-3,0.65) {\tiny$x\geq c+d$};

\draw (a9)--(b9);\path (b9) edge [bend left] (c9); \path (b9) edge [bend right] (c9);
\draw (c9)--(d9);\path (d9) edge [bend left] (e9); \path (d9) edge [bend right] (e9);
\draw (f9)--(e9);\path (f9) edge [bend left](additional-on-d); \path (additional-on-d) edge (a9); \path (a9) edge [bend left] (f9);

\end{scope}
\end{tikzpicture}
\caption{\small{Degree $3$ and rank at least $1$ configurations on the fourth family. The leftmost and rightmost graphs depict the special divisor depending on which among $x$ and $c+d$ is larger.}}\label{genus-4-third}
\end{figure}
\subsection{General graphs of genus 4 via edge contractions.}
\label{sec:general-graphs-of-genus-4}
Let $G$ be a genus 4 graph and let $\varepsilon$ be any of its topological edges, which is not a loop. Denote by $G^{\varepsilon}$ the graph obtained by \textbf{contracting along} $\varepsilon$ (see Figure~\ref{fig:edge-contraction}), and let $\varphi^\varepsilon :V(G)\to V(G^\varepsilon)$ be the \textbf{contraction map}, fixing all vertices outside $\varepsilon$ and collapsing all vertices of $\varepsilon$ to one. Then the \textbf{pushforward} $\varphi_{*}^{\varepsilon}:\mathbb{Z}[V(G)]\to \mathbb{Z}[V(G^\varepsilon)]$ is the map  obtained by linearly extending $\varphi^\varepsilon$. For a divisor $D$ of $G$, viewed as an element of $\mathbb{Z}[V(G)]$, we set $D^\varepsilon =\varphi_{*}^\varepsilon(D)$.  

Every genus 4 graph can be obtained by (a series of) edge contractions from a suitably chosen topologically trivalent one. Furthermore, we can assume that each of these contractions is performed along a topological edge of \emph{shortest} length. 

\begin{figure}[H]
\begin{tikzpicture}
\node[circle,fill=black,inner sep=1.5pt,draw] (a) at (0,0) {};
\node[circle,fill=black,inner sep=1.5pt,draw] (b) at (1,0) {};
\node[circle,fill=black,inner sep=0,draw] (c) at (-0.86,0.5) {};
\node[circle,fill=black,inner sep=0,draw] (d) at (-0.86,-0.5) {};
\node[circle,fill=black,inner sep=0,draw] (e) at (1.86,0.5) {};
\node[circle,fill=black,inner sep=0,draw] (f) at (1.86,-0.5) {};

\node[circle,fill=black,inner sep=1.5pt,draw] (x) at (4.5,0) {};
\node[circle,fill=black,inner sep=1.5pt,draw] (y) at (4.5,0) {};
\node[circle,fill=black,inner sep=0.5,draw] (k) at (3.64,0.5) {};
\node[circle,fill=black,inner sep=0.5,draw] (l) at (3.64,-0.5) {};
\node[circle,fill=black,inner sep=0.5,draw] (m) at (5.36,0.5) {};
\node[circle,fill=black,inner sep=0.5,draw] (n) at (5.36,-0.5) {};
\node[circle,fill=black,inner sep=0.5,draw] (1st) at (0.2,0) {};
\node[circle,fill=black,inner sep=0.5,draw] (2nd) at (0.4,0) {};
\node[circle,fill=black,inner sep=0.5,draw] (3rd) at (0.6,0) {};
\node[circle,fill=black,inner sep=0.5,draw] (4th) at (0.8,0) {};

\coordinate (1) at (2.2,0);
\coordinate (2) at (3.36,0);

\draw (a) edge (b);
\draw (a) edge (c);
\draw (a) edge (d);
\draw (b) edge (e);
\draw (b) edge (f);

\draw (x) edge (y);
\draw (x) edge (k);
\draw (x) edge (l);
\draw (y) edge (m);
\draw (y) edge (n);

\node () at (.5, -0.3) {\tiny$\varepsilon$};

\draw [dashed,->] (1) -- (2);

\end{tikzpicture}
\caption{\small{Edge contraction along the topological edge $\varepsilon$.}}
\label{fig:edge-contraction}
\end{figure}
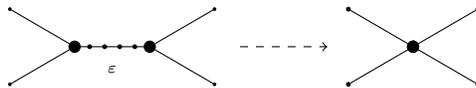

We record the following useful observation.

\begin{lemma}
\label{lem:type(2,2)-and-their-contractions-satisfy-gonality}
If $G$ is a graph with a $(g_1,g_2)$-bridge decomposition and $\varepsilon$ is a topological edge, the graph $G^\varepsilon$ also has $(g_1,g_2)$-bridge decomposition.
\end{lemma}

This result shows that if we start with a trivalent genus 4 graph $G$ with $(2,2)$-bridge decomposition, then any genus 4 graph $G'$ obtained from $G$ by repeated edge contractions satisfies the gonality conjecture. Note that edge contractions do not affect the total number of loops; hence, Lemma~\ref{lem:loop-lemma} applies as well.

\begin{proposition}
\label{prop:contactions-of-genus4-config}
Let $G$ be a genus 4 graph and let $D\in \Div_{+}(G)$ be of degree 3. Suppose all $G_v$ and $D_v$, $v\in V(G)$, as defined in Lemma~\ref{lem:configurations-for-genus-4}, are among the configurations in Figure~\ref{fig:configurations-for-genus-4}. Then, for any set of topological edges $\varepsilon_1,\cdots, \varepsilon_k$ of $G$, the graph $G^{\varepsilon_1,\cdots, \varepsilon_k}$ admits a degree~3 divisor of rank at least $1$.
\end{proposition}
\begin{proof}
Consider an edge contraction along a topological edge $\varepsilon$. Given $v_0 \in {(V(G)\cap \varepsilon)\backslash\supp(D)}$, record~$D_{v_0}$. If $D_{v_0}$ is any of the configurations on the first row of Figure~\ref{fig:configurations-for-genus-4}, $D^\varepsilon$, the pushforward from $D$ to $G^\varepsilon$, has $\rk(D^\varepsilon)\geq 1$. Indeed, if $\varepsilon$ is fully contained in $G_{v_0}$, then contraction along $\varepsilon$ leaves the closure of $G\backslash G_{v_0}$ unchanged, and $D_{v_0}$ remains of rank at least $1$ under contraction along any edge. Recall the assumptions on $D$ from the Remark after Lemma~\ref{lem:configurations-for-genus-4}.

Otherwise, suppose $v_0$ is such that $D_{v_0}$ is any of the configurations on the second row, and $\varepsilon$ is the topological edge of length $a$. Consider $D'_{v_0}$, the divisor on $G^\varepsilon_v$ obtained by substituting $a=0$ for all edges of length $a$ from second row of Figure~\ref{fig:configurations-for-genus-4} as shown in Figure~\ref{fig:edge-contractions-of-second-row}. Note that $\rk(D'_{v_0})\geq 1$ when viewed as a divisor on $G^\varepsilon$, since it satisfies Lemma~\ref{lem:configurations-for-genus-4}.

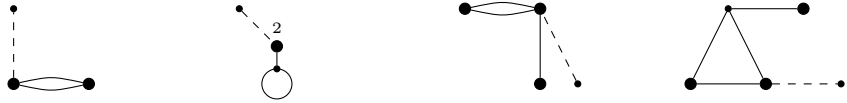
\begin{figure}[H]
\begin{tikzpicture}
\begin{scope}[shift={(0,-1)}]
\node[circle,fill=black,inner sep=1.5pt,draw] (a) at (0,0) {};
\node[circle,fill=black,inner sep=1.5pt,draw] (b) at (1,0) {};
\node[circle,fill=black,inner sep=0.8pt,draw] (c) at (0,1) {};

\draw (a) edge[me=2] (b);
\draw (a)edge[dashed] (c);
\end{scope}

\begin{scope}[shift={(+3,-1)}]
\node[circle,fill=black,inner sep=0.8pt,draw] (a) at (0,1) {};
\node[circle,fill=black,inner sep=1.5pt,draw] (c) at (.5,.5) {};
\node[circle,fill=black,inner sep=0.8pt,draw] (d) at (.5,.2) {};

\node () at (.5,.75) {\tiny $2$};

\draw (a) edge[dashed] (c);
\draw (c) -- (d);
\draw (.5,0) circle (0.2);
\end{scope}

\begin{scope}[shift={(+6,0)}]
\node[circle,fill=black,inner sep=1.5pt,draw] (a) at (0,0) {};
\node[circle,fill=black,inner sep=1.5pt,draw] (b) at (1,0) {};
\node[circle,fill=black,inner sep=1.5pt,draw] (c) at (1,-1) {};
\node[circle,fill=black,inner sep=0.8pt,draw] (d) at (1.5,-1) {};

\draw (a) edge[me=2] (b);
\draw (d) edge[dashed] (b);
\draw (b)--(c);
\end{scope}

\begin{scope}[shift={(+9,0)}]
\node[circle,fill=black,inner sep=1.5pt,draw] (a) at (0,-1) {};
\node[circle,fill=black,inner sep=1.5pt,draw] (b) at (1,-1) {};
\node[circle,fill=black,inner sep=0.8pt,draw] (c) at (0.5,0) {};
\node[circle,fill=black,inner sep=1.5pt,draw] (d) at (1.5,0) {};
\node[circle,fill=black,inner sep=0.8pt,draw] (e) at (2,-1) {};

\draw (a)--(b)--(c)--(a);
\draw (c) -- (d);
\draw (b) edge[dashed] (e);
\end{scope}
\end{tikzpicture}
\caption{\small{Edge contractions of $G_{v_0}$. Note $G_{v_0}$ consists only of the solid edges.}}
\label{fig:edge-contractions-of-second-row}
\end{figure}
\vspace{-0.1cm}
We can analogously deal with all subsequent edge contractions. Indeed, by inspection we verify that each $D_v$ from Figure~\ref{fig:configurations-for-genus-4} remains of rank at least $1$ under any number of edge contractions. Suppose $G$ admits an effective degree 3 divisor $D$ such that all $D_v$ are from the first row of Figure~\ref{fig:configurations-for-genus-4}. Arguing as above, we see that so does the graph~$G^{\varepsilon}$. Proceeding by induction, same holds after arbitrarily many contractions. If $G$ does not admit such divisor, then by degree consideration, it admits a degree three divisor of rank at least $1$ with \textit{precisely} one $D_v$ among the configurations on the second row of the same figure. Let $\varepsilon$ be the topological edge of shortest length within this particular $D_v$. If $\varepsilon$ is part of the cycle (resp. is a side of the triangle), then the closure $G\backslash G_{v_0}$ is unaffected after contracting along $\varepsilon$, thus remaining of rank at least $1$. Otherwise, we consider $D'_v$ obtained by setting $a=0$ as above. Note that $G^\varepsilon$ now admits $D$ with all $D_v$ among the first row. 
\end{proof}

All degenerations of graphs in Section~\ref{sec:straightforward-cases-genus4} admit a degree 3 divisor of rank at least $1$ according to Proposition~\ref{prop:contactions-of-genus4-config}. If $G$ belongs to any of the homeomorphic families from Sections~\ref{sec:straightforward-cases-genus4}, $G^\varepsilon$ satisfies the conditions of Lemma~\ref{prop:contactions-of-genus4-config}. Degenerations of the remaining homeomorphic families are considered independently. We note that the divisors presented below all satisfy the conditions of Proposition~\ref{prop:contactions-of-genus4-config}.

Let $G$ belong to the first family. If $\varepsilon$ is the middle (vertical) topological edge, then $G^\varepsilon$ has $(2,2)$-bridge decomposition. Otherwise, $G^\varepsilon$ is as shown below. For each case we present a degree 3 divisor of rank at least $1$.

\begin{figure}[H]
\begin{tikzpicture}
\node[circle,fill=black,inner sep=0.8pt,draw] (a) at (0,0) {};
\node[circle,fill=black,inner sep=1.5pt,draw] (b) at (0,1) {};
\node[circle,fill=black,inner sep=0.8pt,draw] (c) at (1,1) {};
\node[circle,fill=black,inner sep=0.8pt,draw] (d) at (1,0) {};
\node[circle,fill=black,inner sep=1.5pt,draw] (e) at (-1,0) {};
\node[circle,fill=black,inner sep=1.5pt,draw] (mid-ad) at (0.3,0) {};

\node () at (0.7,-0.2) {\tiny $a$};
\node () at (0.5,1.15) {\tiny $a$};

\draw (e) edge[bend right] (b);
\draw (e) edge[bend left] (b);
\draw (c) edge[me=2] (d);
\draw (c) --(b)--(a) -- (e);
\draw (a) -- (d);

\begin{scope}[shift={(+3.5,0)}]
\node[circle,fill=black,inner sep=1.5pt,draw] (a) at (0,0) {};
\node[circle,fill=black,inner sep=0.8pt,draw] (b) at (0,1) {};
\node[circle,fill=black,inner sep=0.8pt,draw] (c) at (1,1) {};
\node[circle,fill=black,inner sep=0.8pt,draw] (d) at (1,0) {};
\node[circle,fill=black,inner sep=0.8pt,draw] (e) at (-1,0) {};
\node[circle,fill=black,inner sep=1.5pt,draw] (mid-bc) at (0.4,1) {};
\node[circle,fill=black,inner sep=1.5pt,draw] (mid-ab) at (0,0.3) {};

\node () at (0.5,-0.2) {\tiny $a$};
\node () at (-0.5,-0.2) {\tiny $b$};
\node () at (0.75,1.15) {\tiny $a$};
\node () at (0.15,1.15) {\tiny $c$};
\node () at (0.15,0.65) {\tiny $z$};
\node () at (0,-.5) {\tiny $z:=\min (b,c)$};

\draw (e) edge[bend right] (b);
\draw (e) edge[bend left] (b);
\draw (c) edge[me=2] (d);
\draw (c) --(b)--(a) -- (e);
\draw (a) -- (d);
\end{scope}

\begin{scope}[shift={(+7,0)}]
\node[circle,fill=black,inner sep=1.5pt,draw] (a) at (0,0) {};
\node[circle,fill=black,inner sep=0.8pt,draw] (b) at (0,1) {};
\node[circle,fill=black,inner sep=0.8pt,draw] (c) at (1,1) {};
\node[circle,fill=black,inner sep=0.8pt,draw] (d) at (1,0) {};
\node[circle,fill=black,inner sep=0.8pt,draw] (e) at (-1,0) {};
\node[circle,fill=black,inner sep=1.5pt,draw] (mid-bc) at (0.4,1) {};
\node[circle,fill=black,inner sep=1.5pt,draw] (mid-ae) at (-0.3,0) {};

\node () at (0.5,-0.2) {\tiny $a$};
\node () at (0.75,1.15) {\tiny $a$};
\node () at (0.15,1.15) {\tiny $c$};
\node () at (0.15,0.5) {\tiny $d$};
\node () at (-0.65,-0.2) {\tiny $z$};
\node () at (0,-.5) {\tiny $z:=\min (c,d)$};

\draw (e) edge[bend right] (b);
\draw (e) edge[bend left] (b);
\draw (c) edge[me=2] (d);
\draw (c) --(b)--(a) -- (e);
\draw (a) -- (d);
\end{scope}
\end{tikzpicture}
\caption{\small{Edge contraction of the first family.}}
\label{fig:Edge-contraction-of-the-fourth-family}
\end{figure}
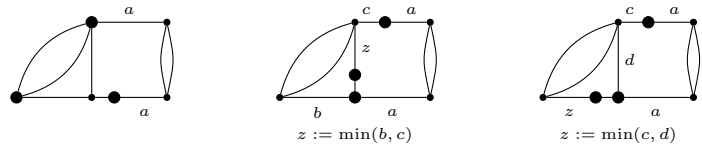

Let $G$ belong to the second family. If $\varepsilon$ participates in a cycle, then $G^\varepsilon$ has at least two loops and Lemma~\ref{lem:loop-lemma} applies. If $\varepsilon$ is part of the fourth configuration in Figure~\ref{fig:configurations-for-genus-4}, the arguments from Section~\ref{sec:fifth-family-genus4} are still valid. There is one remaining choice for $\varepsilon$ and $G^\varepsilon$ is shown in Figure~\ref{fig:Edge-contraction-of-the-fifth-family}. The divisor presented is of rank at least $1$.

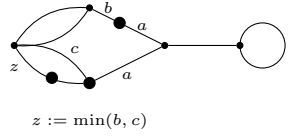
\begin{figure}[H]
\begin{tikzpicture}
\node[circle,fill=black,inner sep=1.5pt,draw] (a) at (0,0) {};
\node[circle,fill=black,inner sep=0.8pt,draw] (b) at (0,1) {};
\node[circle,fill=black,inner sep=0.8pt,draw] (c) at (1,.5) {};
\node[circle,fill=black,inner sep=0.8pt,draw] (d) at (2,0.5) {};
\node[circle,fill=black,inner sep=0.8pt,draw] (e) at (-1,0.5) {};
\node[circle,fill=black,inner sep=1.5pt,draw] (mid-bc) at (0.4,.8) {};
\node[circle,fill=black,inner sep=1.5pt,draw] (mid-ae) at (-0.5,0.07) {};

\node () at (0.7,0.75) {\tiny$a$};
\node () at (0.5,0.1) {\tiny$a$};
\node () at (-1,0.2) {\tiny$z$};
\node () at (-.2,0.45) {\tiny$c$};
\node () at (0.25,1) {\tiny$b$};
\node () at (0,-.5) {\tiny $z:=\min (b,c)$};

\draw (2.30,0.5) circle (0.30);
\draw (e) edge[bend right] (b);
\draw (e) edge[bend left] (b);
\draw (e) edge[bend right] (a);
\draw (e) edge[bend left] (a);
\draw (d)--(c)--(a);
\draw (c) -- (b);
\end{tikzpicture}
\caption{\small{Edge contraction of the second family.}}
\label{fig:Edge-contraction-of-the-fifth-family}
\end{figure}

Let $G$ belong to the third family. Arguing as in the previous case, there is a single possibility for $\varepsilon$ that needs to be examined. The contracted graph $G^\varepsilon$ is shown in Figure~\ref{fig:Edge-contraction-for-the-sixth-family}. For each case the divisor presented is of rank at least $1$.

\begin{figure}[H]
\begin{tikzpicture}
\node[circle,fill=black,inner sep=1.5pt,draw] (a) at (0,0) {};
\node[circle,fill=black,inner sep=0.8pt,draw] (b) at (0,1) {};
\node[circle,fill=black,inner sep=0.8pt,draw] (c) at (1,.5) {};
\node[circle,fill=black,inner sep=0.8pt,draw] (d) at (1.5,0.5) {};
\node[circle,fill=black,inner sep=1.5pt,draw] (e) at (-1,1) {};
\node[circle,fill=black,inner sep=1.5pt,draw] (mid-bc) at (0.4,.8) {};

\node () at (0.7,0.75) {\tiny$a$};
\node () at (0.5,0.1) {\tiny$a$};

\draw (1.80,0.5) circle (0.30);
\draw (e) --(b) -- (a) -- (c) -- (d);
\draw (b) -- (c);
\draw (a) edge[me=2] (e);

\begin{scope}[shift={(+4,0)}]
\node[circle,fill=black,inner sep=0.8pt,draw] (a) at (0,0) {};
\node[circle,fill=black,inner sep=1.5pt,draw] (b) at (0,1) {};
\node[circle,fill=black,inner sep=0.8pt,draw] (c) at (1,.5) {};
\node[circle,fill=black,inner sep=0.8pt,draw] (d) at (1.5,0.5) {};
\node[circle,fill=black,inner sep=0.8pt,draw] (e) at (-1,1) {};
\node[circle,fill=black,inner sep=1.5pt,draw] (mid-ac) at (.3,0.15) {};
\node[circle,fill=black,inner sep=1.5pt,draw] (mid-ab) at (0,.6) {};

\node () at (0.5,0.85) {\tiny$a$};
\node () at (0.7,0.2) {\tiny$a$};
\node () at (0.2,-0.05) {\tiny$b$};
\node () at (-0.5,1.15) {\tiny$c$};
\node () at (0.1,0.3) {\tiny$z$};
\node () at (0,-.5) {\tiny $z:=\min (b,c)$};

\draw (1.80,0.5) circle (0.30);
\draw (e) --(b) -- (a) -- (c) -- (d);
\draw (b) -- (c);
\draw (a) edge[me=2] (e);
\end{scope}

\begin{scope}[shift={(+8,0)}]
\node[circle,fill=black,inner sep=0.8pt,draw] (a) at (0,0) {};
\node[circle,fill=black,inner sep=1.5pt,draw] (b) at (0,1) {};
\node[circle,fill=black,inner sep=0.8pt,draw] (c) at (1,.5) {};
\node[circle,fill=black,inner sep=0.8pt,draw] (d) at (1.5,0.5) {};
\node[circle,fill=black,inner sep=0.8pt,draw] (e) at (-1,1) {};
\node[circle,fill=black,inner sep=1.5pt,draw] (mid-ac) at (.3,0.15) {};
\node[circle,fill=black,inner sep=1.5pt,draw] (mid-be) at (-0.4,1) {};

\node () at (0.5,0.85) {\tiny$a$};
\node () at (0.7,0.2) {\tiny$a$};
\node () at (0.2,-0.05) {\tiny$b$};
\node () at (-0.7,1.15) {\tiny$z$};
\node () at (0.15,0.5) {\tiny$d$};
\node () at (0,-.5) {\tiny $z:=\min (b,d)$};

\draw (1.80,0.5) circle (0.30);
\draw (e) --(b) -- (a) -- (c) -- (d);
\draw (b) -- (c);
\draw (a) edge[me=2] (e);
\end{scope}
\end{tikzpicture}
\caption{\small{Edge contraction of the third family.}}
\label{fig:Edge-contraction-for-the-sixth-family}
\end{figure}
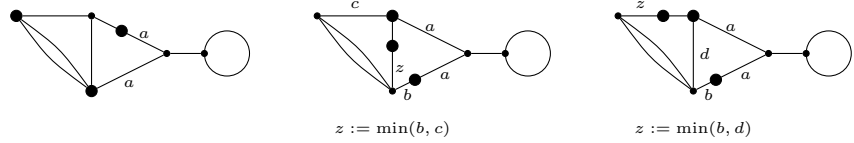
Finally, let $G$ be a loop of loops, i.e. the fourth family. There are two possibilities for $\varepsilon$ -- either it participates in a loop, or it connects two loops. Both cases are considered in Figure~\ref{fig:Edge-contraction-for-loop-of-loops}. The last chip is placed at distance $\min (x,b+c)$ from the gray vertex in the first case and at distance $\min(a,b+c)$ -- in the second. Both divisors depicted have rank at least $1$, which can be verified using Dhar's burning algorithm.

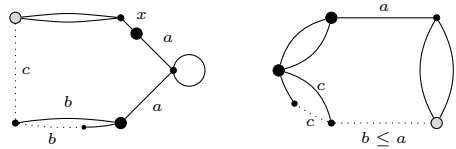
\begin{figure}[H]
\begin{tikzpicture}[scale=0.7]
\node[circle,fill=black,inner sep=0.8pt,draw] (a10) at (-1,-1) {};
\node[circle,fill=gray!30,inner sep=1.5pt,draw] (b10) at (-1,1) {};
\node[circle,fill=black,inner sep=1.5pt,draw] (c10) at (1,-1) {};
\node[circle,fill=black,inner sep=0.8pt,draw] (d10) at (1,1) {};
\node[circle,fill=black,inner sep=0.8pt,draw] (e10) at (2,0) {};

\node[circle,fill=black,inner sep=1.5pt,draw] (mid-outside) at (1.3,0.7) {};
\node[circle,fill=black,inner sep=.5pt,draw] (mid-inside) at (0.3, -1.08) {};

\draw (2.3,0) circle (0.3);
\node () at (1.9,0.6) {\tiny$a$};
\node () at (-0.8,0) {\tiny$c$};
\node () at (1.7,-0.7) {\tiny$a$};
\node () at (1.4, 1.04) {\tiny$x$};
\node () at (0.0, -0.6) {\tiny$b$};
\node () at (-0.3, -1.3) {\tiny$b$};

\draw (a10) edge[dotted] (b10); 
\draw (c10) edge[bend right = 8] (a10);
\draw (a10) edge[dotted, bend right =5] (mid-inside);
\draw (mid-inside) edge[bend right=5] (c10);
\draw (b10) edge[me=2] (d10);
\draw (d10) edge (e10);
\draw (c10) edge (e10);

\begin{scope}[shift={(+5,-1)}]
\node[circle,fill=black,inner sep=0.8pt,draw] (a) at (0,0) {};
\node[circle,fill=gray!30,inner sep=1.5pt,draw] (b) at (2,0) {};
\node[circle,fill=black,inner sep=0.8pt,draw] (c) at (2,2) {};
\node[circle,fill=black,inner sep=1.5pt,draw] (d) at (0,2) {};
\node[circle,fill=black,inner sep=1.5pt,draw] (e) at (-1,1) {};
\node[circle,fill=black,inner sep=0.7pt,draw] (new-mid) at (-.7,0.37) {};

\node () at (-0.2,.7) {\tiny $c$};
\node () at (-0.4,0) {\tiny $c$};
\node () at (1,-0.3) {\tiny $b\leq a$};
\node () at (1,2.2) {\tiny $a$};

\draw (a) edge[dotted] (b); \draw (c)--(d);
\draw (a) edge[bend right] (e);

\draw (a) edge[dotted,bend left=7] (new-mid);
\draw (new-mid) edge [bend left = 7] (e);

\draw (d) edge[bend right] (e);
\draw (d) edge[bend left] (e);

\draw (b) edge[bend left] (c);
\draw (b) edge[bend right] (c);
\end{scope}
\end{tikzpicture}
\caption{\small{Possible $G^\varepsilon$ for loop of loops in genus $4$.}}
\label{fig:Edge-contraction-for-loop-of-loops}
\end{figure}

The graph on the left in Figure~\ref{fig:Edge-contraction-for-loop-of-loops} can also be obtained as a degeneration from the second family, and the one on the right -- from the last family in Section~\ref{sec:straightforward-cases-genus4}. Therefore, the gonality conjecture holds for all of degenerations of the loop of loops. We have exhausted all graphs of genus $4$ and thus confirmed the gonality conjecture in genus $4$. Combined with Proposition~\ref{prop:gonality-implies-b.n.existence-for-genus-4} we deduce the following.

\begin{theorem}
The Brill--Noether existence conjecture holds for all graphs of genus $4$.
\end{theorem}


\section{Graphs of genus 5}
\label{sec:B.N-for-genus-5}
In this section we prove Brill--Noether existence for graphs of genus $5$. In light of Corollary~\ref{cor:B.N-reduces-to-Gonality-for-g=4,5}, it suffices to construct a divisor $D$ of degree $4$ and $\rk(D)\geq 1$ for every genus 5 graph~$G$. As in genus~$4$, we begin with topologically trivalent graphs and extend the constructions to general graphs of genus $5$ via edge contractions.

\subsection{Topologically trivalent graphs}
By applying Lemmas~\ref{lem:bridge-lemma} and~\ref{lem:loop-lemma}, we study only topologically trivalent graphs of genus 5 which have no bridges or loops. The graphs are depicted in Figure~\ref{fig:all-genus5-graphs} below.

\begin{figure}[H]

\begin{tikzpicture}
\node[circle,fill=black,inner sep=0.8pt,draw] (a6) at (-.43,-.5) {};
\node[circle,fill=black,inner sep=0.8pt,draw] (b6) at (-.43,.5) {};
\node[circle,fill=black,inner sep=0.8pt,draw] (c6) at (-0.07,-0.5) {};
\node[circle,fill=black,inner sep=0.8pt,draw] (d6) at (-0.07,0.5) {};
\node[circle,fill=black,inner sep=0.8pt,draw] (x) at (0.2,-0.5) {};
\node[circle,fill=black,inner sep=0.8pt,draw] (y) at (0.2,0.5) {};
\node[circle,fill=black,inner sep=0.8pt,draw] (e6) at (0.57,-0.5) {};
\node[circle,fill=black,inner sep=0.8pt,draw] (f6) at (0.57,0.5) {};

\draw (a6) edge[me=2] (b6); 
\draw (c6) edge (a6);
\draw (b6) edge (d6);
\draw (c6) edge (d6);
\draw (d6) edge (f6);
\draw (c6) edge (e6);
\draw (x) edge (y);
\draw (e6) edge[me=2] (f6);
\begin{scope}[shift={(+2.75,0)}]
\node[circle,fill=black,inner sep=0.8pt,draw] (a9) at (-1,0) {};
\node[circle,fill=black,inner sep=0.8pt,draw] (b9) at (-0.5, 0.5) {};
\node[circle,fill=black,inner sep=0.8pt,draw] (c9) at (0.5, 0.5) {};
\node[circle,fill=black,inner sep=0.8pt,draw] (b9a) at (-0.5, 1) {};
\node[circle,fill=black,inner sep=0.8pt,draw] (c9a) at (0.5, 1) {};
\node[circle,fill=black,inner sep=0.8pt,draw] (d9) at (1,0) {};
\node[circle,fill=black,inner sep=0.8pt,draw] (e9) at (0.5, -0.5) {};
\node[circle,fill=black,inner sep=0.8pt,draw] (f9) at (-0.5, -0.5) {};

\draw (a9)--(b9);\path (b9) edge (c9);
\draw (c9)--(d9);\path (d9) edge [bend left] (e9); \path (d9) edge [bend right] (e9);
\draw (f9)--(e9);\path (f9) edge [bend left] (a9); \path (a9) edge [bend left] (f9);
\draw (b9) edge (b9a);
\draw (c9) edge (c9a);
\draw (c9a) edge[me=2] (b9a);
\end{scope}

\begin{scope}[shift={(+4.8,0)}]
\node[circle,fill=black,inner sep=0.8pt,draw] (41) at (0,0) {};
\node[circle,fill=black,inner sep=0.8pt,draw] (42) at (1,0) {};
\node[circle,fill=black,inner sep=0.8pt,draw] (43) at (0.5,0.43) {};
\node[circle,fill=black,inner sep=0.8pt,draw] (45) at (0.5,-0.43) {};
\node[circle,fill=black,inner sep=0.8pt,draw] (A3) at (1.85,0.43) {};
\node[circle,fill=black,inner sep=0.8pt,draw] (A5) at (1.85,-0.43) {};
\node[circle,fill=black,inner sep=0.8pt,draw] (B3) at (1.25,.43) {};
\node[circle,fill=black,inner sep=0.8pt,draw] (B5) at (1.25,-.43) {};

\draw (41) -- (43) -- (42); \draw (41)--(45)--(42); \draw (41)--(42); 
\draw (A3) edge (43); \draw (A5) edge (45);
\draw (A3) edge[me=2] (A5);
\draw (B3) -- (B5);
\end{scope}

\begin{scope}[shift={(+8.3,-.5)},scale=0.8]
\node[circle,fill=black,inner sep=0.8pt,draw] (A1) at (0,0) {};
\node[circle,fill=black,inner sep=0.8pt,draw] (A2) at (-.5,0.5) {};
\node[circle,fill=black,inner sep=0.8pt,draw] (A8) at (1,0) {};
\node[circle,fill=black,inner sep=0.8pt,draw] (A7) at (1.5,0.5) {};

\node[circle,fill=black,inner sep=0.8pt,draw] (A3) at (-.5,1.25) {};
\node[circle,fill=black,inner sep=0.8pt,draw] (A4) at (0,1.75) {};
\node[circle,fill=black,inner sep=0.8pt,draw] (A5) at (1,1.75) {};
\node[circle,fill=black,inner sep=0.8pt,draw] (A6) at (1.5,1.25) {};

\draw (A1) edge[me=2] (A2);
\draw (A2)--(A3);
\draw (A3) edge[me=2] (A4);
\draw (A4)--(A5);
\draw (A5) edge[me=2] (A6);
\draw (A6)--(A7);
\draw (A7) edge[me=2] (A8);
\draw (A8)--(A1);
\end{scope}


\begin{scope}[shift={(10.75,0)}]
\node[circle,fill=black,inner sep=0.8pt,draw] (a6) at (-.43,-.5) {};
\node[circle,fill=black,inner sep=0.8pt,draw] (b6) at (-.43,.5) {};
\node[circle,fill=black,inner sep=0.8pt,draw] (c6) at (-0.07,-0.5) {};
\node[circle,fill=black,inner sep=0.8pt,draw] (d6) at (-0.07,0.5) {};
\node[circle,fill=black,inner sep=0.8pt,draw] (x) at (0.2,-0.5) {};
\node[circle,fill=black,inner sep=0.8pt,draw] (y) at (0.2,0.5) {};
\node[circle,fill=black,inner sep=0.8pt,draw] (e6) at (0.57,-0.5) {};
\node[circle,fill=black,inner sep=0.8pt,draw] (f6) at (0.57,0.5) {};

\draw (a6) edge[me=2] (b6); 
\draw (c6) edge (a6);
\draw (b6) edge (d6);
\draw (c6) edge (y);
\draw (d6) edge (f6);
\draw (c6) edge (e6);
\draw (x) edge (d6);
\draw (e6) edge[me=2] (f6);

\end{scope}

\begin{scope}[shift={(+12.75,0.3)},scale=0.8]
\node[circle,fill=black,inner sep=0.8pt,draw] (a6) at (-.43,-1) {};
\node[circle,fill=black,inner sep=0.8pt,draw] (b6) at (-.43,.5) {};
\node[circle,fill=black,inner sep=0.8pt,draw] (c6) at (.07,-1) {};
\node[circle,fill=black,inner sep=0.8pt,draw] (d6) at (.07,0.5) {};

\node[circle,fill=black,inner sep=0.8pt,draw] (e6) at (0.57,-1) {};
\node[circle,fill=black,inner sep=0.8pt,draw] (f6) at (0.57,0.5) {};

\node[circle,fill=black,inner sep=0.8pt,draw] (q6) at (.07,0.125) {};
\node[circle,fill=black,inner sep=0.8pt,draw] (p6) at (.07,-0.675) {};

\draw (a6) edge[me=2] (b6); 
\draw (c6) edge (a6);
\draw (b6) edge (d6);
\draw (c6) edge (p6);
\draw (p6) edge[me=2] (q6);
\draw (q6)--(d6);
\draw (d6) edge (f6);
\draw (c6) edge (e6);
\draw (e6) edge[me=2] (f6);
\end{scope}  

\begin{scope}[shift={(-0.5,-2)}]
\node[circle,fill=black,inner sep=0.8pt,draw] (41) at (0,0) {};
\node[circle,fill=black,inner sep=0.8pt,draw] (42) at (1,0) {};
\node[circle,fill=black,inner sep=0.8pt,draw] (43) at (0.5,0.43) {};
\node[circle,fill=black,inner sep=0.8pt,draw] (45) at (0.5,-0.43) {};
\node[circle,fill=black,inner sep=0.8pt,draw] (B1) at (-0.5,0) {};
\node[circle,fill=black,inner sep=0.8pt,draw] (B2) at (1.5,0) {};
\node[circle,fill=black,inner sep=0.8pt,draw] (B3) at (-0.5,0.75) {};
\node[circle,fill=black,inner sep=0.8pt,draw] (B4) at (1.5,0.75) {};

\draw (41) -- (43) -- (42); \draw (41)--(45)--(42); \draw (43)--(45); 
\draw (41)--(B1); \draw (42)--(B2); \draw (B3)--(B4);
\draw (B1) edge[me=2] (B3);
\draw (B2) edge[me=2] (B4);
\end{scope}

\begin{scope}[shift={(+2.55,-1.75)}]
\node[circle,fill=black,inner sep=0.8pt,draw] (a6) at (-.43,-.5) {};
\node[circle,fill=black,inner sep=0.8pt,draw] (b6) at (-.43,.5) {};
\node[circle,fill=black,inner sep=0.8pt,draw] (c6) at (.25,0) {};
\node[circle,fill=black,inner sep=0.8pt,draw] (j6) at (.25,-.5) {};
\node[circle,fill=black,inner sep=0.8pt,draw] (d6) at (.07,0.5) {};
\node[circle,fill=black,inner sep=0.8pt,draw] (h6) at (.5,0.5) {};
\node[circle,fill=black,inner sep=0.8pt,draw] (e6) at (1,-0.5) {};
\node[circle,fill=black,inner sep=0.8pt,draw] (f6) at (1,0.5) {};

\draw (a6) edge[me=2] (b6); 
\draw (e6) edge (a6);
\draw (b6) edge (d6);
\draw (c6) edge (d6);
\draw (d6) edge (f6);
\draw (h6) edge (c6);
\draw (j6)--(c6);
\draw (e6) edge[me=2] (f6);
\end{scope}  

\begin{scope}[shift={(+5.5,-2.3)}]
\node[circle,fill=black,inner sep=0.8pt,draw] (31) at (0.25,0.43) {};
\node[circle,fill=black,inner sep=0.8pt,draw] (32) at (0,0) {};
\node[circle,fill=black,inner sep=0.8pt,draw] (33) at (0.5,0) {};
\node[circle,fill=black,inner sep=0.8pt,draw] (34) at (0.25,1.4) {};
\node[circle,fill=black,inner sep=0.8pt,draw] (35) at (-.65,-.35) {};
\node[circle,fill=black,inner sep=0.8pt,draw] (36) at (1.15,-.35) {};
\node[circle,fill=black,inner sep=0.8pt,draw] (311) at (0.25,0.73) {};
\node[circle,fill=black,inner sep=0.8pt,draw] (312) at (0.25,1.1) {};

\draw (31) -- (33) -- (32) -- (31); \draw (33) -- (36); \draw (32)--(35); \draw (34) -- (35) -- (36) -- (34);
\draw (31)--(311);\draw (34)--(312);
\draw (311) edge[me=2] (312);
\end{scope}

\begin{scope}[shift={(8.5,-2.3)}]
\node[circle,fill=black,inner sep=0.8pt,draw] (31) at (0.25,0.43) {};
\node[circle,fill=black,inner sep=0.8pt,draw] (32) at (0,0) {};
\node[circle,fill=black,inner sep=0.8pt,draw] (33) at (0.5,0) {};
\node[circle,fill=black,inner sep=0.8pt,draw] (34) at (0.25,1.4) {};
\node[circle,fill=black,inner sep=0.8pt,draw] (35) at (-.65,-.35) {};
\node[circle,fill=black,inner sep=0.8pt,draw] (36) at (1.15,-.35) {};
\node[circle,fill=black,inner sep=0.8pt,draw] (37) at (0,-0.35) {};
\node[circle,fill=black,inner sep=0.8pt,draw] (38) at (0.5,-.35) {};

\draw (31) -- (33) -- (32) -- (31); \draw (33) -- (36); \draw (32)--(35); \draw (34) -- (35) --(37) edge[me=2] (38);\draw (36) -- (34); \draw (31)--(34); \draw (36)--(38);
\end{scope}

\begin{scope}[shift={(10.7,-2.7)},scale=0.95]

\node[circle,fill=black,inner sep=0.8pt,draw] (S1) at (0,0) {};
\node[circle,fill=black,inner sep=0.8pt,draw] (S2) at (0,2) {};
\node[circle,fill=black,inner sep=0.8pt,draw] (S4) at (2,0) {};
\node[circle,fill=black,inner sep=0.8pt,draw] (S3) at (2,2) {};

\draw (S1)--(S2)--(S3)--(S4)--(S1);

\node[circle,fill=black,inner sep=0.8pt,draw] (R1) at (.5,.5) {};
\node[circle,fill=black,inner sep=0.8pt,draw] (R2) at (0.5,1.5) {};
\node[circle,fill=black,inner sep=0.8pt,draw] (R4) at (1.5,0.5) {};
\node[circle,fill=black,inner sep=0.8pt,draw] (R3) at (1.5,1.5) {};

\draw (R1)--(R2)--(R3)--(R4)--(R1);
\draw (R1)--(S1);
\draw (R2)--(S2);
\draw (R3)--(S3);
\draw (R4)--(S4);

\end{scope}  

\begin{scope}[shift={(-0.5,-5)},scale=0.9]

\node[circle,fill=black,inner sep=0.8pt,draw] (Q41) at (0.5,-.2) {};
\node[circle,fill=black,inner sep=0.8pt,draw] (Q42) at (1.5,-.2) {};
\node[circle,fill=black,inner sep=0.8pt,draw] (Q43) at (1,0.43) {};
\node[circle,fill=black,inner sep=0.8pt,draw] (Q44) at (1.25,0.23) {};
\node[circle,fill=black,inner sep=0.8pt,draw] (Q45) at (0.75,0.23) {};

\draw (Q41)--(Q45)--(Q43)--(Q44)--(Q42);
\draw (Q41)--(Q44); \draw (Q42)--(Q45);

\node[circle,fill=black,inner sep=0.8pt,draw] (U41) at (0.5,2) {};
\node[circle,fill=black,inner sep=0.8pt,draw] (U42) at (1.5,2) {};\node[circle,fill=black,inner sep=0.8pt,draw] (U45) at (1,1.57) {};

\draw (U41)--(U45)--(U42)--(U41); 
\draw(U41)--(Q41); \draw(U42)--(Q42); \draw (U45)--(Q43);
\draw (Q41)--(Q41);
\end{scope}  
\begin{scope}[shift={(1.8,-5)},scale=0.85]

\node[circle,fill=black,inner sep=0.8pt,draw] (S41) at (0.5,0) {};
\node[circle,fill=black,inner sep=0.8pt,draw] (S42) at (1.5,0) {};
\node[circle,fill=black,inner sep=0.8pt,draw] (S43) at (1,0.43) {};
\node[circle,fill=black,inner sep=0.8pt,draw] (S45) at (1,-0.43) {};
\draw (S41) -- (S43) -- (S42); \draw (S41)--(S45)--(S42); \draw (S43)--(S45); 

\node[circle,fill=black,inner sep=0.8pt,draw] (R41) at (0.5,2) {};
\node[circle,fill=black,inner sep=0.8pt,draw] (R42) at (1.5,2) {};
\node[circle,fill=black,inner sep=0.8pt,draw] (R43) at (1,2.43) {};
\node[circle,fill=black,inner sep=0.8pt,draw] (R45) at (1,1.57) {};

\draw (R41) -- (R43) -- (R42); \draw (R41)--(R45)--(R42); \draw (R43)--(R45); 

\draw (S41)--(R41); \draw (S42)--(R42);

\end{scope}  

\begin{scope}[shift={(+7,-4)},scale=1.2]
\node[circle,fill=black,inner sep=0.8pt,draw] (a) at (-2,-.25) {};
\node[circle,fill=black,inner sep=0.8pt,draw] (b) at (0,-.25) {};
\node[circle,fill=black,inner sep=0.8pt,draw] (c) at (-1.33,0.3) {};
\node[circle,fill=black,inner sep=0.8pt,draw] (d) at (-.66,0.3) {};
\node[circle,fill=black,inner sep=0.8pt,draw] (e) at (-1.33,-.8) {};
\node[circle,fill=black,inner sep=0.8pt,draw] (f) at (-.66,-.8) {};
\node[circle,fill=black,inner sep=0.8pt,draw] (g) at (-0.75,-.25) {};
\node[circle,fill=black,inner sep=0.8pt,draw] (h) at (-0.3,-.25) {};

\draw (a)--(g);\draw (h)--(b);\draw (h) edge[me=2] (g);
\draw (a)--(c)--(d)--(b)--(f)--(e)--(a);
\draw (c)--(f);\draw (d)--(e);

\end{scope}  

\begin{scope}[shift={(+10,-3.85)}]
\node[circle,fill=black,inner sep=0.8pt,draw] (O8) at (-2,-.25) {};
\node[circle,fill=black,inner sep=0.8pt,draw] (O3) at (0,-.25) {};
\node[circle,fill=black,inner sep=0.8pt,draw] (O1) at (-1.33,0.3) {};
\node[circle,fill=black,inner sep=0.8pt,draw] (O2) at (-.66,0.3) {};
\node[circle,fill=black,inner sep=0.8pt,draw] (O6) at (-1.33,-1.25) {};
\node[circle,fill=black,inner sep=0.8pt,draw] (O5) at (-.66,-1.25) {};
\node[circle,fill=black,inner sep=0.8pt,draw] (O7) at (-2,-.7) {};
\node[circle,fill=black,inner sep=0.8pt,draw] (O4) at (0,-.7) {};

\draw (O1)--(O2)--(O3)--(O4)--(O5)--(O6)--(O7)--(O8)--(O1);
\draw (O1)--(O5); \draw (O2)--(O6); \draw (O3)--(O7); \draw (O8)--(O4);

\end{scope}  

\begin{scope}[shift={(10.75,-5)},scale=0.9]

\node[circle,fill=black,inner sep=0.8pt,draw] (Y41) at (0.5,-.2) {};
\node[circle,fill=black,inner sep=0.8pt,draw] (Y42) at (1.5,-.2) {};
\node[circle,fill=black,inner sep=0.8pt,draw] (Y43) at (1,0.43) {};
\node[circle,fill=black,inner sep=0.8pt,draw] (Y44) at (1.25,0.23) {};
\node[circle,fill=black,inner sep=0.8pt,draw] (Y45) at (0.75,0.23) {};

\draw (Y41)--(Y45)--(Y43)--(Y44)--(Y42);
\draw (Y41)--(Y42);\draw (Y44)--(Y45);

\node[circle,fill=black,inner sep=0.8pt,draw] (V41) at (0.5,2) {};
\node[circle,fill=black,inner sep=0.8pt,draw] (V42) at (1.5,2) {};
\node[circle,fill=black,inner sep=0.8pt,draw] (V45) at (1,1.57) {};

\draw (V41)--(V45)--(V42)--(V41); 
\draw(V41)--(Y41); \draw(V42)--(Y42); \draw (V45)--(Y43);
\draw (Y41)--(Y41);
\end{scope}  
\end{tikzpicture}

\caption{\small{The topological types of trivalent genus $5$ graphs with no bridges or loops.}}
\label{fig:all-genus5-graphs}
\end{figure}
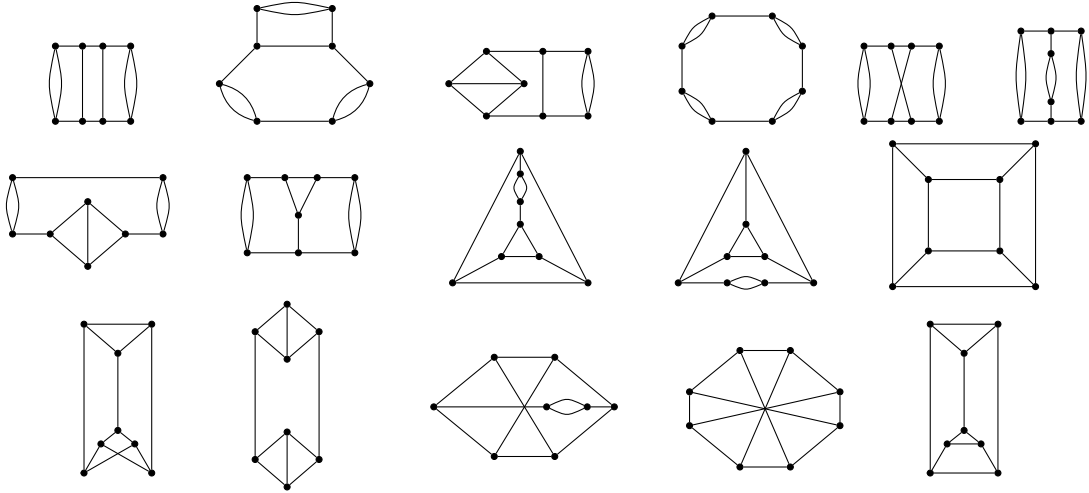

The following result, in the spirit of Lemma~\ref{lem:configurations-for-genus-4} and Proposition~\ref{prop:contactions-of-genus4-config}, not only produces degree~$4$ divisors of rank at least $1$ but also deals with multiple edge contractions.

\begin{proposition}
\label{prop:configuration-for-genus-5}
Let $G$ be a graph of genus 5 and let $D\in \Div_{+}(G)$ be of degree 4. For any $v\in V(G) \backslash \supp(D)$, let $G_v$ and $D_v$ be as defined in Lemma~\ref{fig:configurations-for-genus-4}. If all $G_v$ and $D_v$ are among the following

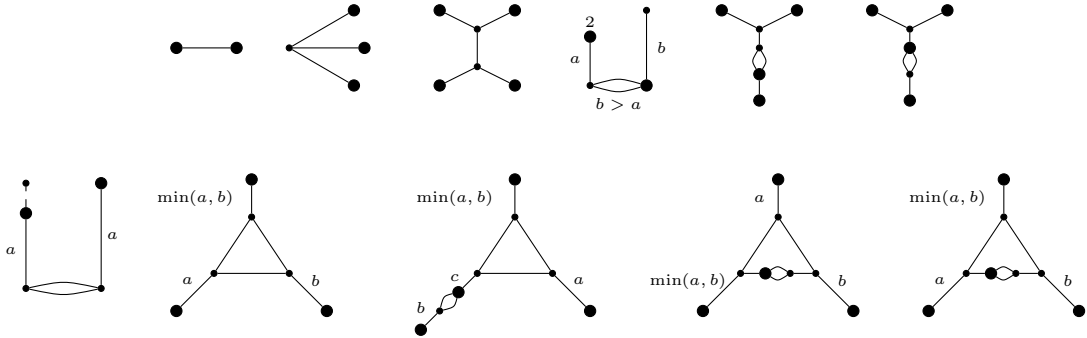
\begin{figure}[H]
\begin{tikzpicture}
\node[circle,fill=black,inner sep=1.5pt,draw] (a) at (0,0) {};
\node[circle,fill=black,inner sep=1.5pt,draw] (b) at (0.8,0) {};
\draw (a) edge (b);

\begin{scope}[shift={(+1.5,0)}]
\node[circle,fill=black,inner sep=0.8pt,draw] (a) at (0,0) {};
\node[circle,fill=black,inner sep=1.5pt,draw] (b) at (1,0) {};
\node[circle,fill=black,inner sep=1.5pt,draw] (c) at (0.86,-0.5) {};
\node[circle,fill=black,inner sep=1.5pt,draw] (d) at (0.86,0.5) {};

\draw (a) edge (b);
\draw (a) edge (c);
\draw (a) edge (d);
\end{scope}

\begin{scope}[shift={(4,0.5)}]
\node[circle,fill=black,inner sep=1.5pt,draw] (a) at (-0.5,0) {};
\node[circle,fill=black,inner sep=1.5pt,draw] (b) at (0.5,0) {};
\node[circle,fill=black,inner sep=0.8pt,draw] (c) at (0,-0.25) {};
\node[circle,fill=black,inner sep=0.8pt,draw] (d) at (0,-0.75) {};
\node[circle,fill=black,inner sep=1.5pt,draw] (e) at (-0.5,-1) {};
\node[circle,fill=black,inner sep=1.5pt,draw] (f) at (0.5,-1) {};

\draw (a)--(c)--(b);\draw (f)--(d)--(e); \draw (c)--(d);
\end{scope}

\begin{scope}[shift={(+5.5,-0.5)}]
\node[circle,fill=black,inner sep=0.8pt,draw] (a) at (0,0) {};
\node[circle,fill=black,inner sep=1.5pt,draw] (b) at (0.75,0) {};
\node[circle,fill=black,inner sep=1.5pt,draw] (c) at (0,0.65) {};
\node[circle,fill=black,inner sep=0.8pt,draw] (d) at (0.75,1) {};

\node () at (-.2,0.35) {\tiny$a$};
\node () at (0.95,0.5) {\tiny$b$};
\node () at (0,.85) {\tiny$2$};
\node () at (0.375,-0.25) {\tiny$b>a$};

\draw (a) edge[me=2] (b);
\draw (a) edge (c);
\draw (b) edge (d);
\end{scope}

\begin{scope}[shift={(7.75,0.5)}]
\node[circle,fill=black,inner sep=1.5pt,draw] (a) at (-0.5,0) {};
\node[circle,fill=black,inner sep=1.5pt,draw] (b) at (0.5,0) {};
\node[circle,fill=black,inner sep=0.8pt,draw] (c) at (0,-0.25) {};
\node[circle,fill=black,inner sep=0.8pt,draw] (d) at (0,-0.5) {};
\node[circle,fill=black,inner sep=1.5pt,draw] (e) at (0,-0.85) {};
\node[circle,fill=black,inner sep=1.5pt,draw] (f) at (0,-1.2) {};

\draw (a)--(c)--(b); \draw (c)--(d); \draw (d) edge[me=2] (e); \draw (e)--(f);
\end{scope}

\begin{scope}[shift={(9.75,0.5)}]
\node[circle,fill=black,inner sep=1.5pt,draw] (a) at (-0.5,0) {};
\node[circle,fill=black,inner sep=1.5pt,draw] (b) at (0.5,0) {};
\node[circle,fill=black,inner sep=0.8pt,draw] (c) at (0,-0.25) {};
\node[circle,fill=black,inner sep=1.5pt,draw] (d) at (0,-0.5) {};
\node[circle,fill=black,inner sep=0.8pt,draw] (e) at (0,-0.85) {};
\node[circle,fill=black,inner sep=1.5pt,draw] (f) at (0,-1.2) {};

\draw (a)--(c)--(b); \draw (c)--(d); \draw (d) edge[me=2] (e); \draw (e)--(f);
\end{scope}


\begin{scope}[shift={(-2,-3.2)}]
\node[circle,fill=black,inner sep=0.8pt,draw] (a) at (0,0) {};
\node[circle,fill=black,inner sep=0.8pt,draw] (b) at (1,0) {};
\node[circle,fill=black,inner sep=1.5pt,draw] (c) at (0,1) {};
\node[circle,fill=black,inner sep=1.5pt,draw] (d) at (1,1.4) {};
\node[circle,fill=black,inner sep=0.8pt,draw] (e) at (0,1.4) {};

\node () at (-.2,0.5) {\tiny$a$};
\node () at (1.15,0.7) {\tiny$a$};

\draw (a) edge[me=2] (b);
\draw (a) edge (c);
\draw (b) edge (d);
\draw (c) edge[dashed] (e);
\end{scope}

\begin{scope}[scale=1, shift={(+0.5,-2.25)}]

\node[circle,fill=black,inner sep=0.8pt,draw] (a) at (0.5,0) {};
\node[circle,fill=black,inner sep=0.8pt,draw] (b) at (0,-0.75) {};
\node[circle,fill=black,inner sep=0.8pt,draw] (c) at (1,-0.75) {};

\node[circle,fill=black,inner sep=1.5pt,draw] (a1) at (0.5,0.5) {};
\node[circle,fill=black,inner sep=1.5pt,draw] (b1) at (-.5,-1.25) {};
\node[circle,fill=black,inner sep=1.5pt,draw] (c1) at (1.5,-1.25) {};

\draw (a)--(b)--(c)--(a);
\draw (a)--(a1);\draw (b)--(b1); \draw (c)--(c1);

\node () at (-.25,0.25) {\tiny$\min(a,b)$};
\node () at (-.35,-0.85) {\tiny$a$};
\node () at (1.35,-0.85) {\tiny$b$};

\end{scope}

\begin{scope}[scale=1, shift={(+4,-2.25)}]

\node[circle,fill=black,inner sep=0.8pt,draw] (a) at (0.5,0) {};
\node[circle,fill=black,inner sep=0.8pt,draw] (b) at (0,-0.75) {};
\node[circle,fill=black,inner sep=0.8pt,draw] (c) at (1,-0.75) {};
\node[circle,fill=black,inner sep=1.5pt,draw] (d) at (-0.25,-1) {};
\node[circle,fill=black,inner sep=0.8pt,draw] (e) at (-0.5,-1.25) {};

\node[circle,fill=black,inner sep=1.5pt,draw] (a1) at (0.5,0.5) {};
\node[circle,fill=black,inner sep=1.5pt,draw] (b1) at (-.75,-1.5) {};
\node[circle,fill=black,inner sep=1.5pt,draw] (c1) at (1.5,-1.25) {};

\draw (a)--(b)--(c)--(a);
\draw (a)--(a1);\draw (b)--(d); \draw (d) edge[me=2] (e);\draw (e)--(b1); \draw (c)--(c1);

\node () at (-.3,0.25) {\tiny$\min(a,b)$};
\node () at (-.3,-0.8) {\tiny$c$};
\node () at (-.75,-1.2) {\tiny$b$};
\node () at (1.35,-0.85) {\tiny$a$};

\end{scope}

\begin{scope}[scale=1, shift={(+7.5,-2.25)}]
\node[circle,fill=black,inner sep=0.8pt,draw] (a) at (0.5,0) {};
\node[circle,fill=black,inner sep=0.8pt,draw] (b) at (0,-0.75) {};
\node[circle,fill=black,inner sep=1.5pt,draw] (b') at (0.33,-0.75) {};
\node[circle,fill=black,inner sep=0.8pt,draw] (c') at (0.66,-0.75) {};
\node[circle,fill=black,inner sep=0.8pt,draw] (c) at (1,-0.75) {};

\node[circle,fill=black,inner sep=1.5pt,draw] (a1) at (0.5,0.5) {};
\node[circle,fill=black,inner sep=1.5pt,draw] (b1) at (-.5,-1.25) {};
\node[circle,fill=black,inner sep=1.5pt,draw] (c1) at (1.5,-1.25) {};

\draw (a)--(b); \draw (b') edge[me=2] (c'); \draw (c)--(a);
\draw (b)--(b');\draw (c)--(c');
\draw (a)--(a1);\draw (b)--(b1); \draw (c)--(c1);

\node () at (-.7,-0.85) {\tiny$\min(a,b)$};
\node () at (0.25,0.25) {\tiny$a$};
\node () at (1.35,-0.85) {\tiny$b$};

\end{scope}

\begin{scope}[scale=1, shift={(+10.5,-2.25)}]

\node[circle,fill=black,inner sep=0.8pt,draw] (a) at (0.5,0) {};
\node[circle,fill=black,inner sep=0.8pt,draw] (b) at (0,-0.75) {};
\node[circle,fill=black,inner sep=1.5pt,draw] (b') at (0.33,-0.75) {};
\node[circle,fill=black,inner sep=0.8pt,draw] (c') at (0.66,-0.75) {};
\node[circle,fill=black,inner sep=0.8pt,draw] (c) at (1,-0.75) {};

\node[circle,fill=black,inner sep=1.5pt,draw] (a1) at (0.5,0.5) {};
\node[circle,fill=black,inner sep=1.5pt,draw] (b1) at (-.5,-1.25) {};
\node[circle,fill=black,inner sep=1.5pt,draw] (c1) at (1.5,-1.25) {};

\draw (a)--(b); \draw (b') edge[me=2] (c'); \draw (c)--(a);
\draw (b)--(b');\draw (c)--(c');
\draw (a)--(a1);\draw (b)--(b1); \draw (c)--(c1);

\node () at (-.25,0.25) {\tiny$\min(a,b)$};
\node () at (-.35,-0.85) {\tiny$a$};
\node () at (1.35,-0.85) {\tiny$b$};

\node () at (0,-2) {\phantom{a}};

\end{scope}
\end{tikzpicture}

\caption{\small{Possible divisors $D_v$ on graphs $G_v$. Topological edges above are allowed to have arbitrary length, unless otherwise indicated.}}
\label{fig:configurations-for-genus-5}
\end{figure}

then 
\begin{enumerate}
\item $\rk(D) \geq 1,$ and
\item  for any set of topological edges $\varepsilon_1,\cdots, \varepsilon_k$ of $G$, the graph $G^{\varepsilon_1,\cdots, \varepsilon_k}$ admits a degree~4 divisor of rank at least $1$.

\end{enumerate}
\end{proposition}
\begin{proof}
 The proof of part (1) is similar to that of Lemma~\ref{lem:configurations-for-genus-4}, and of part (2) is similar to that of Proposition~\ref{prop:contactions-of-genus4-config}.
\end{proof}

We apply Proposition~\ref{prop:configuration-for-genus-5} to produce degree 4 divisors of rank at least $1$ for the remaining families. For each graph $G$, the subgraphs $G_v$ and their corresponding divisors $D_v$ as defined above will be drawn in different edge patterns (dotted, dashed, etc.). Only two families do not fall into the scope of Proposition~\ref{prop:configuration-for-genus-5} and for them we explicitly produce divisors of desired degree and rank. In the Section~\ref{sec:general-graphs-of-genus-4}, we deal with edge contractions performed on graphs from these two families.

 \subsubsection{Straightforward cases}
 Many of the topological types of genus 5 graphs from Figure~\ref{fig:all-genus5-graphs} admit, for all edge lengths, a degree~$4$ divisor $D$ satisfying Proposition~\ref{prop:configuration-for-genus-5}. These graphs and their divisors are depicted below.
 \begin{figure}[H]
\begin{tikzpicture}[scale = 1]
\begin{scope}[shift={(0,0)}]

\node[circle,fill=black,inner sep=1.5pt,draw] (S1) at (0,0) {};
\node[circle,fill=black,inner sep=0.8pt,draw] (S2) at (0,2) {};
\node[circle,fill=black,inner sep=0.8pt,draw] (S4) at (2,0) {};
\node[circle,fill=black,inner sep=1.5pt,draw] (S3) at (2,2) {};

\draw (S1)edge[densely dotted](S2);
\draw(S2)edge[densely dotted](S3);
\draw(S3)--(S4)--(S1);

\node[circle,fill=black,inner sep=0.8pt,draw] (R1) at (.5,.5) {};
\node[circle,fill=black,inner sep=1.5pt,draw] (R2) at (0.5,1.5) {};
\node[circle,fill=black,inner sep=1.5pt,draw] (R4) at (1.5,0.5) {};
\node[circle,fill=black,inner sep=0.8pt,draw] (R3) at (1.5,1.5) {};

\draw (R1)edge[densely dashed](S1);
\draw (R1)edge[densely dashed](R2);
\draw (R1)edge[densely dashed](R4);

\draw (R3)edge[loosely dotted](R2);
\draw (R3)edge[loosely dotted](S3);
\draw (R3)edge[loosely dotted](R4);

\draw (R2)edge[densely dotted](S2);
 \draw (R4)--(S4);

\end{scope}  

\begin{scope}[shift={(2.8,0)}]

\node[circle,fill=black,inner sep=0.8pt,draw] (S41) at (0.5,0.2) {};
\node[circle,fill=black,inner sep=1.5pt,draw] (S42) at (1.5,0.2) {};
\node[circle,fill=black,inner sep=0.8pt,draw] (S43) at (1,0.63) {};
\node[circle,fill=black,inner sep=0.8pt,draw] (S45) at (1,-0.23) {};

\node[circle,fill=black,inner sep=1.5pt,draw] (R41) at (0.5,1.8) {};
\node[circle,fill=black,inner sep=0.8pt,draw] (R42) at (1.5,1.8) {};
\node[circle,fill=black,inner sep=0.8pt,draw] (R43) at (1,2.23) {};
\node[circle,fill=black,inner sep=0.8pt,draw] (R45) at (1,1.37) {};

\node[circle,fill=black,inner sep=1.5pt,draw] (new-up) at (.73,2) {};
\node[circle,fill=black,inner sep=1.5pt,draw] (new-down) at (1.27,0) {};

\node () at (.35,0.8) {\tiny $a$};
\node () at (1.65,0.9) {\tiny $b$};
\node () at (1.3,-0.25) {\tiny $z_1$};
\node () at (.75,2.25) {\tiny $z_2$};
\node () at (.65,1.5) {\tiny $c$};
\node () at (1.35,0.45) {\tiny $d$};

\node () at (1.15,-0.55) {\tiny $z_1:=\min (a,d)$};
\node () at (1.15,-0.75) {\tiny $z_2:=\min (b,c)$};

\draw (R41) -- (new-up);

\draw (new-up) edge[dashed] (R43);
\draw (R45) edge[dashed] (R42);
\draw (S42) edge[dashed] (R42);
\draw (R43) edge[dashed] (R42);
\draw (R45) edge[dashed] (R42);

\draw (S41) edge[dotted] (R41);
\draw (S41) edge[dotted] (S45);
\draw (S41) edge[dotted] (S43);
\draw (S45) edge[dotted] (new-down);
\draw (S42) edge[dotted] (S43);

\draw (R41)edge[dashed](R45); 

\draw (R43)edge[dashed](R45); 
\draw (S42) -- (new-down);

\end{scope}  

\begin{scope}[shift={(5.25,0)}]

\node[circle,fill=black,inner sep=1.5pt,draw] (Q41) at (0.5,-.2) {};
\node[circle,fill=black,inner sep=1.5pt,draw] (Q42) at (1.5,-.2) {};
\node[circle,fill=black,inner sep=1.5pt,draw] (Q43) at (1,0.43) {};
\node[circle,fill=black,inner sep=0.8pt,draw] (Q44) at (1.35,0.17) {};
\node[circle,fill=black,inner sep=0.8pt,draw] (Q45) at (0.65,0.17) {};

\draw (Q41)edge[dashed](Q45);
\draw (Q45)edge[dashed](Q43);
\draw (Q43)edge[dotted](Q44);
\draw (Q44)edge[dotted](Q42);
\draw (Q41)edge[dotted](Q44); 
\draw (Q42)edge[dashed](Q45);

\node[circle,fill=black,inner sep=0.8pt,draw] (U41) at (0.5,2) {};
\node[circle,fill=black,inner sep=0.8pt,draw] (U42) at (1.5,2) {};\node[circle,fill=black,inner sep=1.5pt,draw] (U45) at (1,1.57) {};

\draw (U41)--(U45)--(U42)--(U41); 
\draw(U41)--(Q41); \draw(U42)--(Q42); \draw (U45)--(Q43);
\draw (Q41)--(Q41);
\end{scope}

\begin{scope}[shift={(7.5,0)}]

\node[circle,fill=black,inner sep=1.5pt,draw] (Y41) at (0.5,-.2) {};
\node[circle,fill=black,inner sep=1.5pt,draw] (Y42) at (1.5,-.2) {};
\node[circle,fill=black,inner sep=1.5pt,draw] (Y43) at (1,0.43) {};
\node[circle,fill=black,inner sep=0.8pt,draw] (Y44) at (1.3,0.17) {};
\node[circle,fill=black,inner sep=0.8pt,draw] (Y45) at (0.7,0.17) {};

\draw (Y41)edge[dashed](Y45);
\draw(Y45)edge[dashed](Y43);
\draw (Y43)edge[dashed](Y44);
\draw(Y44)edge[dashed](Y42);
\draw (Y41)edge[dotted](Y42);
\draw (Y44)edge[dashed](Y45);

\node[circle,fill=black,inner sep=0.8pt,draw] (V41) at (0.5,2) {};
\node[circle,fill=black,inner sep=0.8pt,draw] (V42) at (1.5,2) {};\node[circle,fill=black,inner sep=1.5pt,draw] (V45) at (1,1.57) {};

\draw (V41)--(V45)--(V42)--(V41); 
\draw(V41)--(Y41); \draw(V42)--(Y42); \draw (V45)edge[gray=!30, line width=1.2mm](Y43);
\draw (Y41)--(Y41);
\end{scope}  

\begin{scope}[shift={(+12.25,+1.5)},scale=1.2]
\node[circle,fill=black,inner sep=0.8pt,draw] (O8) at (-2,-.25) {};
\node[circle,fill=black,inner sep=1.5pt,draw] (O3) at (0,-.25) {};
\node[circle,fill=black,inner sep=1.5pt,draw] (O1) at (-1.33,0.3) {};
\node[circle,fill=black,inner sep=0.8pt,draw] (O2) at (-.66,0.3) {};
\node[circle,fill=black,inner sep=0.8pt,draw] (O6) at (-1.33,-1.25) {};
\node[circle,fill=black,inner sep=1.5pt,draw] (O5) at (-.66,-1.25) {};
\node[circle,fill=black,inner sep=1.5pt,draw] (O7) at (-2,-.7) {};
\node[circle,fill=black,inner sep=0.8pt,draw] (O4) at (0,-.7) {};

\draw (O1)--(O2) -- (O3);
\draw (O2)--(O6) -- (O7);
\draw (O5) -- (O6);

\draw (O7) edge[dotted] (O3);
\draw (O5) edge[loosely dashdotted] (O1);

\draw (O1) edge[dashed] (O8);
\draw (O7) edge[dashed] (O8);
\draw (O8) edge[dashed] (O4);
\draw (O3) edge[dashed] (O4);
\draw (O5) edge[dashed] (O4);
\end{scope}  
 
\begin{scope}[shift={(+1,-1.55)},scale = 1.35]
\node[circle,fill=black,inner sep=1.5pt,draw] (a) at (-2,-.25) {};
\node[circle,fill=black,inner sep=0.8pt,draw] (b) at (0,-.25) {};
\node[circle,fill=black,inner sep=0.8pt,draw] (c) at (-1.33,0.3) {};
\node[circle,fill=black,inner sep=1.5pt,draw] (d) at (-.66,0.3) {};
\node[circle,fill=black,inner sep=0.8pt,draw] (e) at (-1.33,-.8) {};
\node[circle,fill=black,inner sep=1.5pt,draw] (f) at (-.66,-.8) {};
\node[circle,fill=black,inner sep=0.8pt,draw] (g) at (-0.75,-.25) {};
\node[circle,fill=black,inner sep=1.5,draw] (h) at (-0.3,-.25) {};

\draw (a)--(g);\draw (h)--(b);\draw (h) edge[me=2] (g);

\draw (a)edge[dashed](c);
\draw (c)edge[dashed](f);
\draw (c) edge[dashed] (d);
\draw (d)--(b)--(f);

\draw (f) edge[dotted] (e);
\draw (a) edge[dotted] (e);
\draw (d) edge[dotted] (e);

\draw (f) -- (e);
\end{scope}
\begin{scope}[shift={(+2.5,-2.3)},scale=1.2]
\node[circle,fill=black,inner sep=0.8pt,draw] (41) at (0,0) {};
\node[circle,fill=black,inner sep=0.8pt,draw] (42) at (1,0) {};
\node[circle,fill=black,inner sep=1.5pt,draw] (43) at (0.5,0.43) {};
\node[circle,fill=black,inner sep=1.5pt,draw] (45) at (0.5,-0.43) {};
\node[circle,fill=black,inner sep=0.8pt,draw] (B1) at (-0.5,0) {};
\node[circle,fill=black,inner sep=0.8pt,draw] (B2) at (1.5,0) {};
\node[circle,fill=black,inner sep=1.5pt,draw] (B3) at (-0.5,0.75) {};
\node[circle,fill=black,inner sep=1.5pt,draw] (B4) at (1.5,0.75) {};

\draw (41) edge[dashed] (43);

\draw (41)edge[dashed](45);
\draw (43)edge(45); 
\draw (41)edge[dashed](B1); 
\draw (B3)edge[gray=!30, line width=1.2mm] (B4);
\draw (B1) edge[dashed,me=2] (B3);

\draw (B2) edge[dotted,me=2] (B4);
\draw (43) edge[dotted] (42); 
\draw (45) edge[dotted](42);
\draw (42) edge[dotted](B2); 
\end{scope}

\begin{scope}[shift={(+6.3,-1.7)},scale=1.1]
\node[circle,fill=black,inner sep=0.8pt,draw] (a6) at (-.63,-1) {};
\node[circle,fill=black,inner sep=0.8pt,draw] (b6) at (-.63,.5) {};
\node[circle,fill=black,inner sep=1.5pt,draw] (c6) at (.07,-1) {};
\node[circle,fill=black,inner sep=0.8pt,draw] (d6) at (.07,0.5) {};

\node[circle,fill=black,inner sep=0.8pt,draw] (e6) at (0.77,-1) {};
\node[circle,fill=black,inner sep=0.8pt,draw] (f6) at (0.77,0.5) {};

\node[circle,fill=black,inner sep=1.5pt,draw] (q6) at (.07,0.125) {};
\node[circle,fill=black,inner sep=0.8pt,draw] (p6) at (.07,-0.675) {};

\node[circle,fill=black,inner sep=1.5pt,draw] (s6) at (-.15,.5) {};
\node[circle,fill=black,inner sep=1.5pt,draw] (t6) at (.3,.5) {};

\draw (a6) edge[dashed,me=2] (b6); 
\draw (c6) edge[dashed] (a6);
\draw (b6) edge[dashed] (s6);

\draw (s6)--(d6);
\draw (c6) edge (p6);
\draw (p6) edge[me=2] (q6);
\draw (q6)--(d6)--(t6);

\draw (t6) edge[dotted] (f6);
\draw (c6) edge[dotted] (e6);
\draw (e6) edge[dotted,me=2] (f6);
\end{scope}  

\begin{scope}[shift={(+8,-2)},scale=1.2]
\node[circle,fill=black,inner sep=0.8pt,draw] (41) at (0,0) {};
\node[circle,fill=black,inner sep=0.8pt,draw] (42) at (1,0) {};
\node[circle,fill=black,inner sep=1.5pt,draw] (43) at (0.5,0.43) {};
\node[circle,fill=black,inner sep=0.8pt,draw] (45) at (0.5,-0.43) {};
\node[circle,fill=black,inner sep=0.8pt,draw] (A3) at (2.25,0.43) {};
\node[circle,fill=black,inner sep=0.8pt,draw] (A5) at (2.25,-0.43) {};
\node[circle,fill=black,inner sep=0.8pt,draw] (B3) at (1.5,0.43) {};
\node[circle,fill=black,inner sep=1.5pt,draw] (B5) at (1.5,-0.43) {};
\node[circle,fill=black,inner sep=1.5pt,draw] (new-up) at (1.75,0.43) {};
\node[circle,fill=black,inner sep=1.5pt,draw] (new-left) at (0.25,0.27) {};

\node () at (1.87,-0.6) {\tiny $a$};
\node () at (2,0.6) {\tiny $a$};
\node () at (0,0.15) {\tiny $z$};
\node () at (1,-0.55) {\tiny $b$};
\node () at (.9,0.2) {\tiny $c$};
\node () at (1.25,-0.9) {\tiny $z:=\min(b,c)$};

\draw (41) edge[loosely dotted] (45);
\draw (43) edge[loosely dotted] (42);
\draw (41) edge[loosely dotted] (45);
\draw (41) edge[loosely dotted] (42);
\draw (45) edge[loosely dotted] (B5);
 \draw (45)edge[loosely dotted](42); 
 \draw (41) edge[loosely dotted] (new-left);
 \draw (B5) edge[loosely dotted] (45);
 
 \draw (new-left) edge[thick] (43);
 
\draw (A3) edge[dashed] (new-up);
\draw (A5) edge[dashed] (B5);
\draw (A3) edge[me=2,dashed] (A5);

\draw (B3) edge[densely dotted] (B5);
\draw (new-up) edge[densely dotted] (43);
\end{scope} 

\begin{scope}[shift={(+12,-2.5)},scale=1.3]
\node[circle,fill=black,inner sep=1.5pt,draw] (31) at (0.25,0.43) {};
\node[circle,fill=black,inner sep=1.5pt,draw] (32) at (0,0) {};
\node[circle,fill=black,inner sep=0.8pt,draw] (33) at (0.5,0) {};
\node[circle,fill=black,inner sep=0.8pt,draw] (34) at (0.25,1.4) {};
\node[circle,fill=black,inner sep=0.8pt,draw] (35) at (-.65,-.35) {};
\node[circle,fill=black,inner sep=0.8pt,draw] (36) at (1.3,-.35) {};
\node[circle,fill=black,inner sep=0.8pt,draw] (311) at (0.25,0.63) {};
\node[circle,fill=black,inner sep=1.5pt,draw] (312) at (0.25,1) {};
\node[circle,fill=black,inner sep=1.5pt,draw] (new) at (0.75,-.12) {};

\node () at (0.35,1.13) {\tiny $c$};
\node () at (0.35,0.6) {\tiny $d$};
\node () at (-0.25,0) {\tiny $a$};
\node () at (1,-0.1) {\tiny $z$};
\node () at (0.45,-0.6) {\tiny $z:=min(a,c+d)$};

\draw (31) edge[dashed] (32);
\draw (32) edge[dotted] (33);
\draw (new) edge[dotted] (33);
\draw (31) edge[dotted] (33);

\draw (new) -- (36); \draw (32)--(35); \draw (34) -- (35) -- (36) -- (34);
\draw (31)--(311);\draw (34)--(312);
\draw (311) edge[me=2] (312);
\end{scope}

\end{tikzpicture}
\caption{\small{Topological types of genus 5 graphs with a divisor of degree $4$ and rank at least~$1$, independent of edge lengths. The edges highlighted in gray belong to more than one of the subgraphs $G_v$.}}
\end{figure}
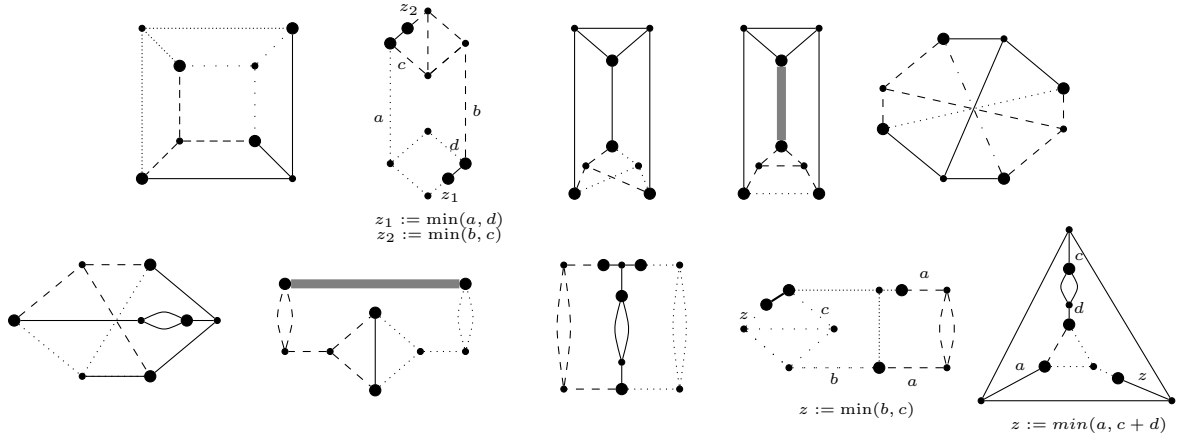

 \subsubsection{First family}
For this family, we consider two cases.  In both cases, the depicted divisor has rank at least $1$ according to Proposition~\ref{prop:configuration-for-genus-5}.

\begin{figure}[H]
\begin{tikzpicture}
\node[circle,fill=black,inner sep=0.8pt,draw] (a6) at (-1,-.5) {};
\node[circle,fill=black,inner sep=0.8pt,draw] (b6) at (-1,.5) {};
\node[circle,fill=black,inner sep=1.5pt,draw] (c6) at (-0.07,-0.5) {};
\node[circle,fill=black,inner sep=0.8pt,draw] (d6) at (-0.07,0.5) {};
\node[circle,fill=black,inner sep=1.5pt,draw] (x) at (0.4,-0.5) {};
\node[circle,fill=black,inner sep=0.8pt,draw] (y) at (0.4,0.5) {};
\node[circle,fill=black,inner sep=0.8pt,draw] (e6) at (1.3,-0.5) {};
\node[circle,fill=black,inner sep=0.8pt,draw] (f6) at (1.3,0.5) {};
\node[circle,fill=black,inner sep=1.5pt,draw] (new-upright) at (0.7,0.5) {};
\node[circle,fill=black,inner sep=1.5pt,draw] (new-upleft) at (-0.5,0.5) {};

\node () at (-.8,0.65) {\tiny $a$};
\node () at (-.5,-0.7) {\tiny $a$};
\node () at (1,0.65) {\tiny $b$};
\node () at (.85,-0.7) {\tiny $b$};

\draw (a6) edge[me=2,dashed] (b6); 
\draw (c6) edge[dashed] (a6);
\draw (b6) edge[dashed] (new-upleft);
\draw (new-upleft) edge[densely dotted] (d6);
\draw (c6) edge[densely dotted] (d6);
\draw (d6) edge[densely dotted] (new-upright);
\draw (c6) edge (x);
\draw (x) edge[densely dotted] (y);

\draw (x) edge[loosely dotted] (e6);
\draw (e6) edge[me=2,loosely dotted] (f6);
\draw (new-upright) edge[loosely dotted] (f6);

\begin{scope}[shift={(+4,0)}]
\node[circle,fill=black,inner sep=0.8pt,draw] (a6) at (-1,-.5) {};
\node[circle,fill=black,inner sep=0.8pt,draw] (b6) at (-1,.5) {};
\node[circle,fill=black,inner sep=1.5pt,draw] (c6) at (-0.07,-0.5) {};
\node[circle,fill=black,inner sep=0.8pt,draw] (d6) at (-0.07,0.5) {};
\node[circle,fill=black,inner sep=0.8pt,draw] (x) at (0.4,-0.5) {};
\node[circle,fill=black,inner sep=1.5pt,draw] (y) at (0.4,0.5) {};
\node[circle,fill=black,inner sep=0.8pt,draw] (e6) at (1.3,-0.5) {};
\node[circle,fill=black,inner sep=0.8pt,draw] (f6) at (1.3,0.5) {};
\node[circle,fill=black,inner sep=1.5pt,draw] (new-downright) at (0.7,-0.5) {};
\node[circle,fill=black,inner sep=1.5pt,draw] (new-upleft) at (-0.5,0.5) {};

\node () at (-.8,0.65) {\tiny $a$};
\node () at (-.5,-0.7) {\tiny $a$};
\node () at (.85,0.65) {\tiny $c$};
\node () at (1,-0.7) {\tiny $c$};

\draw (a6) edge[me=2,dashed] (b6); 
\draw (c6) edge[dashed] (a6);
\draw (b6) edge[dashed] (new-upleft);

\draw (new-upleft) edge[densely dotted] (d6);
\draw (c6) edge[densely dotted] (d6);
\draw (d6) edge[densely dotted] (y);

\draw (c6) edge (new-downright);
\draw (x) edge (y);

\draw (e6) edge[me=2,loosely dotted] (f6);
\draw (y) edge[loosely dotted] (f6);
\draw (new-downright) edge[loosely dotted] (e6); 
\end{scope}
\end{tikzpicture}
\caption{\small{Degree $4$ divisors on the first genus $5$ family broken up into the cases: \textit{1)~${b<c}$;~2)~$b\geq c$.}}}
\label{fig:first-family-genus5}
\end{figure}
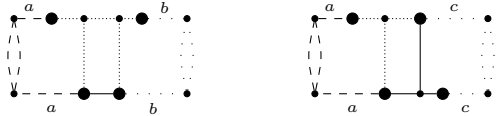

\subsubsection{Second family}
\label{sec:second-family-genus-5} 
Place the first three chips as shown below and the fourth at distance ${\min (a,c+d)}$ from the grey vertex along the dashed arrow.

\begin{figure}[H]
\begin{tikzpicture}[scale=1.4]
\node[circle,fill=black,inner sep=0.8pt,draw] (a9) at (-1,0) {};
\node[circle,fill=black,inner sep=1.5pt,draw] (b9) at (-0.5, 0.5) {};
\node[circle,fill=black,inner sep=0.8pt,draw] (c9) at (0.5, 0.5) {};
\node[circle,fill=black,inner sep=0.8pt,draw] (b9a) at (-0.5, 1.2) {};
\node[circle,fill=black,inner sep=0.8pt,draw] (c9a) at (0.5, 1.2) {};
\node[circle,fill=black,inner sep=1.5pt,draw] (d9) at (1,0) {};
\node[circle,fill=gray!20,inner sep=1.2pt,draw] (e9) at (0.5, -0.5) {};
\node[circle,fill=black,inner sep=0.8pt,draw] (f9) at (-0.5, -0.5) {};

\node[circle,fill=black,inner sep=1.5pt,draw] (new-up) at (0.5, 0.8) {};
\node[circle,fill=black,inner sep=0.8pt,draw] (new-inside) at (-0.75, -0.06) {};

\node () at (-.65,.75) {\tiny $b$};
\node () at (.65,1) {\tiny $b$};
\node () at (-.85,.35) {\tiny $a$};
\node () at (0,-0.7) {\tiny $c$};
\node () at (-.95,-.4) {\tiny $d$};
\node () at (-.5,-.2) {\tiny $d$};

\draw (a9)--(b9);\path (b9) edge (c9);
\draw (c9)--(d9);\path (d9) edge [bend left] (e9); \path (d9) edge [bend right] (e9);
\draw (f9) edge[dashed] (e9);\path (f9) edge [bend left] (a9); \path (a9) edge [bend left=2] (new-inside);

\path (new-inside) edge [bend left=20, dashed] (f9);
\draw (b9) edge (b9a);
\draw (c9) edge (c9a);
\draw (c9a) edge[me=2] (b9a);
\end{tikzpicture}
\caption{\small{The degree $4$ divisor is obtained by placing a fourth chip at distance $\min(a,c+d)$ from the grey vertex along the dashed edge.}}
\end{figure}
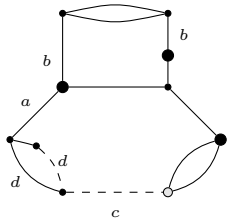

\subsubsection{Fourth family}
There are two possible construction for $D\in \Div(G)$ depending on the relative position of the two longest topological edges. In the first case, the longest two topological edges are adjacent and we may assume $b\geq a \geq \max (c,d)$. The last chip is placed distance $\min(d+e,x)$ from the gray vertex, as indicated by the dashed line. In the second case, the longest two topological edges are not adjacent and suppose $b\geq c\geq a\geq d$. Note that in this case $|y-z|\leq a + x = b$. Verifying that this divisor has rank at least $1$ is done as in Section~\ref{subsubsec:loop-of-loops-genus-4}.

\begin{figure}[H]
\begin{tikzpicture}[scale = 1.3]
\node[circle,fill=black,inner sep=0.8pt,draw] (A1) at (0,0) {};
\node[circle,fill=black,inner sep=1.5pt,draw] (A2) at (-.5,0.5) {};
\node[circle,fill=black,inner sep=1.5pt,draw] (A8) at (1,0) {};
\node[circle,fill=black,inner sep=0.8pt,draw] (A7) at (1.5,0.5) {};

\node[circle,fill=black,inner sep=0.8pt,draw] (A3) at (-.5,1.25) {};
\node[circle,fill=black,inner sep=0.8pt,draw] (A4) at (0,1.75) {};
\node[circle,fill=black,inner sep=0.8pt,draw] (A5) at (1,1.75) {};
\node[circle,fill=gray!30,inner sep=1.2pt,draw] (A6) at (1.5,1.25) {};

\node[circle,fill=black,inner sep=1.5pt,draw] (new-up) at (0.7,1.75) {};
\node[circle,fill=black,inner sep=0.8pt,draw] (new-inloop) at (1.45,0.25) {};

\node () at (.35,1.9) {\tiny $a$};
\node () at (-.7,0.9) {\tiny $a$};
\node () at (.9,1.9) {\tiny $x$};
\node () at (1.65,0.9) {\tiny $d$};
\node () at (1.15,0.5) {\tiny $e$};
\node () at (1.6,0.35) {\tiny $e$};

\draw (A1) edge[me=2] (A2);
\draw (A2)--(A3);
\draw (A3) edge[me=2] (A4);
\draw (A4)--(A5);
\draw (A5) edge[me=2] (A6);
\draw (A6) edge[dashed] (A7);
\draw (A7) edge[bend right] (A8);
\draw (A7) edge[dashed, bend left=20] (new-inloop);
\draw (new-inloop) edge[bend left=20] (A8);
\draw (A8)--(A1);

\begin{scope}[shift={(+3.5,0)}]
\node[circle,fill=black,inner sep=0.8pt,draw] (A1) at (0,0) {};
\node[circle,fill=black,inner sep=1.5pt,draw] (A2) at (-.5,0.5) {};
\node[circle,fill=black,inner sep=0.8pt,draw] (A8) at (1,0) {};
\node[circle,fill=black,inner sep=0.8pt,draw] (A7) at (1.5,0.5) {};

\node[circle,fill=black,inner sep=0.8pt,draw] (A3) at (-.5,1.25) {};
\node[circle,fill=black,inner sep=0.8pt,draw] (A4) at (0,1.75) {};
\node[circle,fill=black,inner sep=0.8pt,draw] (A5) at (1,1.75) {};
\node[circle,fill=black,inner sep=0.8pt,draw] (A6) at (1.5,1.25) {};

\node[circle,fill=black,inner sep=1.5pt,draw] (new-up) at (0.7,1.75) {};
\node[circle,fill=black,inner sep=1.5pt,draw] (new-down) at (0.7,0) {};
\node[circle,fill=black,inner sep=1.5pt,draw] (new-right) at (1.5,0.9) {};

\node () at (.35,1.9) {\tiny $a$};
\node () at (.35,-0.2) {\tiny $a$};
\node () at (.85,-0.2) {\tiny $z$};
\node () at (-.7,0.9) {\tiny $a$};
\node () at (.9,1.9) {\tiny $x$};

\node () at (1.6,1.1) {\tiny $x$};
\node () at (1.6,0.7) {\tiny $y$};

\draw (A1) edge[me=2] (A2);
\draw (A2)--(A3);
\draw (A3) edge[me=2] (A4);
\draw (A4)--(A5);
\draw (A5) edge[me=2] (A6);
\draw (A6)--(A7);
\draw (A7) edge[me=2] (A8);
\draw (A8)--(A1);
\end{scope}
\end{tikzpicture}
\caption{\small{A configuration of $4$ chips having rank at least $1$ on the loops of loops.}}
\end{figure}
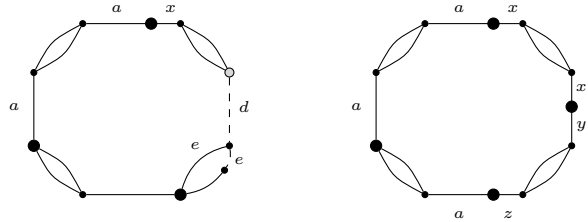

\subsubsection{Sixth family}
For this family, we consider three cases. For each case, the depicted divisor below has rank at least $1$ according to Proposition~\ref{prop:configuration-for-genus-5}.

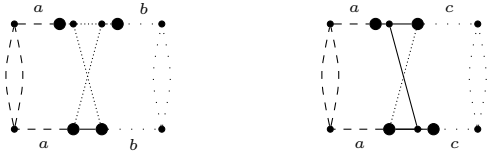
\begin{figure}[H]
\begin{tikzpicture}[scale=1.4]
\node[circle,fill=black,inner sep=0.8pt,draw] (a6) at (-.63,-.5) {};
\node[circle,fill=black,inner sep=0.8pt,draw] (b6) at (-.63,.5) {};
\node[circle,fill=black,inner sep=1.5pt,draw] (c6) at (-0.07,-0.5) {};
\node[circle,fill=black,inner sep=0.8pt,draw] (d6) at (-0.07,0.5) {};
\node[circle,fill=black,inner sep=1.5pt,draw] (x) at (0.2,-0.5) {};
\node[circle,fill=black,inner sep=0.8pt,draw] (y) at (0.2,0.5) {};
\node[circle,fill=black,inner sep=0.8pt,draw] (e6) at (0.77,-0.5) {};
\node[circle,fill=black,inner sep=0.8pt,draw] (f6) at (0.77,0.5) {};

\node[circle,fill=black,inner sep=1.5pt,draw] (m) at (-.2,0.5) {};
\node[circle,fill=black,inner sep=1.5pt,draw] (n) at (0.35,0.5) {};

\node () at (-.4,0.65) {\tiny $a$};
\node () at (-.35,-0.65) {\tiny $a$};
\node () at (.6,0.65) {\tiny $b$};
\node () at (.5,-0.65) {\tiny $b$};

\draw (a6) edge[me=2,dashed] (b6); 
\draw (c6) edge[dashed] (a6);
\draw (b6) edge[dashed] (m);

\draw (d6) edge[densely dotted] (m);
\draw (c6) edge[densely dotted] (y);
\draw (x) edge[densely dotted] (d6);
\draw (d6) edge[densely dotted] (n);

\draw (c6) edge (x);

\draw (x) edge[loosely dotted] (e6);
\draw (y) edge[loosely dotted] (f6);
\draw (e6) edge[me=2,loosely dotted] (f6);

\begin{scope}[shift={(+3,0)}]
\node[circle,fill=black,inner sep=0.8pt,draw] (a6) at (-.63,-.5) {};
\node[circle,fill=black,inner sep=0.8pt,draw] (b6) at (-.63,.5) {};
\node[circle,fill=black,inner sep=1.5pt,draw] (c6) at (-0.07,-0.5) {};
\node[circle,fill=black,inner sep=0.8pt,draw] (d6) at (-0.07,0.5) {};
\node[circle,fill=black,inner sep=0.8pt,draw] (x) at (0.2,-0.5) {};
\node[circle,fill=black,inner sep=1.5pt,draw] (y) at (0.2,0.5) {};
\node[circle,fill=black,inner sep=0.8pt,draw] (e6) at (0.77,-0.5) {};
\node[circle,fill=black,inner sep=0.8pt,draw] (f6) at (0.77,0.5) {};

\node[circle,fill=black,inner sep=1.5pt,draw] (m) at (-.2,0.5) {};
\node[circle,fill=black,inner sep=1.5pt,draw] (n) at (0.35,-0.5) {};
\node () at (.5,0.65) {\tiny $c$};
\node () at (.55,-0.65) {\tiny $c$};

\node () at (-.4,0.65) {\tiny $a$};
\node () at (-.35,-0.65) {\tiny $a$};.

\draw (a6) edge[me=2,dashed] (b6); 
\draw (c6) edge[dashed] (a6);
\draw (b6) edge[dashed] (m);

\draw (d6) -- (m);
\draw (c6) edge[densely dotted] (y);
\draw (c6) edge (x);
\draw (x) edge (d6);
\draw (y) -- (d6);
\draw (c6) -- (n);

\draw (y) edge[loosely dotted] (f6);
\draw (x) edge[loosely dotted] (e6);
\draw (e6) edge[me=2,loosely dotted] (f6);
\end{scope}
\end{tikzpicture}
\caption{\small{Degree $4$ configurations on the sixth family having rank at least $1$, depending on the edge lengths as above.}}
\end{figure}

\subsubsection{Seventh family}
For this family, we consider three cases. For each case, the depicted divisor below has rank at least $1$ according to Proposition~\ref{prop:configuration-for-genus-5}. Note that the last divisor has two chips placed on the same vertex.

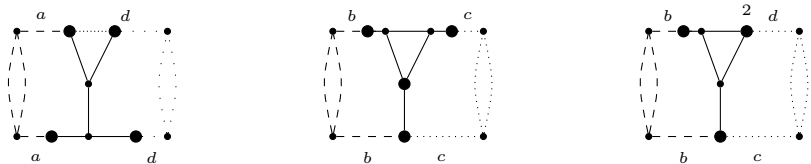
\begin{figure}[H]
\begin{tikzpicture}[scale=1.4]
\node[circle,fill=black,inner sep=0.8pt,draw] (a6) at (-.43,-.5) {};
\node[circle,fill=black,inner sep=0.8pt,draw] (b6) at (-.43,.5) {};
\node[circle,fill=black,inner sep=0.8pt,draw] (c6) at (.25,0) {};
\node[circle,fill=black,inner sep=0.8pt,draw] (j6) at (.25,-.5) {};
\node[circle,fill=black,inner sep=1.5pt,draw] (d6) at (.07,0.5) {};
\node[circle,fill=black,inner sep=1.5pt,draw] (h6) at (.5,0.5) {};
\node[circle,fill=black,inner sep=0.8pt,draw] (e6) at (1,-0.5) {};
\node[circle,fill=black,inner sep=0.8pt,draw] (f6) at (1,0.5) {};

\node[circle,fill=black,inner sep=1.5pt,draw] (x) at (0.7,-0.5) {};
\node[circle,fill=black,inner sep=1.5pt,draw] (y) at (-0.1,-0.5) {};

\node () at (-.2,0.65) {\tiny $a$};
\node () at (-.25,-0.7) {\tiny $a$};
\node () at (.6,0.65) {\tiny $d$};
\node () at (.85,-0.7) {\tiny $d$};

\draw (a6) edge[me=2,dashed] (b6); 
\draw (a6) edge[dashed] (y);
\draw (b6) edge[dashed] (d6);
\draw (c6) edge (d6);
\draw (d6) edge[densely dotted] (h6);
\draw (h6) edge (c6);
\draw (y)--(x);
\draw (j6)--(c6);

\draw (h6) edge[loosely dotted] (f6);
\draw (x) edge[loosely dotted] (e6);
\draw (e6) edge[me=2,loosely dotted] (f6);

\begin{scope}[shift={(+3,0)}]
\node[circle,fill=black,inner sep=0.8pt,draw] (a6) at (-.43,-.5) {};
\node[circle,fill=black,inner sep=0.8pt,draw] (b6) at (-.43,.5) {};
\node[circle,fill=black,inner sep=1.5pt,draw] (c6) at (.25,0) {};
\node[circle,fill=black,inner sep=1.5pt,draw] (j6) at (.25,-.5) {};
\node[circle,fill=black,inner sep=0.8pt,draw] (d6) at (.07,0.5) {};
\node[circle,fill=black,inner sep=0.8pt,draw] (h6) at (.5,0.5) {};
\node[circle,fill=black,inner sep=0.8pt,draw] (e6) at (1,-0.5) {};
\node[circle,fill=black,inner sep=0.8pt,draw] (f6) at (1,0.5) {};

\node[circle,fill=black,inner sep=1.5pt,draw] (x) at (0.7,0.5) {};
\node[circle,fill=black,inner sep=1.5pt,draw] (y) at (-0.1,0.5) {};

\node () at (-.25,0.65) {\tiny $b$};
\node () at (-.1,-0.7) {\tiny $b$};
\node () at (.85,0.65) {\tiny $c$};
\node () at (.6,-0.7) {\tiny $c$};

\draw (a6) edge[me=2,dashed] (b6); 
\draw (a6) edge[dashed] (j6);
\draw (b6) edge[dashed] (d6);

\draw (d6) -- (y);
\draw (c6) edge (d6);
\draw (d6) edge (x);
\draw (h6) edge (c6);
\draw (j6)--(c6);

\draw (x) edge[dotted] (f6);
\draw (e6) edge[me=2,dotted] (f6);
\draw (j6) edge[dotted] (e6);
\end{scope}

\begin{scope}[shift={(+6,0)}]
\node[circle,fill=black,inner sep=0.8pt,draw] (a6) at (-.43,-.5) {};
\node[circle,fill=black,inner sep=0.8pt,draw] (b6) at (-.43,.5) {};
\node[circle,fill=black,inner sep=0.8pt,draw] (c6) at (.25,0) {};
\node[circle,fill=black,inner sep=1.5pt,draw] (j6) at (.25,-.5) {};
\node[circle,fill=black,inner sep=0.8pt,draw] (d6) at (.07, 0.5) {};
\node[circle,fill=black,inner sep=1.5pt,draw] (h6) at (.5,0.5) {};
\node[circle,fill=black,inner sep=0.8pt,draw] (e6) at (1,-0.5) {};
\node[circle,fill=black,inner sep=0.8pt,draw] (f6) at (1,0.5) {};
\node[circle,fill=black,inner sep=1.5pt,draw] (y) at (-0.1,0.5) {};

\node () at (0.5,.7) {\tiny$2$};
\node () at (-.25,0.65) {\tiny $b$};
\node () at (-.1,-0.7) {\tiny $b$};
\node () at (.6,-0.7) {\tiny $c$};
\node () at (.75,0.65) {\tiny $d$};

\draw (a6) edge[me=2,dashed] (b6); 
\draw (a6) edge[dashed] (j6);
\draw (b6) edge[dashed] (d6);

\draw (d6) -- (y);
\draw (c6) edge (d6);
\draw (d6) edge (h6);
\draw (h6) edge (c6);
\draw (j6)--(c6);

\draw (e6) edge[me=2,dotted] (f6);
\draw (j6) edge[dotted] (e6);
\draw (h6) edge[dotted] (f6);
\end{scope}
\end{tikzpicture}
\caption{\small{Degree $4$ divisors on the seventh genus $5$ family broken up into the cases: \textit{1) $a\leq b,~d\leq c$,~2) $b\leq a, c\leq d$,~3) $b\leq a, d\leq c$.}}}

\end{figure}

\subsubsection{Ninth family}
For this family, we consider two cases. In both cases, the divisor shown below has rank at least $1$ according to Proposition~\ref{prop:configuration-for-genus-5}.

\begin{figure}[H]
\begin{tikzpicture}[scale=1.6]
\node[circle,fill=black,inner sep=1.5pt,draw] (31) at (0.25,0.43) {};
\node[circle,fill=black,inner sep=0.8pt,draw] (32) at (0,0) {};
\node[circle,fill=black,inner sep=1.5pt,draw] (33) at (0.5,0) {};
\node[circle,fill=black,inner sep=0.8pt,draw] (34) at (0.25,1.4) {};
\node[circle,fill=black,inner sep=0.8pt,draw] (35) at (-.65,-.35) {};
\node[circle,fill=black,inner sep=0.8pt,draw] (36) at (1.15,-.35) {};
\node[circle,fill=black,inner sep=0.8pt,draw] (37) at (0,-0.35) {};
\node[circle,fill=black,inner sep=1.5pt,draw] (38) at (0.5,-.35) {};
\node[circle,fill=black,inner sep=1.5pt,draw] (new-left) at (-0.25,-.15) {};

\node () at (-0.45,-.15) {\tiny $a$};
\node () at (0.35,.95) {\tiny $a$};
\node () at (0.75,-.07) {\tiny $c$};

\draw (31) edge[dotted] (33);
\draw (33) edge[dashed] (32);
\draw (32) edge[dashed] (31);
\draw (33) -- (36); 
\draw (32) edge[dashed](new-left);
\draw (new-left)--(35); 
\draw (34) -- (35) --(37) edge[me=2] (38);
\draw (36) -- (34); 
\draw (31)--(34);
\draw (36)--(38);

\begin{scope}[shift={(+3,0)}]
\node[circle,fill=black,inner sep=0.8pt,draw] (31) at (0.25,0.43) {};
\node[circle,fill=black,inner sep=1.5pt,draw] (32) at (0,0) {};
\node[circle,fill=black,inner sep=1.5pt,draw] (33) at (0.5,0) {};
\node[circle,fill=black,inner sep=0.8pt,draw] (34) at (0.25,1.4) {};
\node[circle,fill=black,inner sep=0.8pt,draw] (35) at (-.65,-.35) {};
\node[circle,fill=black,inner sep=0.8pt,draw] (36) at (1.15,-.35) {};
\node[circle,fill=black,inner sep=0.8pt,draw] (37) at (0,-0.35) {};
\node[circle,fill=black,inner sep=1.5pt,draw] (38) at (0.5,-.35) {};
\node[circle,fill=black,inner sep=1.5pt,draw] (new-up) at (0.25,.75) {};

\node () at (-0.25,-.03) {\tiny $b$};
\node () at (0.35,1.05) {\tiny $b$};
\node () at (0.75,-.07) {\tiny $c$};

\draw (33) edge[dotted] (32);

\draw (31) edge[dashed] (33);
\draw(32)edge[dashed] (31); 
\draw (31)edge[dashed](new-up);

\draw (33) -- (36); \draw (32)--(35); 
\draw (34) -- (35) --(37) edge[me=2] (38);
\draw (36) -- (34); 
\draw(new-up)--(34);
\draw (36)--(38);
\end{scope}
\end{tikzpicture}
\caption{\small{Degree $4$ divisors on the ninth genus $5$ family broken up into the cases: \textit{1)~$a:=\min(a,b,c)$ 2)~$b:=\min(a,b,c)$}.}}
\end{figure}
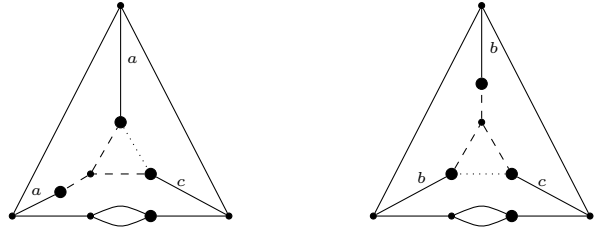

\subsection{General graphs of genus 5 via edge contractions.} 
\label{sec:general-graphs-of-genus-5}
Since Proposition~\ref{prop:configuration-for-genus-5} allows for multiple contractions, we only need examine the families for which it does not apply. These are the second and the fourth. 

\subsubsection{Edge contractions to the second family}
Examining both cases in Section~\ref{sec:second-family-genus-5}, we  need to consider only contractions of the uppermost edge connecting the two loops. We perform the contraction and place the chips as shown below. The new divisor satisfies Proposition~\ref{prop:configuration-for-genus-5} and thus remains of rank at least $1$ under repeated edge contractions.

\begin{figure}[H]
\begin{tikzpicture}[scale = 1.4]
\node[circle,fill=black,inner sep=0.8pt,draw] (a9) at (-1,0) {};
\node[circle,fill=black,inner sep=0.8pt,draw] (b9) at (-0.5, 0.5) {};
\node[circle,fill=black,inner sep=0.8pt,draw] (c9) at (0.5, 0.5) {};
\node[circle,fill=black,inner sep=0.8pt,draw] (b9a) at (-0.5, 1.2) {};
\node[circle,fill=black,inner sep=0.8pt,draw] (c9a) at (0.5, 1.2) {};
\node[circle,fill=black,inner sep=1.5pt,draw] (d9) at (1,0) {};
\node[circle,fill=black,inner sep=0.8pt,draw] (e9) at (0, -0.5) {};
\node[circle,fill=black,inner sep=0.8pt,draw] (f9) at (0, -0.5) {};

\node[circle,fill=black,inner sep=1.5pt,draw] (m) at (-0.5, 0.7) {};
\node[circle,fill=black,inner sep=1.5pt,draw] (n) at (-0.8, .17) {};
\node[circle,fill=black,inner sep=1.5pt,draw] (k) at (0.7, 0) {};

\node () at (0.15,-.13) {\tiny $z$};
\node () at (0.9,-.4) {\tiny $d$};
\node () at (-0.65,.9) {\tiny $b$};
\node () at (0.65,.8) {\tiny $b$};
\node () at (-0.75,.4) {\tiny $c$};
\node () at (-1.02,.1) {\tiny $x$};
\node () at (0.85,.3) {\tiny $c$};
\node () at (0,-.75) {\tiny $z:=\min (x,d)$};

\draw (a9)--(b9);\path (b9) edge (c9);
\draw (c9)--(d9);\path (d9) edge [bend left] (e9); \path (d9) edge [bend right] (e9);
\draw (f9)--(e9);\path (f9) edge [bend left] (a9); \path (a9) edge [bend left] (f9);
\draw (b9) edge (b9a);
\draw (c9) edge (c9a);
\draw (c9a) edge[me=2] (b9a);
\end{tikzpicture}
\caption{\small{Edge contraction of second family and a degree $4$ divisor of rank at least $1$.}}
\end{figure}
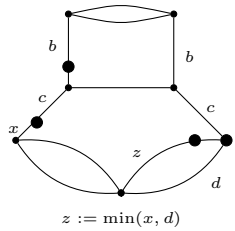

\subsubsection{Edge contractions of the fourth family}
Note that edge contraction of any topological edge participating in a cycle produces a loop and then the existence of the desired divisor follows from the bridge lemma. Therefore, we can contract only a topological edge connecting two loops, as illustrated below. We place the chip as shown below and the remaining three chips can be placed according to the construction of Section~\ref{subsubsec:loop-of-loops-genus-4} as if the cycle of two loops were one loop. 

\begin{figure}[H]
\begin{tikzpicture}[scale=1.3]
\node[circle,fill=black,inner sep=0.8pt,draw] (A1) at (0,0) {};
\node[circle,fill=black,inner sep=1.5pt,draw] (A2) at (-.5,0.85) {};
\node[circle,fill=black,inner sep=0.8pt,draw] (A8) at (1,0) {};
\node[circle,fill=black,inner sep=0.8pt,draw] (A7) at (1.5,0.5) {};

\node[circle,fill=black,inner sep=0.8pt,draw] (A3) at (-.5,.85) {};
\node[circle,fill=black,inner sep=0.8pt,draw] (A4) at (0,1.75) {};
\node[circle,fill=black,inner sep=0.8pt,draw] (A5) at (1,1.75) {};
\node[circle,fill=black,inner sep=0.8pt,draw] (A6) at (1.5,1.25) {};

\draw (A1) edge[me=2] (A2);
\draw (A2)--(A3);
\draw (A3) edge[me=2] (A4);
\draw (A4)--(A5);
\draw (A5) edge[me=2] (A6);
\draw (A6)--(A7);
\draw (A7) edge[me=2] (A8);
\draw (A8)--(A1);
\end{tikzpicture}
\caption{\small{Contraction of a topological edge in loops of loops.}}
\end{figure}
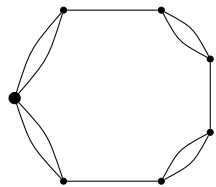

\section{Graphs of high genus}
\label{sec:graphs-of-high-genus}
In this final section, we record some infinite families of graphs of increasing genus for which the existence conjecture holds in rank $1$. The main results of this section are Theorem~\ref{thm:ladder} and Proposition~\ref{prop:insert-kites}.
\subsection{Complete and complete $k$-partite graphs.}
Suppose $G$ is a graph homeomorphic to $K_n$, the complete graph on $n$ vertices. We can place one chip on all but one of its topological vertices and obtain a divisor ~$D$ of rank at least one. Note further that $\deg (D) \leq \lfloor \frac{g(K_n)+3}{2} \rfloor$, where $g(K_n)=\frac{(n-1)(n-2)}{2}$ is the genus of~$K_n$.
\begin{proposition}
Let $n_1 \leq \cdots \leq n_s$ be integers. Suppose $G$ is a graph homeomorphic to the complete $s$-partite graph $K_{n_1,\cdots, n_s}$. Then $G$ admits a divisor $D$ of degree $\sum_{i=1}^{s-1} n_i$ and rank at least one. Furthermore, the gonality of $G$ is precisely $\sum_{i=1}^{s-1} n_i$.
\end{proposition}

\begin{proof}
Let $\{V_{l}\}^{s}_{l=1}$ with $|V_l|=n_l$ partition the set of topological vertices such that two vertices in $V_i$ and $V_j$ are connected along a topological edge if and only if $i\neq j$. Then, consider the divisor
\[D=\sum_{v\in V_1 \cup \cdots \cup V_{s-1}} (v). \]
It has $\deg(D)=\sum_{i=1}^{s-1} n_i$ and rank at least one by running Dhar's burning algorithm. That this is the gonality follows from \cite[Theorem~2]{BruynGijswijt}.
\end{proof}

\subsection{Ladder graph}
Let $G$ be homeomorphic to the genus g ladder graph from Figure~\ref{fig:Ladder-graph-of-genus-g}. In this section we show that $G$ supports a divisor of degree $\lfloor (g+3)/2 \rfloor$ and rank at least $1$. Note that the genus g ladder graph has $g-3$ vertical edges, $2$ cycles, and $g-2$ cells, where we do not count the two end cycles as cells.

\begin{figure}[H]
\begin{tikzpicture}
\node[circle,fill=black,inner sep=0.8pt,draw] (a) at (0,0) {};
\node[circle,fill=black,inner sep=0.8pt,draw] (b) at (1,0) {};
\node[circle,fill=black,inner sep=0.8pt,draw] (c) at (2,0) {};
\node[circle,fill=black,inner sep=0.8pt,draw] (e) at (4,0) {};
\node[circle,fill=black,inner sep=0.8pt,draw] (d) at (5,0) {};
\node[circle,fill=black,inner sep=0.8pt,draw] (f) at (6,0) {};

\node () at (3,0.5) {\large $\cdots$};

\node[circle,fill=black,inner sep=0.8pt,draw] (a') at (0,1) {};
\node[circle,fill=black,inner sep=0.8pt,draw] (b') at (1,1) {};
\node[circle,fill=black,inner sep=0.8pt,draw] (c') at (2,1) {};
\node[circle,fill=black,inner sep=0.8pt,draw] (e') at (4,1) {};
\node[circle,fill=black,inner sep=0.8pt,draw] (d') at (5,1) {};
\node[circle,fill=black,inner sep=0.8,draw] (f') at (6,1) {};

\draw (b) -- (b') -- (c') -- (c)--(d)--(d');
\draw (b)--(c);
\draw (c')--(d');
\draw (a) edge[me=2] (a');
\draw (f) edge[me=2] (f');
\draw (a)--(b);
\draw (a')--(b');
\draw (d)--(f);
\draw (d')--(f');
\draw (e)--(e');
\end{tikzpicture}
\caption{\small{Ladder graph of genus $g$.}}
\label{fig:Ladder-graph-of-genus-g}
\end{figure}
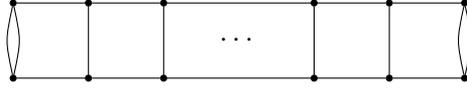

\begin{lemma}
\label{lem:ladder-graph}
Let $G$ be the graph shown below with edge lengths $a,b,c,d\in \mathbb{N}$. Denote by $v_1,v_2$ and $w_1,w_2$ the leftmost and rightmost pairs of vertices, respectively. Then there exists {$D \in \Div(G)$} of $\deg(D)=3$ and $\rk(D)\geq 1$ such that $[D-v_1-v_2]$ and $[D - w_1- w_2]$ are effective.

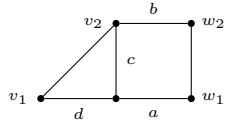
\begin{figure}[H]
\begin{tikzpicture}
\node[circle,fill=black,inner sep=0.8pt,draw] (a) at (0,0) {};
\node[circle,fill=black,inner sep=0.8pt,draw] (b) at (1,0) {};
\node[circle,fill=black,inner sep=0.8pt,draw] (c) at (1,1) {};
\node[circle,fill=black,inner sep=0.8pt,draw] (d) at (0,1) {};
\node[circle,fill=black,inner sep=0.8pt,draw] (e) at (-1,0) {};

\node () at (-1.3,0) {\tiny $v_1$};
\node () at (-.3,1) {\tiny $v_2$};
\node () at (1.3,0) {\tiny $w_1$};
\node () at (1.3,1) {\tiny $w_2$};

\node () at (0.5,-0.2) {\tiny $a$};
\node () at (0.5,1.2) {\tiny $b$};
\node () at (0.2,0.5) {\tiny $c$};
\node () at (-0.5,-.2) {\tiny $d$};

\draw (e) -- (a) -- (b) -- (c) -- (d) -- (a);
\draw (e) -- (d);

\end{tikzpicture}
\end{figure}
\end{lemma}

\begin{proof}
There are four cases to be examined. The order in which they appear in Figure~\ref{fig:ladder-graph-lemma-possible-divisors} from left to right is: (1) $a>b$; (2) $a\leq b < a+\min (c,d)$; (3) $a+\min (c,d) < b$ and $c\leq d$; and (4) $a+\min (c,d) < b$ and $d<c$. Veryfing that each divisor has the desired properties follows by running Dhar's burning algorithm. 

\begin{figure}[H]
\begin{tikzpicture}
\node[circle,fill=black,inner sep=0.8pt,draw] (a) at (0,0) {};
\node[circle,fill=black,inner sep=0.8pt,draw] (b) at (1,0) {};
\node[circle,fill=black,inner sep=0.8pt,draw] (c) at (1,1) {};
\node[circle,fill=black,inner sep=1.5pt,draw] (d) at (0,1) {};
\node[circle,fill=black,inner sep=1.5pt,draw] (e) at (-1,0) {};
\node[circle,fill=black,inner sep=1.5pt,draw] (mid-ab) at (0.3,0) {};

\node () at (0.7,-0.2) {\tiny $b$};
\node () at (0.5,1.2) {\tiny $b$};

\draw (e) -- (a) -- (b) -- (c) -- (d) -- (a);
\draw (e) -- (d);

\begin{scope}[shift={(+3,0)}]
\node[circle,fill=black,inner sep=0.8pt,draw] (a) at (0,0) {};
\node[circle,fill=black,inner sep=0.8pt,draw] (b) at (1,0) {};
\node[circle,fill=black,inner sep=0.8pt,draw] (c) at (1,1) {};
\node[circle,fill=black,inner sep=1.5pt,draw] (d) at (0,1) {};
\node[circle,fill=black,inner sep=1.5pt,draw] (e) at (-1,0) {};
\node[circle,fill=black,inner sep=1.5pt,draw] (mid-ad) at (0,0.4) {};

\node () at (0,-0.4) {\tiny $z:=b-c$};
\node () at (0.2,0.15) {\tiny $z$};

\draw (e) -- (a) -- (b) -- (c) -- (d) -- (a);
\draw (e) -- (d);
\end{scope}

\begin{scope}[shift={(+6,0)}]
\node[circle,fill=black,inner sep=0.8pt,draw] (a) at (0,0) {};
\node[circle,fill=black,inner sep=1.5pt,draw] (b) at (1,0) {};
\node[circle,fill=black,inner sep=1.5pt,draw] (c) at (1,1) {};
\node[circle,fill=black,inner sep=0.8pt,draw] (d) at (0,1) {};
\node[circle,fill=black,inner sep=0.8pt,draw] (e) at (-1,0) {};
\node[circle,fill=black,inner sep=1.5pt,draw] (mid-ad) at (0,0.4) {};

\node () at (0.2,0.7) {\tiny $d$};

\draw (e) -- (a) -- (b) -- (c) -- (d) -- (a);
\draw (e) -- (d);
\end{scope}

\begin{scope}[shift={(+9,0)}]
\node[circle,fill=black,inner sep=0.8pt,draw] (a) at (0,0) {};
\node[circle,fill=black,inner sep=1.5pt,draw] (b) at (1,0) {};
\node[circle,fill=black,inner sep=1.5pt,draw] (c) at (1,1) {};
\node[circle,fill=black,inner sep=0.8pt,draw] (d) at (0,1) {};
\node[circle,fill=black,inner sep=0.8pt,draw] (e) at (-1,0) {};
\node[circle,fill=black,inner sep=1.5pt,draw] (mid-ae) at (-0.4,0) {};

\node () at (-0.7,-0.2) {\tiny $c$};

\draw (e) -- (a) -- (b) -- (c) -- (d) -- (a);
\draw (e) -- (d);
\end{scope}
\end{tikzpicture}
\caption{\small{Divisors with desired properties for all edge lengths.}}
\label{fig:ladder-graph-lemma-possible-divisors}
\end{figure}
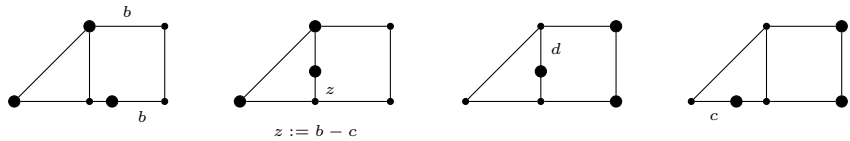\vspace{-0.7cm} \end{proof}

The motivation behind this result comes from the following decomposition of graphs homeomorphic to the genus~6 ladder graph.

\begin{figure}[H]
\begin{tikzpicture}
\begin{scope}[shift = {(-6,0)}]
\node[circle,fill=black,inner sep=0.8pt,draw] (a) at (0,0) {};
\node[circle,fill=black,inner sep=0.8pt,draw] (b) at (1,0) {};
\node[circle,fill=black,inner sep=0.8pt,draw] (c) at (2,0) {};
\node[circle,fill=black,inner sep=0.8pt,draw] (d) at (3,0) {};
\node[circle,fill=black,inner sep=0.8pt,draw] (f) at (4,0) {};

\node[circle,fill=black,inner sep=0.8pt,draw] (a') at (0,1) {};
\node[circle,fill=black,inner sep=0.8pt,draw] (b') at (1,1) {};
\node[circle,fill=black,inner sep=0.8pt,draw] (c') at (2,1) {};
\node[circle,fill=black,inner sep=0.8pt,draw] (d') at (3,1) {};
\node[circle,fill=black,inner sep=0.8pt,draw] (f') at (4,1) {};

\node () at (5,0.5) {\large $=$};

\draw (b) -- (b') -- (c') -- (c)--(d)--(d');
\draw (b)--(c);
\draw (c')--(d');
\draw (a) edge[me=2] (a');
\draw (f) edge[me=2] (f');
\draw (a)--(b);
\draw (a')--(b');
\draw (d)--(f);
\draw (d')--(f');
\end{scope}

\node[circle,fill=black,inner sep=0.8pt,draw] (a1) at (0,0) {};
\node[circle,fill=black,inner sep=0.8pt,draw] (a2) at (0.7,0) {};
\node[circle,fill=black,inner sep=0.8pt,draw] (a3) at (1.2,0) {};

\node[circle,fill=black,inner sep=0.8pt,draw] (a4) at (1.5,0) {};
\node[circle,fill=black,inner sep=0.8pt,draw] (a5) at (2.5,0) {};
\node[circle,fill=black,inner sep=0.8pt,draw] (a6) at (3,0) {};

\node[circle,fill=black,inner sep=0.8pt,draw] (a7) at (4,0) {};
\node[circle,fill=black,inner sep=0.8pt,draw] (a8) at (4.5,0) {};
\node[circle,fill=black,inner sep=0.8pt,draw] (a9) at (5.5,0) {};

\node[circle,fill=black,inner sep=0.8pt,draw] (b1) at (0,1) {};
\node[circle,fill=black,inner sep=0.8pt,draw] (b2) at (1,1) {};
\node[circle,fill=black,inner sep=0.8pt,draw] (b3) at (1.5,1) {};

\node[circle,fill=black,inner sep=0.8pt,draw] (b4) at (2.5,1) {};
\node[circle,fill=black,inner sep=0.8pt,draw] (b5) at (3,1) {};
\node[circle,fill=black,inner sep=0.8pt,draw] (b6) at (4,1) {};

\node[circle,fill=black,inner sep=0.8pt,draw] (b7) at (4.3,1) {};
\node[circle,fill=black,inner sep=0.8pt,draw] (b8) at (4.8,1) {};
\node[circle,fill=black,inner sep=0.8pt,draw] (b9) at (5.5,1) {};

\node () at (0.5,1.15) {\tiny $a$};
\node () at (0.35,-.2) {\tiny $a$};

\node () at (5,-.2) {\tiny $b$};
\node () at (5.15,1.15) {\tiny $b$};

\draw (a1) -- (a2);
\draw (a3) -- (a4)--(a5);
\draw (a6) -- (a7);
\draw (a8) -- (a9);

\draw (b1) -- (b2);
\draw (b3)--(b4);
\draw (b5)--(b6)--(b7);
\draw (b8)--(b9);

\draw (a1) edge[me=2] (b1);
\draw (a9) edge[me=2] (b9);
\draw (a4) -- (b3);
\draw (a5) edge[dashed](b4);
\draw (a6) edge[dashed](b5);
\draw (a7) -- (b6);
\end{tikzpicture}
\caption{\small{Decomposition of a genus 6 ladder graph. The dashed edges are identified.}}
\end{figure}
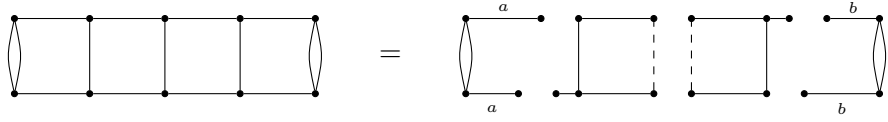

In light of Lemma~\ref{lem:ladder-graph}, the two middle components allow us to perform chip firing moves and advance chips from left to right. More precisely, the lemma asserts that we can place two additional chips on the divisor depicted below and obtain a divisor of rank at least $1$.
\begin{figure}[H]
\begin{tikzpicture}
\node[circle,fill=black,inner sep=0.8pt,draw] (a) at (0,0) {};
\node[circle,fill=black,inner sep=0.8pt,draw] (b) at (1,0) {};
\node[circle,fill=black,inner sep=0.8pt,draw] (c) at (2,0) {};
\node[circle,fill=black,inner sep=0.8pt,draw] (d) at (3,0) {};
\node[circle,fill=black,inner sep=0.8pt,draw] (f) at (4,0) {};
\node[circle,fill=black,inner sep=1.5pt,draw] (new-ab) at (0.7,0) {};

\node () at (0.5, 1.15) {\tiny $a$};
\node () at (0.35, -.2) {\tiny $a$};

\node[circle,fill=black,inner sep=0.8pt,draw] (a') at (0,1) {};
\node[circle,fill=black,inner sep=1.5pt,draw] (b') at (1,1) {};
\node[circle,fill=black,inner sep=0.8pt,draw] (c') at (2,1) {};
\node[circle,fill=black,inner sep=0.8pt,draw] (d') at (3,1) {};
\node[circle,fill=black,inner sep=0.8pt,draw] (f') at (4,1) {};

\draw (b) -- (b') -- (c') -- (c)--(d)--(d');
\draw (b)--(c);
\draw (c')--(d');
\draw (a) edge[me=2] (a');
\draw (f) edge[me=2] (f');
\draw (a)--(b);
\draw (a')--(b');
\draw (d)--(f);
\draw (d')--(f');
\end{tikzpicture}
\caption{\small{Ladder graph of genus 6.}}
\label{fig:Ladder-graph-of-genus6}
\end{figure}
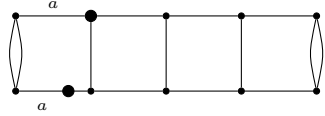

Therefore, we can place two more chips somewhere on the first four cells so that the two divisors below are equivalent.
\begin{figure}[H]
\begin{tikzpicture}
\node[circle,fill=black,inner sep=0.8pt,draw] (a) at (0,0) {};
\node[circle,fill=black,inner sep=0.8pt,draw] (b) at (1,0) {};
\node[circle,fill=black,inner sep=0.8pt,draw] (c) at (2,0) {};
\node[circle,fill=black,inner sep=0.8pt,draw] (d) at (3,0) {};
\node[circle,fill=black,inner sep=0.8pt,draw] (e) at (4,0) {};
\node[circle,fill=black,inner sep=0.8pt,draw] (f) at (5,0) {};
\node[circle,fill=black,inner sep=0.8pt,draw] (g) at (5.5,0) {};
\node[circle,fill=black,inner sep=1.5pt,draw] (new-ab) at (0.7,0) {};

\node[circle,fill=black,inner sep=0.8pt,draw] (a') at (0,1) {};
\node[circle,fill=black,inner sep=1.5pt,draw] (b') at (1,1) {};
\node[circle,fill=black,inner sep=0.8pt,draw] (c') at (2,1) {};
\node[circle,fill=black,inner sep=0.8pt,draw] (d') at (3,1) {};
\node[circle,fill=black,inner sep=0.8pt,draw] (e') at (4,1) {};
\node[circle,fill=black,inner sep=0.8pt,draw] (f') at (5,1) {};
\node[circle,fill=black,inner sep=0.8pt,draw] (g') at (5.5,1) {};

\node () at (0.5,1.15) {\tiny $a$};
\node () at (0.35,-.2) {\tiny $a$};
\node () at (6,0.5) {\large $\cdots$};

\draw (a) edge[me=2] (a');
\draw (a) -- (g);
\draw (a') -- (g');

\draw (e) edge[dashed] (e');

\draw (b)--(b');
\draw (c)--(c');
\draw (d)--(d');
\draw (f)--(f');

\begin{scope}[shift={(+8,0)}]
\node[circle,fill=black,inner sep=0.8pt,draw] (a) at (0,0) {};
\node[circle,fill=black,inner sep=0.8pt,draw] (b) at (1,0) {};
\node[circle,fill=black,inner sep=0.8pt,draw] (c) at (2,0) {};
\node[circle,fill=black,inner sep=0.8pt,draw] (d) at (3,0) {};
\node[circle,fill=black,inner sep=0.8pt,draw] (e) at (4,0) {};
\node[circle,fill=black,inner sep=1.5pt,draw] (f) at (5,0) {};
\node[circle,fill=black,inner sep=0.8pt,draw] (g) at (5.5,0) {};
\node[circle,fill=black,inner sep=1.5pt,draw] (new-ab) at (4.7,1) {};

\node[circle,fill=black,inner sep=0.8pt,draw] (a') at (0,1) {};
\node[circle,fill=black,inner sep=0.8pt,draw] (b') at (1,1) {};
\node[circle,fill=black,inner sep=0.8pt,draw] (c') at (2,1) {};
\node[circle,fill=black,inner sep=0.8pt,draw] (d') at (3,1) {};
\node[circle,fill=black,inner sep=0.8pt,draw] (e') at (4,1) {};
\node[circle,fill=black,inner sep=0.8pt,draw] (f') at (5,1) {};
\node[circle,fill=black,inner sep=0.8pt,draw] (g') at (5.5,1) {};

\node () at (4.35,1.15) {\tiny $b$};
\node () at (4.5,-.2) {\tiny $b$};
\node () at (-1,0.5) {\LARGE $\sim$};
\node () at (6,0.5) {\large $\cdots$};

\draw (a) edge[me=2] (a');
\draw (a) -- (g);
\draw (a') -- (g');

\draw (e) edge[dashed] (e');

\draw (b)--(b');
\draw (c)--(c');
\draw (d)--(d');
\draw (f)--(f');
\end{scope}

\end{tikzpicture}
\caption{\small{A chip configuration on a cluster of the ladder.}}
\label{fig:ladder-pass-through-cluster}
\end{figure}
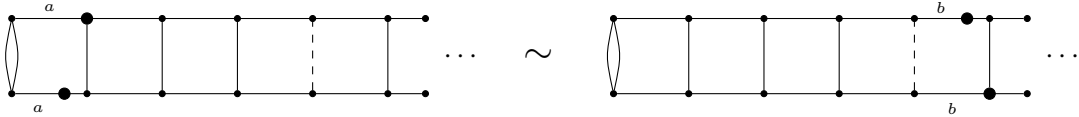

Let us call a \textit{cluster} each configuration of four consecutive cells. Let us also index the cells from left to right, so that the leftmost is numbered $1$, the one on its right -- $2$, and so on. The observation before Figure~\ref{fig:ladder-pass-through-cluster} shows that we can place two chips in each cluster spanning cells numbered from $4m+1$ to $4m+4$, where $0\leq m \leq \lfloor \frac{g-2}{4}\rfloor$, and be able to chip fire to a configuration with two chips in the cell numbered with $4\lfloor \frac{g-2}{4}\rfloor +1$. To finish the argument, we only need to examine four cases depending on the residue $g$ modulo~4.

Suppose $g=4k$. Then $\lfloor (g+3)/2 \rfloor=2k+1$  and $G$ has $4k-2$ cells. We place two chips in each of the $\lfloor (4k-2)/4\rfloor = k-1$ clusters. The remaining $2k+1 - 2(k-1) = 3$ chips we place as depicted in the leftmost graph in Figure~\ref{fig:last-chip-ladder-graph}, which portrays only the remaining cells that are not part of any cluster. We analogously deal with the cases $g=4k+t$ for $t\in \{1,2,3\}$. The placement of the last chips for these cases are shown in Figure~\ref{fig:last-chip-ladder-graph}. The dashed line indicates the rightmost edge of the last cluster.


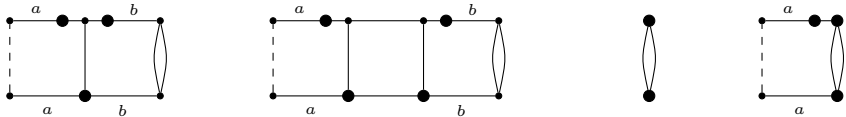
\begin{figure}[H]
\begin{tikzpicture}
\node[circle,fill=black,inner sep=0.8pt,draw] (a) at (0,0) {};
\node[circle,fill=black,inner sep=1.5pt,draw] (b) at (1,0) {};
\node[circle,fill=black,inner sep=0.8pt,draw] (c) at (2,0) {};
\node[circle,fill=black,inner sep=0.8pt,draw] (a1) at (0,1) {};
\node[circle,fill=black,inner sep=0.8pt,draw] (b1) at (1,1) {};
\node[circle,fill=black,inner sep=0.8pt,draw] (c1) at (2,1) {};

\node[circle,fill=black,inner sep=1.5pt,draw] (new-a1b1) at (0.7,1) {};
\node[circle,fill=black,inner sep=1.5pt,draw] (new-b1c1) at (1.3,1) {};

\node () at (0.35,1.15) {\tiny $a$};
\node () at (0.5,-.2) {\tiny $a$};
\node () at (1.65,1.15) {\tiny $b$};
\node () at (1.5,-.2) {\tiny $b$};

\draw (a) edge[dashed] (a1);
\draw (a) -- (c);
\draw (a1) -- (c1);
\draw (b) -- (b1);
\draw (c) edge[me=2] (c1);

\begin{scope}[shift={(+3.5,0)}]
\node[circle,fill=black,inner sep=0.8pt,draw] (a) at (0,0) {};
\node[circle,fill=black,inner sep=1.5pt,draw] (b) at (1,0) {};
\node[circle,fill=black,inner sep=1.5pt,draw] (c) at (2,0) {};
\node[circle,fill=black,inner sep=0.8pt,draw] (d) at (3,0) {};

\node[circle,fill=black,inner sep=0.8pt,draw] (a1) at (0,1) {};
\node[circle,fill=black,inner sep=0.8pt,draw] (b1) at (1,1) {};
\node[circle,fill=black,inner sep=0.8pt,draw] (c1) at (2,1) {};
\node[circle,fill=black,inner sep=0.8pt,draw] (d1) at (3,1) {};

\node[circle,fill=black,inner sep=1.5pt,draw] (new-a1b1) at (0.7,1) {};
\node[circle,fill=black,inner sep=1.5pt,draw] (new-c1d1) at (2.3,1) {};

\node () at (0.35,1.15) {\tiny $a$};
\node () at (0.5,-.2) {\tiny $a$};
\node () at (2.65,1.15) {\tiny $b$};
\node () at (2.5,-.2) {\tiny $b$};

\draw (a) edge[dashed] (a1);
\draw (a) -- (d);
\draw (a1) -- (d1);
\draw (b) -- (b1);
\draw (c) -- (c1);
\draw (d) edge[me=2] (d1);
\end{scope}

\begin{scope}[shift={(+8.5,0)}]
\node[circle,fill=black,inner sep=1.5pt,draw] (a) at (0,0) {};
\node[circle,fill=black,inner sep=1.5pt,draw] (b) at (0,1) {};
\draw (a) edge[me=2] (b);
\end{scope}

\begin{scope}[shift={(+10,0)}]
\node[circle,fill=black,inner sep=0.8pt,draw] (a) at (0,0) {};
\node[circle,fill=black,inner sep=1.5pt,draw] (b) at (1,0) {};
\node[circle,fill=black,inner sep=0.8pt,draw] (c) at (0,1) {};
\node[circle,fill=black,inner sep=1.5pt,draw] (d) at (1,1) {};
\node[circle,fill=black,inner sep=1.5pt,draw] (new-a1b1) at (0.7,1) {};

\node () at (0.35,1.15) {\tiny $a$};
\node () at (0.5,-.2) {\tiny $a$};

\draw (a) edge[dashed] (c);
\draw (d)--(c);
 \draw(a) --(b);
\draw (b) edge[me=2] (d);
\end{scope}

\end{tikzpicture}
\caption{\small{Remaining cells for ladder graph of genus $g=4k+t,~t\in \{0,1,2,3\}$.}}
\label{fig:last-chip-ladder-graph}
\end{figure}

To summarize, we have obtained the following.

\begin{theorem}\label{thm:ladder}
Given a graph $G$, which is homeomorphic to the genus g ladder graph, there exists a divisor $D$ on $G$ of degree $\lfloor (g+3)/2 \rfloor$ and rank at least $1$. 
\end{theorem}

\subsection{Inserting kites to graphs}
In this section we use our knowledge of graphs and their divisors for genera up to 5 to produce graphs of arbitrary high genus for which the gonality conjecture holds. We do so by inserting a kite graph on appropriately chosen vertices as shown below.

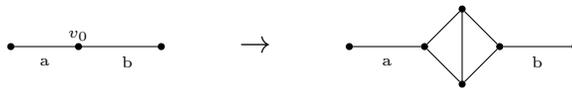
\begin{figure}[H]
\begin{tikzpicture}
\node[circle,fill=black,inner sep=0.8pt,draw] (a) at (-0.5,0) {};
\node[circle,fill=black,inner sep=0.8pt,draw] (d) at (1.5,0) {};
\node[circle,fill=black,inner sep=0.8pt,draw] (b) at (0.4,0) {};

\node () at (-0.05,-0.2) {\tiny a};
\node () at (.4, 0.15) {\tiny $v_0$};
\node () at (1.05,-0.2) {\tiny b};
\draw (a) -- (d);

\node () at (2.75,0) {\Large $\to$};

\begin{scope}[shift={(+5,0)}]
\node[circle,fill=black,inner sep=0.8pt,draw] (a) at (0,0) {};
\node[circle,fill=black,inner sep=0.8pt,draw] (b) at (0.5,0.5) {};
\node[circle,fill=black,inner sep=0.8pt,draw] (c) at (.5,-0.5) {};
\node[circle,fill=black,inner sep=0.8pt,draw] (d) at (1,0) {};
\node[circle,fill=black,inner sep=0.8pt,draw] (e) at (-1,0) {};
\node[circle,fill=black,inner sep=0.8pt,draw] (f) at (2,0) {};

\node () at (-0.5,-0.2) {\tiny a};
\node () at (1.5,-0.2) {\tiny b};

\draw (a) -- (b)--(c)--(d);
\draw (b) -- (d);
\draw (a) -- (c);
\draw (e)--(a);
\draw (d) -- (f);

\end{scope}
\end{tikzpicture}
\caption{\small{Kite insertion at $v_0$.}}
\label{fig:kite-insertion}
\end{figure}

\begin{proposition}
\label{prop:insert-kites}
Let $G$ be a genus $g$ graph and let $D\in \Div_{+}(G)$ be of degree at most~$\lfloor (g+3)/2 \rfloor$. Suppose $D_v$ and $G_v$, defined as in Proposition~\ref{prop:contactions-of-genus4-config}, are among the ones in Figure~\ref{fig:configurations-for-genus-5}, such that no two configurations from the second row share a common vertex. Then, for any bivalent $v_0 \in V(G)$ with $D_{v_0}$ belonging to the first row of the same figure, we can insert a kite at $v_0$ and the newly obtained genus $(g+2)$ graph (as well as any of its contractions) admits a divisor of degree at most~$\lfloor (g+5)/2 \rfloor$ and rank at least $1$. 
\end{proposition}
\begin{proof}
Inserting a kite at a vertex $v_0$ increases the genus by 2, so it is enough to place one additional chip on a vertex in the kite, and show that the newly obtained divisor is of rank at least $1$. We place the last chip as shown in the figure below, assuming $a\geq b$.
\begin{figure}[H]
\begin{tikzpicture}[scale=0.6]
\node[circle,fill=black,inner sep=0.8pt,draw] (41) at (0,0) {};
\node[circle,fill=black,inner sep=0.8pt,draw] (42) at (4,0) {};
\node[circle,fill=black,inner sep=0.8pt,draw] (43) at (2,1.72) {};

\node[circle,fill=black,inner sep=0.8pt,draw] (45) at (2,-1.72) {};

\node[circle,fill=black,inner sep=1.5pt,draw] (47) at (3.5,1.72) {};
\node[circle,fill=black,inner sep=1.5pt,draw] (48) at (.82,-.68) {};
\node[circle,fill=black,inner sep=1.5pt,draw] (new) at (3.5,-1.72) {};

\node () at (2.8,1.55) {\tiny$a$};
\node () at (2.80,-1.92) {\tiny$b$};
\node () at (-0.9,-0.5) {\tiny$\min(a-b,c)$};
\node () at (3.25,-1) {\tiny$c$};

\node () at (2,2.1) {\tiny$v$};
\node () at (2.05,-2) {\tiny$w$};

\draw (45) -- (new);
\draw (47) -- (43)--(42)--(41)--(48);
\draw (42)--(45);
\draw (41) -- (43);
\draw (48) edge[dashed] (45);

\end{tikzpicture}
\caption{\small{Placing the additional chip on the kite.}}
\end{figure}
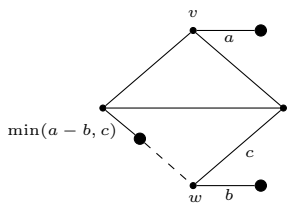
The most delicate part is calculating the lengths $a$ and $b$ since there might be some trivalent vertices through which the chips pass before reaching the kite's endpoints at $v$ or $w$. Since these distances should be independent of the size of the kite inserted, we compute them as follows. Starting with a $G$ and $D$ as above, we insert a kite at $v_0\in V(G)\backslash \supp (D)$ with all edge lengths longer than the sum of the lengths of the elongated edges of $G$. We then place two chips - one at each vertex $v$ and $w$ and run Dhar's burning algorithm for the divisor $D - (u)$. Here $D$ is viewed as a divisor on the new graph and $u$ is any of the trivalent vertices of the kite, different from $v$ and $w$ (see Figure~\ref{fig:kite-insertion-proof}). Record at which run of the algorithm a second chip reaches $v$ and $w$, respectively. These numbers are $a$ and $b$. Both numbers are well-defined by our choice of edge lengths of the kite.

\begin{figure}[H]
\begin{tikzpicture}
\node[circle,fill=black,inner sep=0.8pt,draw] (a) at (-0.5,0) {};
\node[circle,fill=black,inner sep=0.8pt,draw] (d) at (1.5,0) {};
\node[circle,fill=black,inner sep=0.8pt,draw] (b) at (0.4,0) {};

\node () at (.4, 0.15) {\tiny $v_0$};

\draw (a) -- (d);
\node () at (3,0) {\Large $\to$};

\begin{scope}[shift={(+5,0)}]
\node[circle,fill=black,inner sep=1.5pt,draw] (a) at (0,0) {};
\node[circle,fill=black,inner sep=0.8pt,draw] (b) at (0.5,0.5) {};
\node[circle,fill=black,inner sep=0.8pt,draw] (c) at (.5,-0.5) {};
\node[circle,fill=black,inner sep=1.5pt,draw] (d) at (1,0) {};
\node[circle,fill=black,inner sep=0.8pt,draw] (e) at (-1,0) {};
\node[circle,fill=black,inner sep=0.8pt,draw] (f) at (2,0) {};

\node () at (0,0.2) {\tiny $v$};
\node () at (1,0.2) {\tiny $w$};
\node () at (.5,.7) {\tiny $u$};

\draw (a) -- (b)--(c)--(d);
\draw (b) -- (d);
\draw (a) -- (c);
\draw (e)--(a);
\draw (d) -- (f);

\end{scope}
\end{tikzpicture}
\caption{\small{Chip configurations for the insertion of a kite.}}
\label{fig:kite-insertion-proof}
\end{figure}
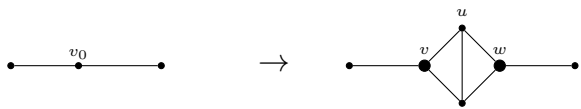
The newly-obtained divisor has rank at least $1$. Indeed, $G\backslash G_{v_0}$ remains unaffected by the kite insertion and the new divisor on $G'$, the graph obtained from $G_{v_0}$ by inserting kite at $v_0$, is also of rank at least $1$ as can be seen by running Dhar's burning algorithm. The edge contractions are dealt with as in the proof of Proposition~\ref{prop:configuration-for-genus-5}. 
\end{proof}

\begin{remark}
{\em Same holds for any $v_0$ with $D_{v_0}$ belonging to the second row as long as $v_0$ lies on one of the topological edges within the triangle. In this case, however, we are not always guaranteed existence of divisors with prescribed rank and degree for its contractions. The ideas in the proof of this proposition can be modified to allow kite insertions in other families of graphs. The authors did not pursue these ideas.}
\end{remark}

For instance, as a consequence of Proposition~\ref{prop:insert-kites} we can insert two kites to the bipartite graph~$K_{3,3}$ and obtain a graph of genus 8 which, for any edge lengths, admits a degree $5$ divisor of rank at least $1$.

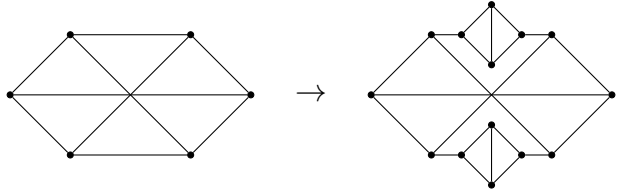
\begin{figure}[H]
\begin{tikzpicture}[scale=0.8]
\node[circle,fill=black,inner sep=0.8pt,draw] (a50) at (-2,0) {};
\node[circle,fill=black,inner sep=0.8pt,draw] (a51) at (-1,1) {};
\node[circle,fill=black,inner sep=0.8pt,draw] (a52) at (1,1) {};
\node[circle,fill=black,inner sep=0.8pt,draw] (a53) at (2,0) {};
\node[circle,fill=black,inner sep=0.8pt,draw] (a54) at (-1,-1) {};
\node[circle,fill=black,inner sep=0.8pt,draw] (a55) at (1,-1) {};

\node[circle,fill=black,inner sep=0.8pt,draw] (c) at (0,1.5) {};
\node[circle,fill=black,inner sep=0.8pt,draw] (d) at (0,0.5) {};
\node[circle,fill=black,inner sep=0.8pt,draw] (m) at (-0.5,1) {};
\node[circle,fill=black,inner sep=0.8pt,draw] (n) at (0.5,1) {};

\node[circle,fill=black,inner sep=0.8pt,draw] (e) at (0,-0.5) {};
\node[circle,fill=black,inner sep=0.8pt,draw] (f) at (0,-1.5) {};
\node[circle,fill=black,inner sep=0.8pt,draw] (k) at (-0.5,-1) {};
\node[circle,fill=black,inner sep=0.8pt,draw] (l) at (0.5,-1) {};

\draw (m)--(c)--(d)--(n);
\draw (m) -- (d);
\draw (c) -- (n);

\draw (a50) edge (a51); 
\draw(a51) -- (m);
\draw (n)-- (a52);
\draw (a50) -- (a53);
\draw (a52)--(a53); 
\draw (a53)--(a55); 

\draw (a55)--(l);
\draw (a54) -- (k);
\draw (k) -- (e) -- (f)--(l);
\draw (k) -- (f);
\draw (l) -- (e);

\draw (a54)--(a50);
\draw (a51)--(a55);
\draw (a52)--(a54);

\begin{scope}[shift={(-6,0)}]
\node[circle,fill=black,inner sep=0.8pt,draw] (a50) at (-2,0) {};
\node[circle,fill=black,inner sep=0.8pt,draw] (a51) at (-1,1) {};
\node[circle,fill=black,inner sep=0.8pt,draw] (a52) at (1,1) {};
\node[circle,fill=black,inner sep=0.8pt,draw] (a53) at (2,0) {};
\node[circle,fill=black,inner sep=0.8pt,draw] (a54) at (-1,-1) {};
\node[circle,fill=black,inner sep=0.8pt,draw] (a55) at (1,-1) {};

\node () at (+3,0) {\Large $\to$};

\draw (a50) edge (a51); 
\draw(a51) edge (a52);
\draw (a50) -- (a53);
\draw (a52)--(a53); 
\draw (a53)--(a55); 
\draw (a55)--(a54);
\draw (a54)--(a50);
\draw (a51)--(a55);
\draw (a52)--(a54);
\end{scope}
\end{tikzpicture}
\caption{\small{The bipartite $K_{3,3}$ with two kites inserted.}}
\label{fig:bipartite-with-two-kites}
\end{figure}

Kite graphs are not the only ones we can insert. With similar arguments we obtain the following.

\begin{proposition}
Let $g$ be an even integer. Let $G$ be a genus $g$ graph and let $D\in \Div_{+}(G)$ be of degree at most~$\lfloor (g+3)/2 \rfloor$. Suppose $D_v$ and $G_v$, defined as in Proposition~\ref{prop:contactions-of-genus4-config}, are among the ones in Figure~\ref{fig:configurations-for-genus-5}, such that no two configurations from the second row share a common vertex. Then, for any bivalent $v_0 \in V(G)$ with $D_{v_0}$ belonging to the first row of the same figure, we can insert a cycle at $v_0$ and the newly obtained genus $(g+1)$ graph (as well as any of its contractions) admits a divisor of degree at most~$\lfloor (g+4)/2 \rfloor$ and rank at least $1$. 
\end{proposition}
\begin{proof}
Note that $\lfloor (g+4)/2 \rfloor = \lfloor (g+3)/2 \rfloor + 1$ for even $g$. Thus we have one additional chip to place. We place it on one of the endpoints of the inserted cycle and obtain a divisor with prescribed degree and rank. The details are omitted.
\end{proof}

\bibliography{ExistenceConjecture}
\bibliographystyle{siam}
\end{document}